%% file: main.tex
\newif\ifpersonal
\personaltrue 
\documentclass[11pt, twoside]{amsart} 

\usepackage[english]{babel}
\usepackage[utf8]{inputenc} 
\usepackage[T1]{fontenc} 

\usepackage[textwidth=20mm, textsize=tiny]{todonotes}
\usepackage{amsmath,amsthm,amssymb,mathtools,amsfonts} 
\usepackage[foot]{amsaddr}
\usepackage{appendix}
\usepackage{microtype,lmodern} 
\usepackage{enumerate,comment,braket,xspace,tikz,tikz-cd} 
\usepackage{geometry}
\usepackage{xr}
\usepackage[hidelinks,linktoc=all]{hyperref}
\usepackage{appendix}
\usepackage{stmaryrd}
\usepackage{subfiles}
\usepackage{faktor}
\usepackage{framed}

\usepackage[capitalize]{cleveref}

\numberwithin{equation}{section}

\theoremstyle{plain}
\newtheorem{thmintro}{Theorem}

\newtheorem{corintro}[thmintro]{Corollary}
\newtheorem{thm}{Theorem}[section]
\newtheorem*{thm*}{Theorem}
\newtheorem{lem}[thm]{Lemma}
\newtheorem{prop}[thm]{Proposition}

\newtheorem{cor}[thm]{Corollary}

\theoremstyle{definition}
\newtheorem{defin}[thm]{Definition}
\newtheorem{warning}[thm]{Warning}
\newtheorem{notation}[thm]{Notation}

\newtheorem{rem}[thm]{Remark}

\newtheorem{fact}[thm]{Fact}
\newtheorem{recall}[thm]{Recall}
\newtheorem{cons}[thm]{Construction}

\input{commands.tex}

\title{Isotopy invariance and stratified $\E_2$-structure of the Ran Grassmannian}

\author[G.~Nocera]{Guglielmo Nocera} 

\author[M.~Porzio]{Morena Porzio} 

\date{\today}

\begin{document}

\begin{abstract}
Let $G$ be a complex reductive group. A folklore result asserts the existence of an $\mathbb{E}_2$-algebra structure on the Ran Grassmannian of $G$ over $\mathbb{A}^1_{\mathbb{C}}$, seen as a topological space with the complex-analytic topology. 
The aim of this paper is to prove this theorem, by establishing a homotopy invariance result: namely, an inclusion of open balls $D' \subset D$ in $\mathbb{C}$ induces a homotopy equivalence between the respective Beilinson--Drinfeld Grassmannians $\mathrm{Gr}_{G, {D'}^n} \xhookrightarrow{} \mathrm{Gr}_{G, D^n}$, for any positive integer $n$. 

We use a purely algebraic approach, showing that automorphisms of a complex smooth algebraic curve $X$ can be lifted to automorphisms of the associated Beilinson--Drinfeld Grassmannian. 
As a consequence, we obtain a stronger version of the usual homotopy invariance result: namely, the homotopies can be promoted to equivariant stratified isotopies, where ``equivariant'' refers to the action of the arc group $\mathrm{L}^+G$ and ``stratified'' refers to the stratification induced by the Schubert stratification of $\mathrm{Gr}_G$ and the incidence stratification of $\mathbb{C}^n$.
\end{abstract}

\maketitle
\vspace{-1cm}
\tableofcontents

\section{Introduction}

\noindent Let $G$ be a complex reductive group and let $\Gr_G$ be the affine Grassmannian associated to it.
This is the  complex ind-scheme which parametrizes $G$-torsors on the affine line $\Ac$ together with a trivialization away from the origin:
$$
\Gr_G: \textup{Aff}_\C^\op \to \sets, \quad \Gr_G(R) = \{\cF\in \Bun_G(\mathbb A^1_R), \alpha \textup{ trivialization of } \cF \textup{ on } \mathbb A^1_R \setminus \{0\}_R\}_{/\textup{isom.}}.
$$
Given a connected smooth  curve $X$ (locally of finite type) over $\C$  and a non-empty finite set $I$, the Beilinson--Drinfeld Grassmannian $\Gr_{G,X^I}$ 
is in turn the complex ind-scheme parametrizing
\begin{equation}\label{moduliinterpretationBD}
\Gr_{G,X^I}(R)  
=
\{x_I\in X^I(R), \cF\in \Bun_G(X_R), \alpha\textup{ trivialization of }\cF\textup{ on }X_R \setminus \Gamma_{x_I}\}_{/\textup{isom}},  
\end{equation}
where $\Gamma_{x_I}$ is the union of the graphs of points $x_I$ in $X_R$ (see \cref{graphs-of-points}).
By letting $I$ vary in the opposite category of non-empty finite sets with surjections between them, one can take the presheaf colimit of the $\Gr_{G,X^I}$'s, and obtain the so-called \emph{Ran Grassmannian} $\Gr_{G, \Ran(X)}$ {(\cref{defin-Ran-Grassmannian})}, naturally living over the presheaf $\Ran(X)$ parametrizing finite nonempty subsets of $X$.

These presheaves carry a stratification (\Cref{Schubert}, \Cref{BD-stratification} and \Cref{Ran-stratification}), 
induced by the stratification $\mathfrak{s}$ in \emph{Schubert cells} of the affine Grassmannian $\Gr_G$ and the \emph{incidence stratification} $\Inc_I$ of $X^I$.

\medskip Let $\Str\Top$ be the category of stratified topological spaces (\cref{defin-stratification}). 
We prove the following generalization of the usual analytification procedure.

\begin{thmintro}[\Cref{theoremanalyti}]\label{theorem-stratifications-intro}
The analytification functor from \cite[Th\'eor\`eme et D\'efinition 1.1]{SGA1-XII} can be enhanced and extended to  
\[
(-)^{\an}_{\PSh\Str}: \PSmall(\Str\Sch_{\C}) \to \Str\Top.
\]
\end{thmintro}

This allows us to consider the associated stratified analytifications of $(\Gr_G, \mathfrak{s})$, $( \Gr_{G, X^I}, \mathfrak{s}_I)$, $(\Ran(X), \Inc_\Ran)$, $( \Gr_{G, \Ran(X)}, \mathfrak{s}_\Ran)$ in $\Str\Top$ (see \Cref{analytification-of-GrBD} and \Cref{theoremanalyti}).
For simplicity, in this introduction we will refrain from expliciting the stratifications and simply write $(-)^\an$ for any stratified analytification. 

\subsection{Lifting of homotopies}
Let $X=\Ac$. 
Consider an open \textit{metric disk} $D$ in $(\Ac)^\an=\C$, {i.e.} an open ball $B(z,r)\subset \C$ centered in $z\in \C$ with radius $r\in \R_{>0}$.

We denote by $\Gr_{G,D^I}$ the fiber product $\Gr_{G, (\Ac)^I}^\an \times_{\C^I} D^I$ of stratified topological spaces.
In the same way, we define $\Gr_{G, \Ran(D)}$ to be the pullback of $\Gr_{G, \Ran(\Ac)}^\an$ to $\Ran(D)$, the open subset of $\Ran(\Ac)^\an$ given by nonempty finite collection of points in $D$.

\medskip The first aim of this paper is to prove the following folklore result. 
\begin{thmintro}[\Cref{homotopy-equivalence-D-D'}, \Cref{homotopy-invariance-Ran}]\label{main-theorem-introduction}
Let $D'\subset D\subset \C$ be two metric open disks.
The induced open embeddings
\[
i_I : \Gr_{G, {D'}^I} \xhookrightarrow{} \Gr_{G, {D}^I}, 
\quad \quad i_\Ran : \Gr_{G, \Ran({D'})} \xhookrightarrow{} \Gr_{G, \Ran({D})} 
\]
are stratified homotopy equivalences, and the homotopies involved can be taken to be isotopies. 
\end{thmintro}

A reason why \cref{main-theorem-introduction} is not straightforward (even if one wants to show that the mentioned maps are plain homotopy equivalences, regardless of stratifications) is that the map
\[
\Gr_{G,\C^I}\to \C^I 
\quad \quad \textup{(respectively }\Gr_{G, \Ran(\C)}\to \Ran(\C) \textup{)}
\]
is not a Serre fibration: {indeed, by the so-called \textit{factorization property} of the Beilinson-Drinfeld Grassmannian }(isomorphism \eqref{factorization-property}), the fiber at a point $(x_1,\dots,x_{|I|})\in \C^I$ splits as the product of $k$ copies of $\Gr_G^\an$, where $k$ is the number of \textit{distinct} points in the tuple.
Thus one cannot automatically lift a homotopy $D^I \times [0,1] \to D^I$ to one {at the level of the Beilinson-Drinfeld Grassmannians} $\Gr_{G, D^I} \times [0,1] \to \Gr_{G, D^I}$ (respectively Ran Grassmannians) via the usual homotopy lifting property. 

In the paper \cite{HY}, Hahn and Yuan outline a proof of 
$i_\Ran$ being an equivalence (see \cite[Proposition 3.17]{HY}). 
The idea is 
\begin{enumerate}
    \item\label{firststep-intro} to first prove that $i_\Ran^*$ is a cohomological equivalence,
    \item\label{secondstep-intro} to show that $\Gr_{G, \Ran(D)}$ is simply connected, and then 
    \item since $\Gr_G$ is simply connected as well (a consequence of \cref{Gr-OmegaG} below), conclude by applying the cohomological Whitehead theorem.
\end{enumerate}

Step \eqref{firststep-intro} follows from \cref{main-theorem-introduction}. At the present moment we are not aware of a different, purely cohomological argument, not relying on the fact that $i_{\Ran(D)}$ is a homotopy equivalence.

For what concerns \eqref{secondstep-intro}, as kindly explained to us by the authors of \cite{HY}, the idea of the proof is to use a stratified refinement of the homotopy lifting property for Serre fibrations. 
Namely, again by the aforementioned factorization property, one can easily show that $\Gr_{G, \Ran(D)} \to \Ran(D)$ becomes a Serre fibration when restricted to each locally closed stratum $\Ran_{\leq n}(D) \setminus \Ran_{\leq n-1}(D)$, where $\Ran_{\leq n}(D)$ parametrizes finite subsets of $D$ of cardinality at most $n$. 
Then, the strategy is to use this to prove that that one can lift homotopies $H: \Ran_{\leq n}(D) \times [0,1] \to \Ran_{\leq n}(D)$ with the following property:
\begin{itemize}
    \item[(*)] let $x\in \Ran_{\leq n}(D)$ and $m\leq n.$ For every $t$ such that $H(x,t)\in \Ran_{\leq m}(D)$, then $H(x,t')\in \Ran_{\leq m}(D)$ for any $t'>t$ as well.
\end{itemize}

To the best of our knowledge, no proof of this stratified homotopy lifting property is to be found in the literature.
This motivated us to use a different approach, more rooted in the algebro-geometric side of the story. Namely, we take advantage of the moduli space description \eqref{moduliinterpretationBD} of $\Gr_{G, X^I}$ to prove that one can lift an automorphism of $X^I$ induced by an automorphism $\phi$ of $X$ to $\Gr_{X^I}$ (\cref{lifting-for-Ran}). This is a weaker version of the homotopy lifting principles mentioned above, but it turns out to be sufficient. The formal properties of the analytification procedure established in \cref{theorem-stratifications-intro} allow to transfer this principle to the complex-analytic side, and by choosing $\phi$ to be an affine transformation of $D$ onto $D'$ we achieve the proof of \cref{main-theorem-introduction} (see \cref{homotopy-equivalence-D-D'}).

\subsection{\texorpdfstring{$\mathbb{E}_2$}{}-structure on the Ran Grassmannian}
Our next main result, which will be a consequence of \cref{main-theorem-introduction}, establishes the existence of a (nonunital) $\E_2$-structure on the Ran Grassmannian.

\begin{recall} 
For an elementary introduction to the notion of $\E_k$-algebra we recommend the introduction to Chapter 5 of \cite{HA}. 
Let us just mention that an $\E_2$-algebra structure on a topological space $Y$ is the datum of a multiplication on $Y$, defined up to homotopy, associative up to homotopy, and satisfying a certain degree of commutativity, generalizing the ``weak commutativity'' satisfied by spaces of the form $\Omega^2Z$ for $Z$ a pointed topological space. 
Indeed, if $Y$ is pointed and the given $\E_2$-algebra structure on $Y$ is grouplike (i.e. it admits an inverse operator for the multiplication, defined up to homotopy, where the marked point works as unit element), then $Y$ is homotopy equivalent to $\Omega^2Z$ for some pointed space $Z$. 
This is known as \textit{May's recognition principle} and can be found, in the language we are using for this paper, as \cite[Theorem 1.3.16]{DAG-VI} (see also the discussion at the beginning of \textit{loc. cit.}, §1.3). 
Our case has the peculiarity of living in the setting of \textit{stratified} topological spaces and stratified homotopy. 
Also, our algebra structure is \textit{non-unital}, i.e. it does not have a unit element: hence our situation somehow differs from the setting of May's recognition principle, but the rest of the intuition is intact.
\end{recall}

\begin{rem}\label{Gr-OmegaG}
The topological space  $\Gr_G^\an$ is homotopy equivalent to the double loop space $\Omega^2\tB (G^\an)$, see e.g. \cite[Theorem 2.1]{Nadler-Matsuki}, \cite[Theorem 8.6.2, 8.6.3]{Pressley-Segal}.
Therefore, it inherits an $\E_2$-structure in topological spaces up to homotopy.
\end{rem}

As a corollary of \cref{main-theorem-introduction}, we get that a (nonunital) $\mathbb{E}_2$-structure exists also for $\Gr_{G, \Ran(D)}$.

\begin{corintro}[{\Cref{corollary-algebra-structure}}]\label{GrGRanisE2}
Let $W$ be the family of arrows of $\Str\Top$ given by stratified homotopy equivalences. 
For any metric disk $D$, $\Gr_{G,\Ran(D)}$ carries a non-unital $\mathbb{E}_2$-algebra structure in $\Str\Top[W^{-1}]$, independent of $D$.
\end{corintro}

A first instance of \cref{GrGRanisE2} is in \cite[Remark 9.4.20]{Gaitsgory-Lurie} where it is stated that the $\mathbb{E}_2$-coalgebra structure on $C^*(\Gr_G;\mathbb{Z}_\ell)$ (in complexes up to homotopy) can be recovered from the sheaf $$
\mathcal{A}:\textup{Open}(\Ran(\C))^\op\to 
\textup{Ch}^*(\textup{Mod}_{\mathbb{Z}_\ell}), \quad U\mapsto C^*(\Gr_{G,\Ran(\C)}\times_{\Ran(\C)}U;\mathbb{Z}_\ell).
$$
This phenomenon is also spelt out in \cite[Theorem 3.10]{HY}. 
The existence of these $\E_2$-structures locates in the bigger picture of the Geometric Langlands Program, where the Beilinson--Drinfeld and Ran Grassmannians are crucial objects.
In particular, they are often used to establish \textit{factorization algebra} structures (which are ``algebraic avatars'' of $\E_2$-structures), one of the key ingredients of the proof of the Geometric Langlands Conjecture, see \cite{Campbell-Raskin} and \cite[§1]{GLC2}.

\begin{rem}
\Cref{GrGRanisE2} can be viewed as the first step towards an \emph{unstable} version of \cite[Theorem 3.10]{HY}: namely that $\Gr_G^\an$ admits a non-unital $\mathbb{E}_2$-algebra structure in $\Str\Top[W^{-1}]$. 
Indeed, \cite[Proposition 3.17]{HY} says that the map
\begin{equation*}\label{eq-stable-equivalence}
\Sigma_+^\infty(\Gr^\an_G)\to \Sigma_+^\infty(\Gr_{G,\Ran(\Ac)}^\an)
\end{equation*} 
associated to any point $x\in \Ac(\C)$ is an equivalence of spectra.

With the present work at hand, one is left to inspect the map $\Gr_G^\an \xhookrightarrow{} \Gr_{G,\Ran(D)}$ and prove that the $\E_2$-structure can be transferred to the left-hand-side, in analogy to the stable result.
Note that this would {also} refine the usual $\E_2$-structure on $\Gr_G^\an$ 
enhancing it from topological spaces up to homotopy to \emph{stratified} topological spaces up to stratified homotopy. 
\end{rem}

\subsection{An \texorpdfstring{$\LpG$}{}-equivariant version of \texorpdfstring{\cref{main-theorem-introduction}}{}}

Let now $\LpG_{X^I}$ be the Beilinson--Drinfeld version of the arc group (\Cref{defin-BD-arc-group}). 
This is a relative group scheme over $X^I$ acting on $\Gr_{G, X^I}$. 
It inherits the incidence stratification $\Inc_I$ from $X^I$.
One can consider its stratified analytification $\LpG_{X^I}^\an$, which is still a group scheme acting on $\Gr_{G, X^I}^\an$ via a stratified action (\Cref{analytification-of-GrBD}).
Denote by $\LpG_{D^I}$ the fiber product $\LpG_{X^I}^\an \times_{(X^\an)^I} D^I$.
Given two open metric disks $D'\subset D\subset \C$, we again get that the induced open embedding $i^+:\LpG_{{D'}^I}\hookrightarrow \LpG_{{D}^I}$ is a stratified homotopy equivalence, and the homotopies involved can be taken to be isotopies as well (\Cref{homotopy-equivalence-D-D'-groups}).
This allows us to formulate and prove an equivariance property for the homotopies in \cref{main-theorem-introduction}, as follows.

\begin{thmintro}[\Cref{isotopies-are-equivariant}, \cref{equivariance-theorem-Ran}]\label{thm-intro-equivariant}
Given two metric open disks $D'\subset D\subset \C$, all the mentioned isotopies are compatible with the action of $\LpG_{D^I}$ on $\Gr_{G, D^I}$.
More precisely, there are stratified isotopies $\Psi^{\textup{equiv}}_{[0,1]}$ and $\Psi_{[0,1]}$ fitting in
\[
\begin{tikzcd}[column sep=3cm]
    {[0,1]\times \left( \LpG_{(\Ac)^I}^\an\times_{(\Ac)^I}\Gr_{G, (\Ac)^I}^\an \right)}
    \arrow[r, "{\Psi^{\textup{equiv}}_{[0,1]}}"]
    \arrow[d, "\id\times \act_I"]
    &
    \LpG_{(\Ac)^I}^\an\times_{(\Ac)^I}\Gr_{G,(\Ac)^I}^\an
    \arrow[d, "\act_I"]
    \\
    {[0,1]\times \Gr_{G,(\Ac)^I}^\an}
    \arrow[r, "{\Psi_{[0,1]}}"]
    &
    \Gr_{G, (\Ac)^I}^\an,
\end{tikzcd}    
\]
which provide stratified isotopies for the diagram
$$
\begin{tikzcd}
    \LpG_{{D'}^I} \times_{{D'}^I} \Gr_{G, {D'}^I} 
    \arrow[r, hook, "{i^+\times i}"]
    \arrow[d, "\act_I"]
    &
    \LpG_{D^I} \times_{D^I} \Gr_{G, D^I}
    \arrow[d, "\act_I"]
    \\ 
    \Gr_{G, {D'}^I}
    \arrow[r, hook, "i"]
    &
    \Gr_{G, D^I}.
\end{tikzcd}
$$
\end{thmintro}

\subsection{Outline of the paper}
In \Cref{Section-stratified-analytification} we formalize the fact that the usual analytification functor $(-)^\an:\Sch^\lft_\C\to \Top$ can be enhanced to a functor $(-)^\an_{\PSh\Str}$ between the category of \emph{small stratified} presheaves and stratified topological spaces  (see \Cref{theoremanalyti}).

In \Cref{section-Ran-Grassmannian}, we determine several properties of the Ran Grassmannian, first from an algebraic geometry perspective and then from a (complex-analytic) topology one. 
Some of those are not formal consequences of the analogous properties of the Beilinson-Drinfeld Grassmannian, since we look at $\Gr_{\Ran(X)}^\an$ as a stratified topological space and not as a presheaf of topological spaces (i.e. we ``realize'' it in $\Str\Top$). 
In particular, the existence of a stratified continuous action of $\LpG_{\Ran(X)}^\an$ on $\Gr_{G, \Ran(X)}^\an$ over $\Ran(X)^\an$ is non-trivial (see \cref{extension-of-analytification}).

\Cref{Section-invariance} is devoted to the proofs of the main results of the paper.
We first observe that for any connected smooth complex curve $X$ there is a morphism of presheaves 
$$\uAut_\C(X)\to \uAut_\C(\Gr_{G,X^I})$$ 
lifting an automorphism of $X$ to a (stratified) automorphism of the Beilinson--Drinfeld Grassmannian $\Gr_{G,X^I}$ (see \Cref{remark-lift-automorphism} and \Cref{lemma-stratified-automorphism}). 
In particular, if $X=\Ac$, one can lift affine transformations $z\mapsto \alpha z+\beta$. 
One can apply this lifting principle to isotopically transform the restrictions $\Gr_{G, D^I}$ from any open metric disk $D$ to another. This lifting result is also true at the Ran level, i.e. there is a lifting morphism $$\uAut_\C(X)\to \uAut_\C(\Gr_{G,\Ran(X)}).$$ 
These arguments achieve the proof of \Cref{main-theorem-introduction} (see \cref{homotopy-equivalence-D-D'} and \cref{homotopy-invariance-Ran}). 
\Cref{thm-intro-equivariant} is proven similarly: indeed, the fact that $\LpG_{D'^I}\xhookrightarrow{}\LpG_{D^I}$ and $\LpG_{\Ran(D')}\xhookrightarrow{} \LpG_{\Ran(D)}$ are stratified homotopy equivalences follows from a similar lifting principle, and the compatibility with the action follows from the constructions.

Finally, we deduce \cref{GrGRanisE2} from \cref{main-theorem-introduction} by applying Lurie's theorem \cite[Theorem 5.4.5.15]{HA} saying that non-unital $\E_2$-algebras with values in a symmetric monoidal category $\cc^\otimes$ are equivalent to \textit{locally constant} non-unital $\Disk(\R^2)^\otimes$-algebras with values in $\cc^\otimes$. Here $\Disk(\R^2)^\otimes$ is the operad of topological disks in the real plane, and the local constancy property corresponds to \cref{main-theorem-introduction}.

Appendix \ref{recollection} is devoted to some recollections about the affine Grassmannian and the Beilinson-Drinfeld Grassmannian, as well as some detailed proofs of small and useful folklore facts needed in the paper.

\subsection*{Notation}
In this paper $G$ will always denote a complex reductive group, and $X$ will be a smooth (not necessarily projective) connected complex curve.

For a scheme $Y$, $\Bun_G(Y)$ is the groupoid of \'etale
$G$-torsors over $Y$. 
Let $\cT_G$ be the trivial $G$-torsor over $\Spec \C$: for any complex scheme $S$ we denote by  $\cT_{G,S}$ its base change along the structural map $S\to \Spec \C$. When it does not cause confusion, we will just write $\cT_G$ or $\cT$.

\subsection*{Acknowledgements}
We thank Jeremy Hahn and Allen Yuan very much for providing multiple clarifications about their paper \cite{HY}, and for encouraging us to provide a proof of \Cref{main-theorem-introduction}. 
We also thank Dustin Clausen, Marius Kj\ae rsgaard, Yonatan Harpaz, Sam Raskin and Marco Volpe for fruitful discussions.

During the process of writing this paper, the first author was supported by the ERC Grant ``Foundations of Motivic Real K-theory'' held by Yonatan Harpaz, and later by the grant ``Simons Collaboration on Perfection in Algebra, Geometry and Topology'' co-held by Dustin Clausen.

\section{Stratifications and the analytification functor}\label{Section-stratified-analytification}
Let $G$ be a complex reductive group. The main objects of this paper are the \emph{affine Grassmaniann} $\Gr_G$, the \emph{Beilinson-Drinfeld Grassmannians} $\Gr_{G, X^I}$, and the \emph{Ran Grassmannian} $\Gr_{G, \Ran(X)}$, considered with their respective stratifications. 
We want to see these objects both from the algebro-geometric and the complex-analytic point of view.
In order to do so, we first need to formalize how to analytify \emph{stratified} schemes, and \emph{stratified small presheaves} in order to obtain stratified topological spaces.

\subsection{Stratified small presheaves}\label{Subsection-stratifications}
	
Let $Y$ be a topological space.
Among the slightly different definitions of \emph{stratification} (see \cite{strsp} for a full comparison between them) we will stick to the \emph{poset-stratified} one due to its good categorical properties (see the discussion on page 2 of \cite{strsp}).

\begin{defin}\label{defin-stratification}
A \emph{poset-stratified} space is a triple $(Y,P,s:Y\to \Alex(P))$ where
\begin{enumerate}
    \item $Y$ is a topological space, and $P$ is a poset,
    \item $\Alex:\Pos\to \Top$ is the functor associating to a poset $P$ the topological space of elements of $P$ endowed with the Alexandroff topology, and 
    \item $s$ is a continuous surjective map. 
\end{enumerate}  
We will often use $(Y,s)$
to denote the triple $(Y,P,s:Y\to \Alex(P))$ and we will refer to poset-stratified spaces simply as \emph{stratified spaces}.
		
A map of stratified spaces is a pair $(f,r):(Y,s)\to (W,s')$ where $f: Y\to W$ is a continuous map and $r:P\to Q$ is an order-preserving function such that
$$
\begin{tikzcd}
    Y
    \arrow[r, "{f}"]
    \arrow[d, "{s}"']
    &
    W
    \arrow[d, "{s'}"]
    \\
    \Alex(P)
    \arrow[r, "{r}"]
    &
    \Alex(Q).
\end{tikzcd}
$$   
commutes.
We denote by $\Str\Top$ the category of stratified topological spaces.
\end{defin}

\begin{rem}\label{StrTopiscocomplete}
The category $\Str\Top$ is complete and cocomplete.
Both properties are proven in \cite[Proposition 6.1.4.1]{Nand-Lal-thesis} for the category of \textit{stratified compactly generated spaces} but the proof for $\Str\Top$ is the same.
Moreover, small colimits are realized as follows:
\begin{multline*}
    \colim_{\alpha \in A} \left(Y_{\alpha},P_\alpha, s_\alpha: Y_{\alpha}\to \Alex(P_{\alpha})\right) =\\
    \left( \colim_{\alpha\in A} Y_{\alpha}, \colim_{\alpha\in A} P_{\alpha}, \colim_{\alpha\in A} Y_{\alpha} \xrightarrow{\underset{\alpha\in A}{\colim}\, s_\alpha} \colim_{\alpha\in A} \Alex(P_{\alpha}) \to \Alex( \colim_{\alpha\in A} P_{\alpha}  ) \right).    
\end{multline*}
In particular, the underlying topological space (resp. the poset) of the colimit in $\Str\Top$ is the colimit in $\Top$ (resp. in $\Pos$) of the underlying topological spaces (resp. posets).

For limits, the situation is slightly different: the underlying poset still coincides with the $\lim_{\alpha\in A} P_\alpha$ but in general the underlying topological space will have a finer topology than $\lim_{\alpha \in A}Y_{\alpha}$ in $\Top$.
Nevertheless for \emph{finite} limits $F\to \Str\Top$, we still get that
\begin{multline*}
    \lim_{\alpha \in F} (Y_{\alpha},P_\alpha, s_\alpha: Y_{\alpha}\to \Alex(P_{\alpha}))=\\
    \left( \lim_{\alpha\in F} Y_{\alpha}, \lim_{\alpha\in F} P_{\alpha}, \lim_{\alpha\in F} Y_{\alpha} \xrightarrow{ \underset{\alpha\in F}{\lim}\, s_\alpha}    \lim_{\alpha\in F} \Alex(P_{\alpha}) 
    \xleftarrow{\sim}  
    \Alex( \lim_{\alpha\in F} P_{\alpha}) \right).
\end{multline*}
For a proof, one first reduces to the case of a finite product and then observes that the Alexandroff topology on a product coincides with the box topology, which in turn is the same as the product topology if the product is finite.

Note also that if the diagram of posets is constant $P_{\alpha}\equiv P$ (and without any finiteness assumption), then we still get 
\[\lim_{\alpha \in A} (Y_{\alpha},P, s_\alpha: Y_{\alpha}\to \Alex(P)) =
\left( \lim_{\alpha\in A} Y_{\alpha}, P, s: \lim_{\alpha\in A} Y_{\alpha} \to \Alex( P) \right). 
\]

In synthesis, if we denote by $\Fgt_\str:\Str\Top\to \Top$ the functor which forgets the stratification, it preserves all colimits, finite limits, and limits of diagrams where the poset is constant.
\end{rem}
	
\begin{defin}
    Let $R$ be a $\C$-algebra, locally of finite type.
    A \emph{stratified scheme} (locally of finite type over $R$) is a triple $(Y,P, s:Y^{\Zar}\to \Alex(P))$, where $Y$ is a scheme (locally of finite type over $R$) and $(Y^\Zar,P,s)$ is a stratified topological space.  
    A map of stratified schemes is a pair $(f,r)$ where $f$ is a map of $R$-schemes and $(f^{\Zar}, r)$ is a map of stratified topological spaces.
    
    We denote by $\Str\Sch^{\lft}_{R}$ the category of stratified schemes locally of finite type over $R$.
\end{defin}

\begin{rem}\label{StrSch admits finite limits}
In an analogous way to the case of $\Str\Top$, one can verify that the category $\Str\Sch^\lft_R$ admits finite limits and they have the form
\[
\lim_{\alpha \in F} \left(Y_\alpha, P_\alpha, s_\alpha: Y_{\alpha}^\Zar\to \Alex(P_\alpha) \right)   
=
\left( 
\lim_{\alpha \in F} Y_\alpha , \lim_{\alpha \in F} P_\alpha, s : \left( \lim_{\alpha \in F} Y_{\alpha} 
\right)^\Zar 
\to 
\lim_{\alpha \in F} Y_{\alpha}^\Zar \to \Alex( \lim_{\alpha\in F} P_\alpha) \right).
\]
\end{rem}
	
\begin{defin}		
Let $\cc$ be a locally small category.
A \emph{small presheaf} on $\mathcal{C}$ is a small colimit over a diagram of the form 
$\gamma : A \to \mathcal{C} \xhookrightarrow{} \PSh(\cc)$ where  $\cc\xhookrightarrow{} \PSh(\cc)$ 
is the Yoneda functor $\yo$. 
We denote by $\PSmall(\cc)$ the full subcategory of $\PSh(\mathcal{C})$ of small presheaves.  
\end{defin}  
		
\begin{rem}\label{cocompletion}
By definition, $\PSmall(\cc)$ is the small free cocompletion of $\cc$ and $\cc$ embeds in it via the Yoneda functor $\yo: \cc\xhookrightarrow{} \PSmall(\cc)$, see \cite[Theorem 2.11]{Lindner-enriched}.
See \cite[Definition 4.1]{Coco} for the definition of free cocompletion of a locally small category. 
\end{rem}

Recall that the categories $\Str\Sch^\lft_\C, \Str\Sch$ are locally small.
        
\begin{defin}\label{defin-stratified-presheaves}
A \emph{stratified small presheaf} over $R$ is an object of $\PSmall(\Str\Sch_R)$. 
A \emph{stratified small presheaf locally of finite type over $R$} is an object of $\PSmall(\Str\Sch^{\lft}_R)$.
\end{defin}

\begin{rem}\label{embedding-strschlft-strsch}
Denote by $\lambda$ the left Kan extension
\[
\begin{tikzcd}
    \Sch^\lft_R
    \arrow[r, hook]
    \arrow[d, "{\yo}", hook]
    &
    \Sch_R 
    \arrow[r, hook, "{\yo}"]
    &
    \PSmall(\Sch_R),
    \\
    \PSmall(\Sch^\lft_R)
    \arrow[rru, dashed, "{\lambda}"']
\end{tikzcd}
\]
which preserves colimits and is left adjoint to the restriction functor $\PSmall(\Sch_R)\to \PSmall(\Sch_R^\lft)$.
Analogously, denote by $\lambda_\Str$ the left Kan extension 
\[
\begin{tikzcd}
    \Str\Sch^\lft_R
    \arrow[r, hook, "{i_{\Str,\lft}}"] 
    \arrow[d, hook, "{\yo}"']
    &
    \Str\Sch_R
    \arrow[r, hook, "{\yo}"]
    &
    \PSmall(\Str\Sch_R)
    \\
    \PSmall(\Str\Sch_R^\lft).
    \arrow[rru, dashed, "{\lambda_{\Str}}"']
\end{tikzcd}
\]
It preserves colimits and is the left adjoint to the restriction functor $$\PSmall(\Str\Sch_R)\to \PSmall(\Str\Sch_R^\lft).$$
\end{rem}

\subsection{Stratified analytification}\label{Subsection-analytification}
	
Let us recall the notion of the \emph{analytification functor} from SGA1-XII.
For this, let $\mathfrak{L}_\C$ be the category of locally $\C$-ringed spaces and let $\mathfrak{A}\textup{n}_\C$ the full subcategory of complex analytic spaces inside $\mathfrak{L}_\C$.
	
\begin{thm}[{\cite[Thm. XII.1.1]{SGA1-XII} and \cite[\S XII.1.2]{SGA1-XII}}]\label{Raynaud}
Let $Y$ be a scheme locally of finite type over $\C$.
Then the functor 
$$
\Hom_{\mathfrak{L}_\C}(-,Y): \mathfrak{A}\textup{n}_\C^\op \to \sets
$$
is representable by a complex analytic space $\an(Y)$: namely there exists a map of locally $\C$-ringed spaces $\varphi_Y:\an(Y)\to Y$ such that 
$$
\Hom_{\mathfrak{A}\textup{n}_\C}(T, \an(Y))\xrightarrow{\sim} 
\Hom_{\mathfrak{L}_\C}(T,Y),\quad f\mapsto \varphi_Y\circ f
$$
is a natural bijection (controvariant in $T$ and covariant in $Y$).
Moreover, $\an(Y)$ coincides, as sets, with $Y(\C)$. 
Denote by $Y^\an$ the underlying topological space of $\an(Y)$ (namely, forget the sheaf). This then defines an \emph{analytification functor} 
$$
(-)^\an: \Sch_\C^{\lft} \to \Top, \quad Y\mapsto Y^{\an}
$$
which preserves finite limits.
\end{thm}

The notation $Y^\an$ differs from the one used in SGA1 \cite{SGA1-XII}, where $Y^{\an}$ denotes the complex analytic space and not its underlying topological space.
	
We now want to enhance and extend this functor to the category of small stratified presheaves $\PSmall(\Str\Sch_\C)$.

\begin{thm}[Stratified Analytifications]\label{theoremanalyti}
The analytification functor of \cref{Raynaud} can be enhanced and extended to  
\begin{gather*}
    (-)^{\an}_{\Str, \lft}: \Str\Sch^{\lft}_{\C} \to \Str\Top, \quad
    (-)^{\an}_{\Str}: \Str\Sch_{\C} \to \Str\Top, \\
    (-)^{\an}_{\PSh\Str, \lft}: \PSmall(\Str\Sch^\lft_{\C}) \to \Str\Top, \quad 
    (-)^{\an}_{\PSh\Str}: \PSmall(\Str\Sch_{\C}) \to \Str\Top
\end{gather*}
where the first functor preserves finite limits, the second one preserves small limits, the last two preserve small colimits.
They fit in the following commutative diagram:
\begin{equation}\label{big-analytification-diagram}
\begin{tikzcd}[column sep=2cm]
    \Sch^\lft_\C
    \arrow[rrr,"\hspace{0.8cm}(-)^\an"] 
    & & &
    \Top
    \\
    \Str\Sch_\C^\lft
    \arrow[dr, hook, "\yo"']
    \arrow[dd, hook, "{i_{\Str, \lft}}"']
    \arrow[rrr, "\hspace{1.2cm}(-)^{\an}_{\Str, \lft}"]
    \arrow[u, "{\Fgt_{\str}}"]
    & & &
    \Str\Top 
    \arrow[dd, equal]
    \arrow[u, "{\Fgt_{\str}}"']
    \\
    & \PSmall(\Str\Sch_\C^\lft) 
    \arrow[dd, near start, "\lambda_\Str"']
    \arrow[rru, dashed, "{(-)^{\an}_{\PSh\Str, \lft} = \textup{LKE}_{\yo} (-)^\an_{\Str,\lft}}"'] & &
    \\
    \Str\Sch_\C
    \arrow[dr, "{\yo}"', hook]
    \arrow[rrr, "{\hspace{4cm} (-)^\an_{\Str} = \textup{RKE}_{i_{\Str,\lft}} (-)^\an_{\Str,\lft}}", dashed]
    & & &
    \Str\Top.
    \\
    & \PSmall(\Str\Sch_\C)
    \arrow[rru, "{(-)^\an_{\PSh\Str} = \textup{LKE}_{\yo} (-)^\an_{\Str}}"', dashed] 
    & &
\end{tikzcd}
\end{equation}
\end{thm}

\begin{proof}
The only non-trivial parts are: the construction of $(-)^\an_{\Str,\lft}$ and checking that the square involving $\lambda_\Str, (-)^\an_{\PSh\Str,\lft}$ and $(-)^\an_{\PSh\Str}$ commutes. 
The rest of the statement follows by properties of left and right Kan extensions along fully faithful functors.

So, let $(Y,P,s:Y^{\Zar}\to \Alex(P))$ be an element of $\Str\Sch^{\lft}_\C$.
The morphism $\varphi_Y: \an(Y)\to Y$ induces a map of topological spaces $\varphi_Y^{\textup{top}} : Y^\an\to Y^\Zar$.
Define $s^{\an}$ to be the composite
$$
s^\an = s\circ \varphi_Y^{\textup{top}}: Y^{\an}\to Y^{\Zar} \to \Alex(P).
$$
Let $(f,r):(Y,s)\to (W,s')$ be a stratified map. 
Consider the map $\an(f): \an(Y)\to \an(W)$: by definition the map $\an(f)$ fits in the commutative diagram of ringed spaces
$$
\begin{tikzcd}
    \an(Y)
    \arrow[r, "{\an(f)}"]
    \arrow[d, "{\varphi_Y}"']
    &
    \an(W)
    \arrow[d, "{\varphi_W}"]
    \\
    Y
    \arrow[r, "f"]
    &
    W.
\end{tikzcd}
$$
By forgetting the sheaves, we have the commutative diagram
$$
\begin{tikzcd}
    Y^{\an}
    \arrow[r, "{f^{\an}}"]
    \arrow[d, "{\varphi_Y^{\textup{top}}}"']
    &
    \an(W)
    \arrow[d, "{\varphi_W^{\textup{top}}}"]
    \\
    Y^{\Zar}
    \arrow[r, "f^{\Zar}"]
    \arrow[d, "{s}"']
    &
    W^{\Zar}
    \arrow[d, "{s'}"]
    \\
    \Alex(P)
    \arrow[r, "r"]
    &
    \Alex(Q).
\end{tikzcd}
$$
Therefore $(f^{\an},r)$ is a map of stratified spaces $(Y^{\an},s^\an)\to (W^{\an}, {s'}^\an)$.
This defines a functor
$$
(-)^{\an}_{\Str,\lft}: \Str\Sch^{\lft}_{\C}\to \Str\Top, \quad \quad (Y,s)\mapsto (Y^\an, s^\an), \ \textup{ and }\ (f,r) \mapsto (f^{\an},r),
$$
which enhances $(-)^\an: \Sch^{\lft}_\C\to \Top$, in the sense that the top square in \eqref{big-analytification-diagram} commutes. This functor still preserves finite limits: indeed, given a finite diagram $F\to \Str\Sch^{\lft}_\C$, by \cref{StrSch admits finite limits}, the limit $\lim_{\alpha \in F} (Y_\alpha, P_\alpha, s_\alpha: Y_{\alpha}^\Zar\to \Alex(P_\alpha) ) $ is 
\begin{equation*}
    \left(\lim_{\alpha \in F} Y_\alpha , \lim_{\alpha \in F} P_\alpha, s : \left( \lim_{\alpha \in F} Y_{\alpha} \right)^\Zar \to \lim_{\alpha \in F} Y_{\alpha}^\Zar \to \Alex( \lim_{\alpha\in F} P_\alpha) \right).
\end{equation*}
By the definition of $(-)^\an_{\Str,\lft}$ and by the fact that the original $(-)^\an$ preserves finite limits, this in turns is equal to
\[
\left(  \lim_{\alpha \in F} Y_\alpha^\an , \lim_{\alpha \in F} P_\alpha, s^\an : 
\lim_{\alpha \in F} Y_{\alpha}^\an
\to \left( \lim_{\alpha \in F} Y_{\alpha} \right)^\Zar 
\xrightarrow{} \lim_{\alpha \in F} Y_{\alpha}^\Zar 
\to \Alex( \lim_{\alpha\in F} P_\alpha) 
\right).
\]
By the universal property of limits, the map $\lim_{\alpha \in F} Y_{\alpha}^\an
\to \left( \lim_{\alpha \in F} Y_{\alpha} \right)^\Zar 
\xrightarrow{} \lim_{\alpha \in F} Y_{\alpha}^\Zar $ coincides with the limit map $\lim_{\alpha \in F} Y_{\alpha}^\an
\to \lim_{\alpha \in F} Y_{\alpha}^\Zar $, and we conclude.

For what concerns the commutativity of the square with the diagonal dashed arrows in \eqref{big-analytification-diagram}, note that any element of $\PSmall(\Str\Sch_\C^\lft)$ is a colimit $\colim_i (Y_i,s_i)$ of objects in $\Str\Sch_\C^\lft$.
We thus have the assignments
\[
\begin{tikzcd}
\colim_i (Y_i,s_i)
\arrow[d, mapsto, "\lambda_\Str"]
\arrow[rr, mapsto, "(-)^\an_{\PSh\Str,\lft}"]
& &
\colim_i (Y_i,s_i)^\an_{\Str,\lft}
\arrow[d, equal]
\\
\colim_i  \yo \circ i_{\Str,\lft} (Y_i,s_i)
\arrow[rr, mapsto, "(-)^\an_{\PSh\Str}"]
& &
\colim_i  \left(i_{\Str,\lft} (Y_i,s_i) \right)^\an_\Str,
\end{tikzcd}
\]
hence the claim.
\end{proof}

Note that at priori $(-)^\an_{\PSh\Str,\lft}$ does not preserve finite limits. 
However, let $\Fgt_{\str, \textup{top}}: \Str\Top \to \sets$ be the functor forgetting stratification and topology. 

\begin{lem}\label{underlying-set-an}
    The composite $\Fgt_{\str, \textup{top}}\circ (-)^\an_{\PSh\Str, \lft}$, that is the functor associating to a stratified presheaf its set of $\C$-points, preserves finite limits.
\end{lem}

\begin{proof}
We want to apply \cite[Lemma B.55]{E3}.
In order to do this we note that
\begin{enumerate}
\item since $\Fgt_{\str, \textup{top}}$ preserves colimits, the composite $\Fgt_{\str, \textup{top}} \circ (-)^\an_{\PSh\Str, \lft}$ coincides with the left Kan extension 
\[
\begin{tikzcd}[column sep=2cm]
\Str\Sch_\C^\lft 
\arrow[r, "(-)^\an_{\Str, \lft}"]
\arrow[d, "{\yo}"]
&
\Str\Top \arrow[r, "{\Fgt_{\str, \textup{top}}}"]
&
\sets,
\\
\PSmall(\Str\Sch_\C^\lft)
\arrow[rru, "{LKE}"', dashed]
&
&
\end{tikzcd}
\]
\item the categories $\Str\Sch^\lft_\C$ and $\PSmall(\Str\Sch^\lft_\C)$ have finite limits (respectively by \Cref{StrSch admits finite limits}, and because the presheaf category $\PSh(\Str\Sch^\lft_\C)$ has all limits), 
\item $\Str\Sch^\lft_\C$ is small,
\item both $(-)^\an_{\Str, \lft}$ and $\Fgt_{\str, \textup{top}}$ preserve finite limits (respectively by \Cref{theoremanalyti} and \Cref{StrTopiscocomplete}).
\end{enumerate}
Hence the statement.
\end{proof}

\begin{rem}
Let us notice that are unstratified versions of the functors  $(-)^\an_{\Str}$, $(-)^\an_{\Str\PSh, \lft}$, $(-)^\an_{\Str\PSh}$ introduced in \cref{theoremanalyti}. 
Indeed one can similarly consider the left (or right) Kan extensions starting from $(-)^\an$: 
\[
\begin{tikzcd}[column sep=2cm]
    \Sch_\C^\lft
    \arrow[dr, hook, "\yo"']
    \arrow[dd, hook, "{i_{\lft}}"']
    \arrow[rrr, "(-)^\an"]
    & & &
    \Top 
    \arrow[dd, equal]
    \\
    & 
    \PSmall(\Sch_\C^\lft) 
    \arrow[dd, near start, "\lambda"']
    \arrow[rru, dashed, "{(-)^{\an}_{\PSh, \lft} 
    = 
    \textup{LKE}_{\yo} (-)^\an}"'] 
    & &
    \\
    \Sch_\C
    \arrow[dr, "{\yo}"', hook]
    \arrow[rrr, "{\hspace{2.5cm} (-)^\an_{\Sch} = \textup{RKE}_{i_{\lft}} (-)^\an}", dashed]
    & & &
    \Top.
    \\
    & \PSmall(\Sch_\C)
    \arrow[rru, "{(-)^\an_{\PSh} = \textup{LKE}_{\yo} (-)^\an_\Sch}"', dashed] 
    & &
\end{tikzcd}
\]
Since in this paper we are mainly interested in constructions involving stratifications, we will not make use of these unstratified versions. However, let us comment on the relationship between the stratified and unstratified versions.

We have three analogues to the top square of \cref{big-analytification-diagram}:
\[\begin{tikzcd}
{\Str\Sch_\C}& {\Str\Top} \\
{\Sch_\C} & \Top,
\arrow[from=1-1, to=1-2, "{(-)^\an_\Str}"]
\arrow[from=1-1, to=2-1, "\Fgt_\str"]
\arrow[from=1-2, to=2-2, "\Fgt_\str"']
\arrow[from=2-1, to=2-2, "{(-)^\an_\Sch}"]
\end{tikzcd}
\begin{tikzcd}[column sep=1.5cm]
{\PSmall(\Str\Sch_\C^\lft)} & {\Str\Top} \\
{\PSmall(\Sch_\C^\lft)} & \Top,
\arrow[from=1-1, to=1-2, "(-)^\an_{\PSh\Str, \lft}"]
\arrow[from=1-1, to=2-1, "\Fgt_\str"]
\arrow[from=1-2, to=2-2, "\Fgt_\str"']
\arrow[from=2-1, to=2-2, "(-)^\an_{\PSh, \lft}"]
\end{tikzcd}
\begin{tikzcd}[column sep=1.2cm]
	{\PSmall(\Str\Sch_\C)} & {\Str\Top} \\
	{\PSmall(\Sch_\C)} & \Top.
	\arrow[from=1-1, to=1-2, "(-)^\an_{\PSh\Str}"]
	\arrow[from=1-1, to=2-1, "\Fgt_\str"]
	\arrow[from=1-2, to=2-2, "\Fgt_\str"']
	\arrow[from=2-1, to=2-2, "(-)^\an_{\PSh}"]
\end{tikzcd}\]

The last two squares commute, because the forgetful functors preserve colimits and the horizontal maps are defined as left Kan extensions.  
The same argument cannot be run for the first square, since the horizontal maps are right Kan extensions and the forgetful functors do not preserve limits in general.
\end{rem}

\subsection{Topological realizations over \texorpdfstring{$X^I$}{}}
		
We remark that, in contrast to the approach of \cite{E3}, we choose to ``realize'' $(\Gr_G,\mathfrak{s}),(\Gr_{G,X^I},\mathfrak{s}_I)$ (and $(\Gr_{G,\Ran(X)},\mathfrak{s}_\Ran)$ later in \Cref{Topological realizations part 2}) in the category $\Str\Top$, instead of viewing them as presheaves on $\Str\Top$. 
As we will see especially in \Cref{Topological realizations part 2}, this makes the proof of certain properties less trivial, and ultimately relying on categorical features of locally compact Hausdorff topological spaces.

\begin{rem}\label{analytification-of-GrBD}
Thanks to \Cref{theoremanalyti}, we can formally talk about the analytification in $\Str\Top$ of stratified schemes and presheaves, such as
\begin{enumerate}
\item $(\Gr_G^{(N)}, \mathfrak s^{(N)})^{\an}_{\Str,\lft} = (\Gr_G^{(N), \an }, \mathfrak s^{(N),\an})$, $\left( \Gr^{(N)}_{G,X^I},  \mathfrak{s}_I^{(N)} \right)^\an_{\Str,\lft} = \left( \Gr^{(N), \an}_{G,X^I},  \mathfrak{s}_I^{(N), \an} \right)$;
\item $(\LmG,\textup{triv})^\an_{\Str,\lft} = (\LmG^\an,\textup{triv})$, $(\LmG_{X^I}, \Inc_I)^\an_{\Str,\lft} = (\LmG_{X^I}^\an, \Inc_I)$. 
\end{enumerate}
Since $(-)^\an_{\Str,\lft}$ preserves finite limits, the group structure of $\LmG$ (respectively $\LmG_{X^I}$ over $(X^{I, \an}, \Inc_I)$) is preserved, making it an object of $\Grp(\Str\Top)$ (respectively $\Grp\left(\Str\Top/_{(X^{I,\an},\Inc_I)}\right)$).
Moreover $\forall\, m\geq m_N$ we have a stratified action 
    $$
    (\LmG^\an,\textup{triv}) \times (\Gr_G^{(N),\an}, \mathfrak{s}^{(N),\an})\to (\Gr_G^{(N), \an}, \mathfrak{s}^{(N),\an}),
    $$ 
and $\forall\, m\geq m_{N, I}$ we have a stratified action over $(X^{I, \an},\Inc_I)$ 
    $$
    (\LmG_{X^I}^\an,\Inc_I) \underset{(X^{I,\an},\Inc_I)}{\times} \left(\Gr^{(N), \an}_{G,X^I},  \mathfrak{s}_I^\an \right) \to \left(\Gr^{(N),\an}_{G,X^I},  \mathfrak{s}_I^\an  \right).
    $$
Furthermore, since $(-)^\an_{\PSh\Str, \lft}$ preserves small colimits, we have the following equalities in $\Str\Top$:
\begin{enumerate}
    \item $(\Gr_G, \mathfrak s)^{\an}_{\PSh\Str,\lft} = \underset{N\in \nn}{\colim}\ \left(  \Gr_G^{(N), \an}, \mathfrak s^{(N),\an} \right)$;
    \item $\left(\Gr_{G,X^I}, \mathfrak{s}_I \right)^{\an}_{\PSh\Str, \lft} = \underset{N\in \nn}{\colim} \left( \Gr^{(N)}_{G,X^I}, \mathfrak{s}_I^{(N)}\right)^\an$, in $\Str\Top$ over $(X^{I, \an}, \Inc_I)$;
\end{enumerate}
Similarly, by the definition of $(-)^\an_{\Str}$ and by the fact that it preserves arbitrary small limits, we have that
\begin{enumerate}
    \item $(\LpG,\textup{triv})^\an_{\Str} = \left( \underset{m\in \nn}{\lim} \LmG^\an, \textup{triv}\right)$, the group structure is preserved, making it an object of $\Grp(\Str\Top)$ and $\forall\, N$ we have a stratified action 
    $$
    (\LpG,\textup{triv})^\an_\Str \times (\Gr_G^{(N),\an}, \mathfrak{s}^{(N), \an})\to (\Gr_G^{(N), \an}, \mathfrak{s}^{(N), \an});
    $$ 
    \item  $(\LpG_{X^I},\Inc_I )^\an_\Str = \left(\underset{m\in \nn}{\lim} \LmG_{X^I}^\an,\Inc_I\right)$, the group structure over $(X^{I, \an}, \Inc_I)$ is preserved, making it into an object of $\Grp\left(\Str\Top/_{(X^{I,\an},\Inc_I)}\right)$ and $\forall\, N$ we have a stratified action over $(X^{I, \an},\Inc_I)$ 
    $$
    (\LpG_{X^I},\Inc_I)^\an_\Str \underset{(X^{I,\an},\Inc_I)}{\times} \left(\Gr^{(N), \an}_{G,X^I},  \mathfrak{s}_I^{(N),\an} \right) \to \left(\Gr^{(N),\an}_{G,X^I},  \mathfrak{s}_I^{(N),\an}  \right).
    $$
\end{enumerate}
\end{rem}

\begin{warning}\label{warning}
The reader may notice that the (relative) group actions of $(\LpG,\textup{triv})^\an_\Str$ and $(\LpG_{X^I}, \Inc_I)^\an$ on respectively $(\Gr_G,\mathfrak{s})^\an_{\PSh\Str,\lft}$ and $(\Gr_{G,X^I},\mathfrak{s}_I)^\an_{\PSh\Str,\lft}$ have been left out of the statement of \cref{analytification-of-GrBD}.
This is because universality of colimits fails in $\Str\Top$ (just like it fails in $\Top$): therefore passing to the colimit in $N$ may not commute with the pullback \textit{a priori}.
\end{warning}

The key fact that makes us overcome this issue is that our ``building blocks'', namely the $\LmG_{X^I}^\an, \Gr_{G,X^I}^{(N),\an}$'s, are locally compact Hausdorff topological spaces.
Indeed, let us recall the following result by Harpaz.

\begin{prop}[{\cite{Harpaz-colimits-finite-limits}}]\label{Har15}
Consider three $\N$-indexed diagrams of topological spaces $(X_i)_{i\in \N}$, $(Y_i)_{i\in \N}$, $(Z_i)_{i\in \N}$ whose transition maps are all closed embeddings
\[
 X_i\xhookrightarrow{} X_{i+1}, \quad  Y_i\xhookrightarrow{} Y_{i+1}, \quad Z_i\xhookrightarrow{} Z_{i+1}.
\]
Let $f_i:X_i\rightarrow Z_i$ and $g_i:Y_i\rightarrow Z_i, i\in \N$, be morphisms compatible with the transition maps. For every $i\in \N$, consider the cartesian square
\[
\begin{tikzcd}
    X_i \times_{Z_i} Y_i
    \arrow[d, "\pi_{X,i}"] \arrow[r, "\pi_{Y,i}"]
    &
    Y_i
    \arrow[d, "g_i"]
    \\
    X_i
    \arrow[r, "f_i"]
    \
    &
    Z_i.
\end{tikzcd}
\]
Assume that
\begin{enumerate}
    \item $X_i$'s, $Y_i$'s are locally compact and Hausdorff, and
    \item $Z_i$'s are Hausdorff.
\end{enumerate}
Then the natural map
$$
\colim_{i\in \N}(X_i\times_{Z_i}Y_i) \xrightarrow{ \underset{i\in \N}{\colim}(\pi_{X,i}) \times \underset{i\in \N}{\colim}(\pi_{Y,i})  } (\colim_{i\in N}X_i) \underset{\underset{i\in \N}{\colim} Z_i}{\times} (\colim_{i\in \N}Y_i)
$$
is an isomorphism.
\end{prop}

Note that we cannot automatically extend this result to $\Str\Top$, because, unless trivial, $\Alex(P)$ of a poset $P$ is locally compact but not Hausdorff.
So let us consider a restrictive setting, which, nevertheless, will be enough for our discussion.

\begin{cor}\label{stratified-Harpaz}
Consider three $\N$-indexed diagrams of stratified topological spaces
\[
(Z_i, Q_i, t_i: Z_i \rightarrow \Alex(Q_i) )_{i\in \N}, \quad
(X_i, R_i, u_i: X_i \rightarrow \Alex(R_i) )_{i\in \N}, \quad
(Y_i, P_i, s_i: Y_i \rightarrow \Alex(P_i) )_{i\in \N},  
\]
together with compatible stratification-preserving maps $f_i:X_i\to Z_i$, $g_i:Y_i\to Z_i$.
Suppose that the underlying topological data $(X_i,Y_i,Z_i, f_i, g_i)_{i\in \N}$ satisfy the conditions of \Cref{Har15}.
Assume furthermore that $u_i\simeq t_i\circ f_i$ (in particular $R_i\to Q_i$ is an isomorphism of posets).
Then the induced stratified morphism
\[
\colim_{i\in \N} \left( (X_i, u_i) \underset{(Z_i, t_i)}{\times} (Y_i, s_i) \right) 
\xrightarrow{\underset{i\in \N}{\colim}(\pi_{X,i}) \times \underset{i\in \N}{\colim}(\pi_{Y,i})}
\left( \colim_{i\in \N} (X_i, u_i) \right) \underset{\underset{i\in \N}{\colim} (Z_i,t_i)}{\times}  \left(\colim_{i\in \N} (Y_i, s_i) \right)
\]
is an isomorphism in $\Str\Top$.
\end{cor}
\begin{proof}
    Each morphism $\pi_{X,i}$ is of the form $(\pi_{X,i}^\textup{top},\pi_{X,i}^\textup{pset})$. 
    Same for $\pi_{Y,i}$. 
    \Cref{Har15} tells us that $\underset{i\in \N}{\colim}(\pi_{X,i}^\textup{top}) \times \underset{i\in \N}{\colim}(\pi_{Y,i}^\textup{top})$ is an isomorphism.
    Since $R_i \xrightarrow{\sim} Q_i$, the map
    \[
    \underset{i\in \N}{\colim}(\pi_{X,i}^\textup{pset}) \times \underset{i\in \N}{\colim}(\pi_{Y,i}^\textup{pset}): \underset{i\in \N}{\colim} (R_i\times_{Q_i} P_i) 
    \rightarrow \underset{i\in \N}{\colim}(R_i) \underset{\underset{i\in \N}{\colim} Q_i}{\times}  \underset{i\in \N}{\colim} (P_i)
    \]
    is an isomorphism.
\end{proof}
 
\begin{rem}\label{Har15-and-an}
Consider now $\N$-indexed diagrams of stratified schemes locally of finite type over $\C$
\[
 (X_i,u_i)_{i\in \N}  \xrightarrow{f_i} (Z_i,t_i)_{i\in \N} \xleftarrow{g_i}  (Y_i,s_i)_{i\in \N},
\] 
where the transition maps in $i$ are closed embeddings, the $f_i, g_i$ are compatible with the three diagrams and such that $u_i = t_i\circ f_i$ for every $i\in \N$.
Then the family of diagrams
\[
 (X_i^\an,u_i^\an)_{i\in \N} \xrightarrow{f_i^\an} (Z_i^\an,t_i^\an)_{i\in \N} \xleftarrow{g_i^\an}  (Y_i^\an,s_i^\an)_{i\in \N}
\]
obtained by analytification satisfies the conditions of \Cref{stratified-Harpaz} and thus the colimit commutes with the fiber product.
\end{rem}

\begin{prop}\label{prop:analitificationofthelowerobjs}
The action of $(\LpG,\textup{triv})_\Str^\an$ defined in \Cref{analytification-of-GrBD} extends to a stratified action on $(\Gr_G,\mathfrak{s})^\an_{\PSh\Str,\lft}$ (thus, compatible with the actions at $N$-th level for every $N$).

Analogously, the action of $(\LpG_{X_I},\Inc_I)_\Str^\an$ defined in \Cref{analytification-of-GrBD} extends to a stratified action on $(\Gr_{G, X^I},\mathfrak{s})^\an_{\PSh\Str,\lft}$ over $(X^{I,\an},\Inc_I)$ (thus, compatible with the actions at the $N$-th level for every $N$).
\end{prop}

\begin{proof}	
Consider the stratified actions
\begin{equation}\label{(N)-actions}
\begin{gathered}
(\LpG,\textup{triv})^\an_\Str \times (\Gr_G^{(N)}, \mathfrak{s}^{(N)})^\an_{\Str,\lft}\to (\Gr_G^{(N)}, \mathfrak{s}^{(N)})^\an_{\Str,\lft},
\\
(\LpG_{X^I},\Inc_I)^\an_\Str \times (\Gr_{G,X^I}^{(N)}, \mathfrak{s}_I^{(N)})^\an_{\Str,\lft}\to (\Gr_{G,X^I}^{(N)}, \mathfrak{s}_I^{(N)})^\an_{\Str,\lft}.
\end{gathered}\end{equation}
Notice that each $\Gr_{G}^{(N), \an}$, resp. $\Gr_{G,X^I}^{(N), \an}$, is locally compact Hausdorff, being the analytification of a projective variety, resp. a projective variety over $X^I$. 
The same holds for $\LmG^\an, \LmG_{X^I}^\an$, and hence for $\LpG^\an, \LpG_{X^I}^\an$, since limits of locally compact Hausdorff spaces are locally compact Hausdorff.

To get the wanted actions on $(\Gr_G, \mathfrak s)_{\PSh\Str, \lft}^\an, (\Gr_{G,X^I},\mathfrak s_I)^\an_{\PSh\Str,\lft}$, it suffices to apply \cref{stratified-Harpaz} and pass to the colimit in $N$ in the expressions \eqref{(N)-actions}.
\end{proof}

\section{The Beilinson-Drinfeld Grassmannian over the Ran space}\label{section-Ran-Grassmannian}

\subsection{Stratification of the Ran Grassmannian}

In Appendix \ref{recollection} we recall definitions and properties of the Beilinson--Drinfeld Grassmannians $\Gr_{G, X^I}$'s relevant for the present work: in particular we see how they carry a stratification, $(\Gr_{G, X^I},\mathfrak{s}_I)$, see \Cref{BD-stratification}. 
In this section we will recall how to combine them into one stratified small presheaf. 
We also provide a topological realization with the complex-analytic topology.

Let us start by putting together the different $X^I$'s.
	
\begin{defin}[{\cite[Definition 3.3.1]{Zhu}}]
The \emph{Ran presheaf} of $X$ is the functor of unordered non-empty finite sets of \emph{distinct} points on $X$.
Precisely, it is defined as
\begin{gather*}
    \Ran(X): \textup{Aff}_{\mathbb{C}}^{\textup{op}}
    \to \sets, 
    \\
    \Spec R \ \mapsto\ \{ \underline{x}=\{x_1,\dots , x_k\} \subset X(R) \mbox{ non-empty and finite}\}.
\end{gather*}
This is what is called $\Ran^u(X)$ in \cite[Definition 2.4.2]{Gaitsgory-Lurie}.
\end{defin}
Let $\Delta_\phi$ the diagonal embedding associated to a surjective map $\phi: I\twoheadrightarrow J$ (see \Cref{BD-stratification}).
	
\begin{lem}\label{Ran-is-colimit}
We have an isomorphism of $\PSh(\Aff)$ 
$$
\Ran(X)\simeq \underset{I\in \fs^{\op}}{\colim} X^I
$$ 
where the transition maps are the $\Delta_{\phi}$'s and the colimit is taken in $\PSh(\Aff_\C)$. In particular, $\Ran(X)$ is an element of $\PSmall(\Aff_\C)$.
\end{lem}
	
\begin{proof}
Fix $I\in \fs$.
Consider the \emph{unordering functor}
$$
\mathcal{U}_I: X^I \to \Ran(X),\quad \quad x_I=(x_1,\dots, x_{|I|})\mapsto \{ x_1',\dots, x_k' \}
$$
where we forget the order of the $x_i$'s and we do not repeat maps that are equal (so $k$ is the number of different maps in $x_I$). 
Notice that for any $J,\phi:I\twoheadrightarrow J$, we have $\mathcal{U}_J=\mathcal{U}_I\circ \Delta_\phi$. 
Hence we get a well-defined surjective map
$$
\mathcal{U}: \underset{I \in  \fs^{\op} }{\colim} X^I \to \Ran(X).  
$$
Let us check that it is injective as well.
Suppose that $x_I\in X^I$ and $y_{I'}\in X^{I'}$ are sent to the same $\{x_1',\dots,x_k'\}$.
Fix an order on $\{x_1',\dots,x_k'\}$: $(x_{1}',\dots,x_{|J|}')$ where $J$ has cardinality $k$. 
Define
\begin{align*}
    \psi: I\twoheadrightarrow J,\quad \quad & \psi(i)=j \iff  x_i= x_{j}'
    \\
    \psi':I'\twoheadrightarrow J,\quad \quad & \psi'(i')=j \iff  y_{i'}= x_{j}'.
\end{align*}
Consider now the fiber product $I\times_J I'$
\[
\begin{tikzcd}
    I\times_J I'
    \arrow[r, "{p_2}"] \arrow[d, "{p_1}"']
    &
    I'
    \arrow[d, "{\psi'}"]
    \\
    I
    \arrow[r, "{\psi}"']
    &
    J
\end{tikzcd}
\]
and the element $z_{I\times_J I'}$ in $X^{I\times_J I'}$ defined as $z_{(i,i')}=x_i=y_{i'}$: then $\Delta_{p_1}(x_I)=\Delta_{p_2}(y_{I'})=z_{I\times_J I'}$, making $x_I$ and $y_{I'}$ the same element in the colimit.
This proves that $\mathcal{U}$ is an isomorphism in $\PSh(\Aff_\C)$.
Finally, $X^I$ coincides with the small colimit of all its affine open subschemes.
Since composition of small colimits is small, we have that $\Ran(X)$ is actually an element of $\PSmall(\Aff_\C)$.
\end{proof}
	
\begin{defin}[{\cite[Definition 3.3.2]{Zhu}}]\label{defin-Ran-Grassmannian} 
The \emph{Ran Grassmannian}
associated to $G$ and $X$ is the presheaf
\begin{gather*}
    \Gr_{G, \Ran(X)}:\Aff^\op_\C\to \sets,
    \\
    \Spec R
    \mapsto 
    \{(\underline{x}, \cF, \alpha):
    \underline{x} \in \Ran (X)(R), 
    \cF\in \Bun_G(X_R), 
    \alpha: \cF|_{X_R\setminus \Gamma_{\underline{x}}}\triv\cT_{G,X_R\setminus \Gamma_{\underline{x}}} \} / \sim 
\end{gather*}
(where the equivalence relation is the analogous of the one for $\Gr_{G, X^I}$, see \Cref{BD-Grassmannian}, and $\Gamma_{\underline{x}}$ is the union of the graphs, see \Cref{graphs-of-points}).
On morphisms, $\Gr_{G,\Ran(X)}$ sends 
$$
[\Spec S \overset{f}{\to} \Spec R] \mapsto 
\bigg[[(\underline{x},\cF, \alpha)]
\mapsto
[(\underline{x}\circ f, (\id\times f)^*\cF, (\id\times f)^*\alpha)]\bigg].
$$
\end{defin}

Our definition aligns with \cite{Zhu}'s and \cite{James}'s. 
A groupoid-valued version, $\Ran^u_G(X)$, is considered in \cite[Definition 3.2.3]{Gaitsgory-Lurie}: if $\pi_0$ denotes the functor 
$$
\Fun(\Aff_\C^\op,\textup{Grpd})\to \Fun(\Aff_\C^\op, \sets)
$$ induced by $\pi_0:\textup{Grpd}\to \sets$, then 
$$
\Gr_{G,\Ran(X)}\simeq \pi_0\Ran^u_G(X).
$$

\begin{defin}
Define $\delta_\phi:\Gr_{G,X^J}\to \Gr_{G,X^I}$ to be the morphism 
\begin{gather*}
    (x_J',\cF,\alpha)
    \mapsto (\Delta_{\phi}(x_J'),\cF,\alpha).
\end{gather*} 
Note that this definition is well posed since $\Gamma_{x_J}=\Gamma_{\Delta_{\phi}(x_J')}$ as closed topological subspaces of $X_R$.
\end{defin}

\begin{lem}\label{Ran-Gr-is-colim}
For each $I\in \fs$, $\Gr_{G,X^I}$ coincides with the pullback (taken in $\PSh(\Aff_\C)$) 
$$X^I\times_{\Ran(X)}\Gr_{G,\Ran(X)},$$ 
where the map $X^I\to \Ran(X)$ is $\cU_I$.
Moreover, there is an isomorphism of presheaves in $\PSh(\Aff_\C)$
$$
\Gr_{G,\Ran(X)}\simeq \underset{I\in\fs^{\op}}{\colim} \Gr_{G,X^I}\simeq \underset{I\in\fs^{\op}, N\geq 0}{\colim} \Gr_{G,X^I}^{(N)}
$$ 
where the transition maps in the variable $I$ are the $\delta_\phi$'s and the colimits are taken in $\PSh(\Aff_\C)$.
In particular, $\Gr_{G,\Ran(X)}$ is an element of $\PSmall(\Aff_\C)$.
\end{lem}
\begin{proof}
The first part follows directly from the definition, since for any $x_I\in X^I(R)$, $\Gamma_{x_I}$ only depends on $\cU_I(x_I)$.
By universality of small colimits, we get
$$
\underset{I\in \fs^\op}{\colim}  \Gr_{G,X^I}\simeq \underset{I\in \fs^\op}{\colim} \left( {X^I}\times_{\Ran(X)}\Gr_{G,\Ran(X)} \right) \simeq \left(\underset{I\in \fs^\op}{\colim} X^I \right)\times_{\Ran(X)}\Gr_{G,\Ran(X)}
$$
which is isomorphic to $\Gr_{G,\Ran(X)}$ by \Cref{Ran-is-colimit}.
As observed at the end of the proof of \Cref{Ran-is-colimit}, both $\Gr_{G, X^I}^{(N)}$'s and $\Gr_{G, X^I}$ can be viewed as small presheaves because they are ind-schemes.
Thus so is $\Gr_{G,\Ran(X)}$.
\end{proof}

\begin{rem}\label{Ran-from-Aff-to-Sch}\label{GrRan-from-Aff-to-Sch}
Consider the left Kan extension $\sigma$
\[
\begin{tikzcd}
    \Aff_\C
    \arrow[r, hook]
    \arrow[d, "{\yo}", hook]
    &
    \Sch_\C 
    \arrow[r, hook, "{\yo}"]
    &
    \PSmall(\Sch_\C).
    \\
    \PSmall(\Aff_\C)
    \arrow[rru, dashed, "{\sigma}"']
\end{tikzcd}
\]
Unlike the analogous functor between categories of sheaves, this functor is not an equivalence.
Nevertheless, it preserves colimits and it is left adjoint to the restriction morphism $\PSmall(\Sch_\C) \to \PSmall(\Aff_\C)$.
In particular, 
\begin{align*}
\sigma(\Ran(X))\simeq & \ \underset{I\in \fs^{\op}}{\colim}^{\hspace{-0.4cm}\PSh(\Sch_\C)} X^I,
\\
\sigma(\Gr_{G,\Ran(X)})\simeq & \ \underset{I\in\fs^{\op}}{\colim}^{\hspace{-0.4cm}\PSh(\Sch_\C)} \sigma(\Gr_{G,X^I}) \simeq \ \underset{I\in\fs^{\op}, N\geq 0}{\colim}^{\hspace{-0.7cm}\PSh(\Sch_\C)} \Gr_{G,X^I}^{(N)}.
\end{align*}
Similarly, since $X$ and $\Gr_{G,X^I}^{(N)}$ are locally of finite type over $\C$, the objects $\sigma(\Ran(X))$ and $\sigma(\Gr_{G,X^I})$ lie in the essential image of $\lambda$ (defined \Cref{embedding-strschlft-strsch}).
And so does $\sigma(\Gr_{G,\Ran(X)})$.
\end{rem}  

\begin{notation}
    Later it will be useful to identify $\Ran(X), \Gr_{G,X^I}$ and $\Gr_{G,\Ran(X)}$ with their images under $\sigma$ in $\PSmall(\Sch^\lft_\C)$. 
    Hence, from now on, we will see $\Ran(X)$, $\Gr_{G,X^I}$ and $\Gr_{G,\Ran(X)}$ as objects of $\PSmall(\Sch^\lft_\C)$. 
\end{notation}

\begin{lem}\label{lemma:deltaphistratified}
    The maps $\delta_\phi$'s from \cref{defin-Ran-Grassmannian} respect the stratification in \Cref{BD-stratification}, making
    \[
    \delta_\phi: \left(\Gr_{G,X^J}, \mathfrak{s}_J\right) \to 
    \left(\Gr_{G,X^I}, \mathfrak{s}_I\right)
    \]
    into a map of stratified small presheaves locally of finite type over $\C$ (so in $\PSmall(\Str\Sch^\lft_\C)$).
\end{lem}
	
\begin{proof}
    Recall that the $X^I$'s are endowed with the incidence stratification, with respect to which the maps $\Delta_\phi$'s are indeed stratified. 
    Consider the stratum $\Gr_{G, X^{\psi}, \underline{\nu}}$ inside $\Gr_{G,X^J}$ indexed by $([J\overset{\psi}{\twoheadrightarrow} J'],\underline{\nu} \in (\xt)^{|J'|})$.
    The map $\delta_\phi$ sends $\Gr_{G, X^{\psi}, \underline{\nu}}$ into the stratum $\Gr_{G, X^{\psi\circ \phi},\underline{\nu}}$ of $\Gr_{G,X^I}$ indexed by $([I\overset{\psi\circ \phi}{\twoheadrightarrow} J'], \underline{\nu}\in ( \xt)^{|J'|})$.
    Thus, the $\delta_\phi$'s are stratified.
\end{proof}
	
\begin{prop}[Stratification of $\Ran(X)$ and of $\Gr_{G, \Ran(X)}$]\label{Ran-stratification}
There exists a stratified small presheaf $(\Ran(X),\Inc_\Ran)$, locally of finite type over $\C$, whose underlying presheaf is $\Ran(X)$, which  recovers the incidence stratification $(X^I,\Inc_I)$ when pulled-back along $X^I\xrightarrow{\mathcal{U}_I} \Ran(X)$.

Analogously, there exists a stratified small presheaf $(\Gr_{G, \Ran(X)}, \mathfrak{s}_{\Ran})$, locally of finite type over $\C$, whose underlying presheaf is $\Gr_{G, \Ran(X)}$, which recovers $(\Gr_{G, X^I}, \mathfrak{s}_I)$ when pulled back along $\mathcal{U}_I$.
\end{prop}

\begin{proof}
Both $(X^I,\Inc_I)$'s and $(\Gr_{G,X^I}^{(N)}, \mathfrak{s}_{I})$'s are objects of $\Str\Sch^\lft_\C$ and $\Delta_\phi, \delta_\phi$'s are stratified maps (\Cref{lemma:deltaphistratified}).
Consider then the following colimits in $\PSmall(\Str\Sch^\lft_\C)$
\begin{equation}\label{stratified-Ran-forgets} 
\underset{I\in \fs^\op}{\colim} (X^I, \Inc_I), 
\quad 
\underset{ I\in \fs^\op}{\colim} (\Gr_{G, X^I}, \mathfrak{s}_I) 
\cong   
\underset{I\in \fs^\op, N\in\nn}{\colim} 
(\Gr^{(N)}_{G, X^I}, \mathfrak{s}_{I}).
\end{equation}
Since the forgetful functor preserves colimits, by \Cref{Ran-is-colimit} and \Cref{Ran-Gr-is-colim} we have 
\begin{equation*}
    \Fgt_\str 
    \left( \underset{I\in \fs^\op}{\colim} (X^I,\Inc_I) \right) 
    \simeq 
    \Ran(X),\quad
    \Fgt_\str
    \left( \underset{I\in \fs^\op}{\colim} (\Gr_{G,X^I},\mathfrak{s}_I) \right)
    \simeq \Gr_{G,\Ran(X)}.
\end{equation*}
Finally, since both stratified presheaves are defined as colimits, pulling back along the colimit map $\mathcal{U}_I$ recovers the $I$-th level by universality of colimits in the category of stratified presheaves.
\end{proof}

\subsection{Stratified action of \texorpdfstring{$\LpG_{\Ran(X)}$}{} on \texorpdfstring{$\Gr_{G,\Ran(X)}$}{}}
In this subsection, we describe a \emph{stratified group presheaf} $(\LpG_{\Ran(X)}, \Inc_\Ran)$ and its action on the stratified small presheaf $(\Gr_{G,\Ran(X)},\mathfrak{s}_{\Ran})$ relative to $(\Ran(X), \Inc_\Ran)$.

\begin{defin}[Ran version of $\LpG$]
Define
$$
\LpG_{ \Ran(X)}: \Aff^{\op}_{\C}\to \sets, \quad \quad
\Spec R \mapsto
\{ (\underline{x},g):  \underline{x} \in \Ran(X)(R), g\in G(\widetilde{\Gamma}_{\underline{x}})\}.
$$
This is well defined because the scheme $\tilde\Gamma_{\underline{x}}$ depends neither on the order of the points nor on the schematic structure of $\Gamma_{\underline{x}}$ (only on its topology).
\end{defin}

\begin{lem}\label{LpG-Ran-is-colim}
For any $I\in \fs$, we have the following isomorphisms in $\PSmall(\Aff_\C)$:
$$\LpG_{X^I}\simeq X^I\times_{\Ran(X)}\LpG_{\Ran(X)}, \quad \LpG_{\Ran(X)}\simeq \underset{I\in \fs^{\op}}{\colim} \LpG_{X^I},$$
where transition maps in the second colimit are $\delta_{\phi}^{\textup{grp}}: (x_I,g)\mapsto (\Delta_\phi(x_I),g)$. 
\end{lem}

\begin{proof}
Analogous to the proof of \Cref{Ran-Gr-is-colim}.
\end{proof}
	
\begin{rem}\label{LpG-from-Aff-to-Sch}
By the same argument of \cref{Ran-from-Aff-to-Sch}, we can see $\LpG_{\Ran(X)}$ as an object of $\PSmall(\Sch_\C)$, which we will do from now. 
Note that the $\LpG_{X^I}$'s are not locally of finite type over $\C$ and the presheaf $\LpG_{\Ran(X)}$ does not lie in the essential image of $\lambda$.
\end{rem}

\begin{prop}\label{LpGRan-is-group}
There exists a stratified small presheaf $(\LpG_{\Ran(X)},\Inc_\Ran)$ whose underlying presheaf is $\LpG_{\Ran(X)}$, that recovers $(\LpG_{X^I},\Inc_I)$ when pulled-back along $\LpG_{X^I} \xrightarrow{\mathcal{U}_I}  \LpG_{\Ran(X)}$.

Moreover there exists a multiplication law which makes $(\LpG_{\Ran(X)},\Inc_\Ran)$ into an element of $${\Grp\left({\PSmall({\Str\Sch_\C})}/_{(\Ran(X),\Inc_\Ran)}\right)}$$ and recovers $(\LpG_{X^I}, \Inc_I)\in \Grp({\Str\Sch_\C}/_{(X^I,\Inc_I)})$ after pullback to $(X^I,\Inc_I)$.
\end{prop}
\begin{proof}
Forgetting the group structure, $(\LpG_{X^I}, \Inc_I)$ is an element of ${\PSmall(\Str\Sch_\C)/}_{(X^I,\Inc_I)}$.
Via the composite
\[
(\LpG_{X^I},\Inc_I) \to (X^I,\Inc_I) \to (\Ran(X), \Inc_\Ran)
\]
we actually have that $(\LpG_{X^I},\Inc_I) \in {\PSmall(\Str\Sch_\C)/}_{(\Ran(X),\Inc_{\Ran})}$.
Therefore, with the same argument done in proof of \cref{Ran-stratification}, by \Cref{LpG-Ran-is-colim} there exists a stratified small presheaf 
\[
(\LpG_\Ran(X), \Inc_\Ran) \cong \underset{I\in \fs}{\colim} (\LpG_{X^I}, \Inc_I)
\]
which recovers $(\LpG_{X^I}, \Inc_I)$'s by pull-back.

We now want to see that there is a multiplication law on $(\LpG_{\Ran(X)}, \Inc_\Ran)$ that respects this pullback.
At the level of the underlying presheaf $\LpG_{\Ran(X)}$, it is defined as 
\begin{equation}\label{eq-multiplication-Ran}
\LpG_{\Ran(X)}\times_{\Ran(X)}\LpG_{\Ran(X)} \to \LpG_{\Ran(X)}, \quad (\underline{x},g).(\underline{x},h)\mapsto (\underline{x}, gh).
\end{equation}
To check that it is stratified, we describe it in a different way.
Consider the colimit
\[
\colim_{I\in \fs^\op} \left( (\LpG_{X^I},\Inc_I)\times_{(X^I,\Inc_I)} (\LpG_{X^I}, \Inc_I), \right)\quad \textup{taken in } {\PSmall(\Str\Sch_\C)}/_{(\Ran(X),\Inc_\Ran)}.
\]
By replacing $(\LpG_{X^I}, \Inc_I)$ with $(\LpG_{\Ran(X)}, \Inc_\Ran)\times_{(\Ran(X),\Inc_\Ran)} (X^I,\Inc_I)$, the previous colimit can be written as
\[
\colim_{I\in \fs^\op}\left(
( \LpG_{\Ran(X)}, \Inc_\Ran) 
\underset{(\Ran(X), \Inc_\Ran)}{\times}
( \LpG_{\Ran(X)}, \Inc_\Ran) 
\underset{(\Ran(X), \Inc_\Ran)}{\times} (X^I,\Inc_I)
\right).
\]
By universality of colimits in ${\PSmall ({\Str\Sch_\C})}/_{(\Ran(X),\Inc_{\Ran})}$, this is exactly
\[
( \LpG_{\Ran(X)}, \Inc_\Ran) 
\underset{(\Ran(X), \Inc_\Ran)}{\times}
( \LpG_{\Ran(X)}, \Inc_\Ran).
\]
In this way, the multiplication law \eqref{eq-multiplication-Ran} can be presented as a colimit of the multiplication laws of $(\LpG_{X^I},\Inc_I)$'s, and hence it is stratified. Applying universality of colimits in the other direction we see that it recovers the multiplication on $\LpG_{X^I}$ when pull-backed.
\end{proof}

\begin{prop}
There exists a map in ${\PSmall({\Str\Sch_\C})/}_{(\Ran(X),\Inc_\Ran)}$ 
$$
\act_\Ran : (\LpG_{\Ran(X)}, \Inc_\Ran) \underset{(\Ran(X), \Inc_\Ran)}{\times} (\Gr_{G,\Ran(X)}, \mathfrak{s}_\Ran)
\to (\Gr_{G,\Ran(X)}, \mathfrak{s}_\Ran),
$$
recovering the action $\act_I$ of \cref{action-BD} when pulled back to $X^I$.
\end{prop}
			
\begin{proof}
By universality of small colimits, it is enough to give $\delta_\phi$-compatible actions
\[
\act_{\Ran, I} : (\LpG_{\Ran(X)}, \Inc_\Ran) \underset{(\Ran(X), \Inc_\Ran)}{\times} (\Gr_{G,X^I}, \mathfrak{s}_I)
\to (\Gr_{G,X^I}, \mathfrak{s}_I),
\]
and then pass to the colimit on both sides.
The LHS is the same as
\[
(\LpG_{\Ran(X)}, \Inc_\Ran) \underset{(\Ran(X), \Inc_\Ran)}{\times} (X^I, \Inc_I) \underset{(X^I, \Inc_I)}{\times}(\Gr_{G,X^I}, \mathfrak{s}_I),
\]
which is isomorphic to $(\LpG_{X^I}, \Inc_I)\times_{(X^I, \Inc_I)}(\Gr_{G,X^I}, \mathfrak{s}_I)$ by \Cref{LpGRan-is-group}.
Hence we can define $\act_{\Ran, I}$ as $\act_I$ (see \Cref{action-BD} and \Cref{stratifiedactions}).
We now only need to show that the $\act_I$'s are $\delta_\phi$-compatible.

This follows from noticing that, for any locally closed subscheme $\Gr_{G, X^{\psi}}$, $\psi: J \twoheadrightarrow L$, the map $\delta_\phi$ becomes the identity via the isomorphism \eqref{factorization-property} of the factorization property
\[
\begin{tikzcd}
\left ( \prod_{l=1}^{|L|}\Gr_{G,X} \right )_{\disj} 
\arrow[r, "\id"]
\arrow[d, "\mathfrak{f}_{\psi}"]
&
\left ( \prod_{l=1}^{|L|}\Gr_{G,X} \right )_{\disj} 
\arrow[d, "\mathfrak{f}_{\psi\circ\phi}"]
\\
\Gr_{G, X^\psi}
\arrow[r, "\delta_\phi"]
&
\Gr_{G, X^{\psi\circ \phi}}.
\end{tikzcd}
\]
\end{proof}

\subsection{Topological realizations {over the Ran space}}\label{Topological realizations part 2}

We are now ready to apply the analytification functors introduced in \cref{theoremanalyti}.

\begin{defin}\label{topological-incidence}
    Let $M$ be a topological manifold, and $I\in\fs$. The \textit{incidence stratification} on $M^I$ is the one having as poset $$\{[\phi:I\to J]\textup{ partition of }I\}$$ and defined by 
    $$(m_1,\dots, m_{|I|})\in [\phi]\iff (m_{i}=m_{i'}\ \forall \ \phi(i)=\phi(i')\in J).
    $$
\end{defin}
			
\begin{rem}
The analytification of the incidence stratification on $X^I$ (in the sense of \cref{BD-stratification}) coincides with the incidence stratification on $(X^\an)^I$ (in the sense of \cref{topological-incidence}).
\end{rem}

\begin{cor}\label{analytification-of-GrBD-Ran}
By \Cref{theoremanalyti}, we have the following analytifications and equalities in $\Str\Top$:
\begin{enumerate}
    \item $\left(\Ran(X), \Inc_\Ran \right)^\an_{\PSh\Str,\lft} = \colim_{I\in \fs^\op}^{\Str\Top}(X^{I, \an}, \Inc_I)$;
    \item $\left( \LpG_{\Ran(X)} , \Inc_\Ran \right)^\an_{\PSh\Str} = \underset{I\in \fs}{\colim} \left( \LpG_{X^I},\Inc_I \right)^\an_{\Str}$;
    \item $\left( \Gr_{G,\Ran(X)}, \mathfrak{s}_{\Ran}\right)^\an_{\PSh\Str, \lft} = \colim_{I\in \fs^\op}^{\Str\Top} \left(\Gr_{G,X^I}, \mathfrak s_I \right)^\an_{\PSh\Str,\lft}$, in $\Str\Top$ over \\ $(\Ran(X), \Inc_\Ran)^{\an}_{\PSh\Str, \lft}$.
\end{enumerate}
\end{cor}
			
\begin{proof}
    Statements $(1)$ and $(3)$ follow from the fact that $(-)^\an_{\PSh\Str, \lft}$ preserves small colimits and 
    Statement $(2)$ from the fact that $(-)^\an_{\PSh\Str}$ preserves small colimits of stratified schemes.
\end{proof}

\begin{warning}\label{warning-Ran}
The same issue noticed in \Cref{warning} (namely, the failure of universality of colimits) applies here as well.
In addition, the group presheaf $\LpG_{\Ran(X)}$ is realized as a colimit in $\PSh(\Str\Sch_\C)$ and the analytification functor $(-)^\an_{\PSh\Str}$ does not preserve finite limits in general (it is not even Cartesian lax-monoidal): therefore, already the reconstruction of a (relative) group structure for $(\LpG_{\Ran(X)}, \Inc_\Ran)^\an_{\PSh\Str}$ is less straightforward than the one for $(\LpG_{X^I}, \Inc_I)^\an_\Str$.
\end{warning}

The rest of the present subsection addresses the realization problem explained in \cref{warning-Ran}. 
We want to remark that, among our main results, everything in \cref{Section-invariance} up to \cref{isotopies-are-equivariant} (included) is not influenced by this discussion. 
On the other hand, the results from \cref{equivariance-theorem-Ran} until the end of \cref{Section-invariance} do depend on it, and specifically on \cref{extension-of-analytification}.

\begin{defin}[{\cite[§3.7]{Engelking}}]\label{perfect-map}
A \textit{perfect map} is a closed continuous map of topological spaces $X\to Y$ where $X$ is Hausdorff and all fibers are compact.
\end{defin}

\noindent Notice that a perfect surjection is in particular a closed surjection, and hence topological quotient. This motivates the following definition.

\begin{defin}\label{perfect-quotients}
A \textit{perfect quotient} is a perfect surjective map $f:X\to Y$. 
\end{defin}

\begin{recall}\label{recall-perfect-quotients-Hausdorff-locc}
    Let $f:X\to Y$ be a perfect quotient. 
    If $X$ is Hausdorff, so is $Y$, by \cite[Theorem 3.7.20]{Engelking}.
    If $X$ is locally compact, so is $Y$, by \cite[Theorem 3.7.21]{Engelking}.
    
    It is also easy to see that $f$ is universally closed, i.e. for any map $Z\to Y$ the map $f_Z:Z\times_YX\to Z$ obtained by pullback is closed (and surjective). If moreover $Z$ is Hausdorff, one can prove that $f_Z$ is again a perfect quotient. 
\end{recall}

\begin{lem}\label{properties-of-perfect-quotients}
    Let 
    \[\begin{tikzcd}
	{X'} & {Z'} & {Y'} \\
	X & Z & Y
	\arrow[from=1-1, to=1-2, "a'"]
	\arrow[from=1-1, to=2-1, "f"]
	\arrow[from=1-2, to=2-2, "h"]
	\arrow[from=1-3, to=1-2, "b'"']
	\arrow[from=1-3, to=2-3, "g"]
	\arrow[from=2-1, to=2-2, "a"]
	\arrow[from=2-3, to=2-2, "b"']
\end{tikzcd}\]
    be a commutative diagram in $\Top$, where $f,g, h$ are perfect quotients. 
    Then the induced map $f\times_hg:X'\times_{Z'}Y'\to X\times_ZY$ is again a perfect quotient. 
\end{lem}

\begin{proof}
    Perfect maps are stable under products by \cite[Theorem 3.7.7]{Engelking}, and so are surjections. 
    Since a finite limit of topological spaces is always a closed subspace of a product when the spaces involved are Hausdorff, one can deduce the statement from \cite[Proposition 3.7.4]{Engelking} and from the fact that surjections are stable under pullback. 
\end{proof}

\begin{cons}\label{construction-Rann}
Let $n\geq 1$. 
Define
$$
(\Ran_{\leq n}(X), \Inc_{\leq n}) = 
\colim_{I\in\fs^\op, |I|\leq n}
(X^I, \Inc_I)
$$
in $\PSh(\Str\Sch^\lft_\C)$.
The stratifying poset of $\Inc_{\leq n}$ is isomorphic to the totally ordered set of natural numbers less or equal than $n$.
Explicitly, $\Ran_{\leq n}(X)(\C)$ is the set of $k$ unordered and distinct $X$-points $k\leq n$.

Note that $(\Ran(X), \Inc_\Ran)$ coincides with $\underset{n\geq 1}{\colim} \ (\Ran_{\leq n}(X), \Inc_{\leq n})$.
Similarly, let us set 
$$
(\Gr^{(N)}_{G,\Ran_{\leq n}(X)}, \mathfrak{s}_{{\leq n}}^{(N)}) \coloneqq \colim_{I\in\fs^\op, |I|\leq n} (\Gr^{(N)}_{G,X^I}, \mathfrak{s}_I^{(N)}), 
$$ and $$
(\Gr_{G,\Ran_{\leq n}(X)}, \mathfrak{s}_{{\leq n}}) \coloneqq \colim_{I\in\fs^\op, |I|\leq n} (\Gr_{G,X^I}, \mathfrak{s}_I)
$$ 
in $\PSh(\Str\Sch^\lft_\C)$. Note that both of them have a natural stratified map to $(\Ran_{\leq n}(X),\Inc_{\leq n})$. 
Then $(\Gr_{G,\Ran(X)}, \mathfrak{s}_\Ran)$ coincides with 
$$ 
\underset{n\geq 1}{\colim} \ (\Gr_{G, \Ran_{\leq n}(X)}, \mathfrak{s}_{{\leq n}})
=
\underset{n\geq 1, N\in \nn}{\colim} \ (\Gr^{(N)}_{G, \Ran_{\leq n}(X)}, \mathfrak{s}_{{\leq n}}^{(N)}).
$$

Note also, for later use, that by universality of colimits we have 
\begin{equation}\label{pullback-GrRann}
\begin{gathered}
\Ran_{\leq n}(X)\times_{\Ran(X)} \Gr_{\Ran(X)}
\simeq 
\colim_{I\in\fs^\op,|I|\leq n} \left( X^I\times_{\Ran(X)}\Gr_\Ran(X) \right)
\simeq
\\ 
\colim_{I\in\fs^\op,|I|\leq n}\Gr_{G, X^I}
=
\Gr_{\Ran_{\leq n}(X)}
\end{gathered}
\end{equation} 
and the analogous isomorphism holds if we add stratifications.
Finally, we can do the same for the arc group, and define in $\PSh(\Str\Sch_\C)$
\[
(\LpG_{\Ran_{\leq n}(X)}, \Inc_{{\leq n}}) 
\coloneqq \colim_{I\in\fs^\op,|I|\leq n} (\LpG_{X^I}, \Inc_I),
\]
so that $(\LpG_{\Ran(X)}, \Inc_\Ran)$ can be written as 
\begin{equation}\label{new-decomposition-LpG}
\colim_{n\geq 1}\  (\LpG_{\Ran_{\leq n}(X)}, \Inc_{{\leq n}}).
\end{equation}
\end{cons}

\begin{rem}
By \cite[Lemma 2.5]{Handel} the map
\begin{equation}\label{eq:perfectquotXn}
    \cU^\an_n:(X^n, \Inc_n)^\an_{\Str, \lft} \to (\Ran_{\leq n}(X), \Inc_{\leq n})^\an_{\PSh\Str,\lft}
\end{equation}
is a closed quotient at the level of underlying topological spaces. 
Because $X^{n, \an}$ is Hausdorff and the fibers of \eqref{eq:perfectquotXn} are finite nonempty (hence compact), the underlying topological map of \eqref{eq:perfectquotXn} is a perfect quotient.
\end{rem}

Note that by \eqref{pullback-GrRann} the diagram
\begin{equation}\label{eq:diagramrelatedtoperfectquoz}
\begin{tikzcd}
\left(\Gr_{G,X^n}^{(N)}, \mathfrak{s}_n^{(N)} \right)
\arrow[d, "{p_{n}^{(N)}}"]
\arrow[r, "{\widetilde{\mathcal{U}}_n^{(N)}}"]
&
\left(\Gr_{G,\Ran_{\leq n}(X)}^{(N)}, \mathfrak{s}_{\leq n}^{(N)}\right)
\arrow[d, "{p_{\Ran_{\leq n}}^{(N)}}"]
\\
\left(X^n,\Inc_n\right)
\arrow[r, "{\mathcal{U}_n}"]
&
\left(\Ran_{\leq n}(X), \Inc_{\leq n} \right)
\end{tikzcd}
\end{equation}
is cartesian.

\begin{lem}\label{lem:widetildecUisperfquoz}
    The analytification via $(-)^\an_{\PSh\Str, \lft}$ of the map $\widetilde{\cU}^{(N)}_n$ 
    is a perfect quotient. Similarly, 
    the analytification via $(-)^\an_{\PSh\Str}$ of \[
    \widetilde{\cU}_n^+: \left(\LpG_{X^n}, \Inc_n \right) \to \left(\LpG_{\Ran_{\leq n}(X)}, \Inc_{{\leq n}}\right).
    \]
    is a perfect quotient.
\end{lem}

\begin{proof}
Let us first show that $\widetilde{\mathcal{U}}_n^{(N), \an}$ is closed.
Let $A\subseteq \Gr_{G, X^n}^{(N), \an}$ be a closed subset. 
By definition of the colimit topology, $\widetilde\cU^{(N), \an}_n(A)$ is closed in $ \underset{I\in \fs, |I|\leq n}{\colim} \Gr^{(N),\an}_{G, X^I}$ 
if and only if $(\widetilde\cU^{(N), \an}_I)^{-1}( \widetilde\cU_n^{(N), \an}(A) )$ is closed in $\Gr_{G, X^I}^{(N), \an}$ for any $I\in \fs, |I|\leq n$. 
For any $\{1,\dots,n\}\overset{\phi}{\twoheadrightarrow} I$ we have $\widetilde\cU^{(N), \an}_I = \widetilde\cU^{(N), \an}_n\circ \delta_\phi$, 
and hence it is enough to check that $(\widetilde\cU^{(N), \an}_I)^{-1}\circ \widetilde\cU^{(N), \an}_n(A)$ is closed for $I=\{1,\dots,n\}$.  
This is done by induction on $n$ as in the proof of \cite[Lemma 2.5]{Handel}.
Note that $\widetilde\cU^{(N)}_n(\C)$ has finite nonempty fibers (for instance, this follows easily by taking complex points in \eqref{eq:diagramrelatedtoperfectquoz} and using \cref{underlying-set-an} to argue that $\widetilde\cU^{(N)}_n(\C)$ is a pullback of $\cU_n(\C)$). 
Therefore $\widetilde\cU^{(N), \an}_n$ has finite nonempty fibers as well. 
Because $\Gr^{(N), \an}_{G, X^n}$ is Hausdorff, $\widetilde\cU^{(N), \an}_n$ is a perfect quotient.
An analogous proof shows the statement for $\widetilde{\cU}^+_{n}$ (recall that $\LpG_{X^n}^\an$ is Hausdorff because limit of analytifications of quasi-projective complex schemes).
\end{proof}

The following result is not necessary for the upcoming proofs but we think it is still worth mentioning.
\begin{lem}
The diagram~\ref{eq:diagramrelatedtoperfectquoz} stays cartesian after applying $(-)^\an_{\PSh\Str, \lft}$.    
\end{lem}

\begin{proof}
By \Cref{underlying-set-an}, it does after applying $\Fgt_{\str, \textup{top}}\circ (-)^\an_{\PSh\Str, \lft}$: so 
\begin{equation}\label{underlying-sets-square}
\begin{tikzcd}[column sep=2cm]
{\Gr^{(N)}_{G, X^n}(\C)} 
\arrow[r, "{\tilde{\mathcal U}_n^{(N)}(\C)}"]
\arrow[d, "{p_{n}^{(N)}(\C)}"]
& 
{\Gr^{(N)}_{G, \Ran_{\leq n}(X)}(\C)} 
\arrow[d, "{p_{\Ran_{\leq n}}^{(N)}(\C)}"]
\\
{X^n(\C)} 
\arrow[r, "{\mathcal U_n(\C)}"]
& 
{\Ran_{\leq n}(X)(\C)}
\end{tikzcd}
\end{equation}
is cartesian in $\sets$. 
To show that it was already cartesian in $\Top$ (so before forgetting the topology), it suffices to prove that ${\Gr^{(N)}_{G, X^n}(\C)}$ was endowed with the fiber product topology.
Namely that a subset $A$ of $\Gr^{(N)}_{G, X^n}(\C)$ is closed if and only if $p_n^{(N), \an}(A)$ and $\widetilde \cU_n^{(N),\an}(A)$ are both closed.
This is true because:
\begin{itemize}
    \item $p_n^{(N),\an}$ is a proper map by \cite[Remark 3.1.4]{Zhu}, hence its analytification is a closed map;
    \item $\widetilde \cU_n^{(N),\an}(A)$ is closed by \Cref{lem:widetildecUisperfquoz}.
\end{itemize}
At the level of the stratifying posets, the diagram is
\[
\begin{tikzcd}
\mathfrak{s}^{(N), \an}_n
\arrow[r, hook]
\arrow[d]
&
\colim_{|I|\leq n} \mathfrak{s}^{(N), \an}_I
\arrow[d]
\\
\Inc_n
\arrow[r, hook, dashed]
&
\colim_{|I|\leq n} \Inc_I.
\end{tikzcd}
\]
This is cartesian by \Cref{eq:diagramrelatedtoperfectquoz}.
\end{proof}

On the other hand, the next result will play a crucial role in the proof of Theorem~\ref{extension-of-analytification}.

\begin{lem}\label{remark-Rann-Hausdorff}
The topological spaces underlying the analytifications $\left(\Ran_{\leq n}(X), \Inc_{{\leq n}}\right)^\an_{\PSh\Str, \lft}$, $\left(\Gr^{(N)}_{G,\Ran_{\leq n}(X)},  \mathfrak{s}_{\leq n}^{(N)}\right)^\an_{\PSh\Str, \lft}$ and $\left( \LpG_{\Ran_{\leq n}(X)}, \Inc_{\leq n} \right)^\an_{\PSh\Str}$ are locally compact Hausdorff spaces.
\end{lem}

\begin{proof}
We noticed and used already that $X^{n, \an}, {\Gr^{(N),\an}_{G,X^n}}$ and $(\LpG_{X^n}, \Inc_n)^\an_{\Str\Sch}$ are locally compact Hausdorff spaces, because they are (limits of) analytifications of quasi-projective complex schemes.
Since $\cU^\an_n, \widetilde\cU^{(N),\an}_n$ and $\widetilde\cU^{+,\an}_n$ are perfect quotients by \Cref{lem:widetildecUisperfquoz}, we can conclude applying \cref{recall-perfect-quotients-Hausdorff-locc}.
\end{proof}

Now we are ready to recover the relative group structure of $(\LpG_{\Ran(X)}, \Inc_\Ran)^\an_{\PSh\Str}$ over $(\Ran(X),\Inc_\Ran)^\an_{\PSh\Str, \lft}$ and its action on $(\Gr_{\Ran(X)}, \mathfrak s_{\Ran})^\an_{\PSh\Str,\lft}$
                
\begin{prop}\label{extension-of-analytification}
The analytification procedure yields an object 
$$
(\LpG_{\Ran(X)}, \Inc_\Ran)^\an_{\PSh\Str} \in \Grp(\Str\Top/_{(\Ran(X), \Inc_{\Ran})^\an_{\PSh\Str,\lft}})
$$ 
together with a stratified action on $(\Gr_{G,\Ran(X)}, \mathfrak{s}_\Ran)^\an_{\PSh\Str,\lft}$ over $( \Ran(X), \Inc_{\Ran} )^\an_{\PSh\Str,\lft}$:
$$
(\LpG_{\Ran(X)}, \Inc_\Ran)^\an_{\PSh\Str} \hspace{-0.1cm}
\underset{(\Ran(X), \Inc_\Ran)^\an_{\PSh\Str,\lft}}{\times} \hspace{-0.8cm}
(\Gr_{G,\Ran(X)}, \mathfrak{s}_\Ran)^\an_{\PSh\Str,\lft}
\to 
(\Gr_{G,\Ran(X)}, \mathfrak{s}_\Ran)^\an_{\PSh\Str,\lft}.
$$
\end{prop}

\begin{proof}
By \cref{remark-Rann-Hausdorff} we can apply \cref{stratified-Harpaz} to the fiber product 
\[
(\LpG_{\Ran(X)}, \Inc_\Ran)^\an_{\PSh\Str} 
\underset{(\Ran(X), \Inc_\Ran)^\an_{\PSh\Str,\lft}}{\times} 
(\LpG_{\Ran(X)}, \Inc_\Ran)^\an_{\PSh\Str} 
\]
and by \Cref{construction-Rann} we can rewrite it as
\begin{equation}\label{application-of-Harpaz}
\colim_{n\geq 1} 
\left( 
(\LpG_{\Ran_{\leq n}(X)} , \Inc_{\leq n})^\an_{\PSh\Str} 
\underset{{(\Ran_{\leq n}(X), \Inc_{\leq n})^\an_{\PSh\Str, \lft}}}{\times} 
(\LpG_{\Ran_{\leq n}(X)} , \Inc_{\leq n})^\an_{\PSh\Str} 
\right).
\end{equation}
 
Let $\mu^+_n$ be the multiplication on $(\LpG_{X^n},\Inc_n)$.
Consider the diagram 
$$
\begin{tikzcd}
(\LpG_{X^n}, \Inc_n)^\an_{\PSh\Str}
\underset{(X^{n, \an}, \Inc_n)}{\times} 
(\LpG_{X^n}, \Inc_n)^\an_{\PSh\Str}
\arrow[r, "{\mu^{+, \an}_n}"]
\arrow[d, "\widetilde{\cU}^{+,\an}_n \times_{\cU_n^\an}\, \widetilde{\cU}^{+,\an}_n"]
&
(\LpG_{X^n}, \Inc_n)^\an_{\PSh\Str}
\arrow[d, "\widetilde{\cU}^{+,\an}_n"]
\\
\hspace{-0.2cm} (\LpG_{\Ran_{\leq n}(X)}, \Inc_{\leq n})^\an_{\PSh\Str} 
\underset{\Ran^\an_{\leq n}}{\times} 
(\LpG_{\Ran_{\leq n}(X)}, \Inc_{\leq n})^\an_{\PSh\Str}
&
(\LpG_{\Ran_{\leq n}(X)}, \Inc_{\leq n})^\an_{\PSh\Str},
\end{tikzcd}
$$
where $\Ran^\an_{\leq n}$ denotes $(\Ran_{\leq n}(X),\Inc_{\leq n})^\an_{\PSh\Str, \lft}$.
By \cref{properties-of-perfect-quotients} the map $\widetilde{\cU}^{+, \an}_n \times_{\cU_n^\an} \widetilde{\cU}^{+, \an}_n$ is a perfect quotient. 
Therefore there is an arrow $\mu^{+, \an}_{\leq n}$ completing the above diagram into a commutative square in $\Top$.
Since at the level of the posets the diagram is 
$$
\begin{tikzcd}
\Inc_n
\arrow[r, equal]
\arrow[d]
&
\Inc_n
\arrow[d]
\\
\colim_{|I|\leq n} \Inc_I
\arrow[r, equal, dashed]
&
\colim_{|I|\leq n} \Inc_I.
\end{tikzcd}
$$
the map $\mu^{+, \an}_{\leq n}$ is also stratified.  

By \cref{stratified-Harpaz}, taking the colimit in $n$, one recovers a well-defined continuous stratified group law on $(\LpG_{\Ran(X)},\Inc_\Ran)^\an_{\PSh\Str}$ over $(\Ran(X), \Inc_\Ran)^\an_{\PSh\Str, \lft}$.
Analogously, if we apply \cref{stratified-Harpaz} twice, the fiber product 
\[
(\LpG_{\Ran(X)}, \Inc_\Ran)^\an_{\PSh\Str}
\underset{(\Ran(X), \Inc_\Ran)^\an_{\PSh\Str,\lft}}{\times} 
(\Gr_{G,\Ran(X)}, \mathfrak{s}_\Ran)^\an_{\PSh\Str,\lft}
\]
is isomorphic to
\[
\colim_{n\geq 1, N\geq 0} \left( (\LpG_{\Ran_{\leq n}(X)}, \Inc_{\leq n})^\an_{\PSh\Str}
\underset{(\Ran_{\leq n}(X), \Inc_{\leq n})^\an_{\PSh\Str,\lft}}{\times} 
(\Gr_{G,\Ran_{\leq n}(X)}^{(N)}, \mathfrak{s}_{\leq n}^{(N)})^\an_{\PSh\Str,\lft}
\right).
\]
And in the same way, using that  $\widetilde{\cU}^{+, \an}_n \times_{\cU^\an_n}\widetilde{\cU}^{(N), \an}_n$ is a perfect quotient, 
one recovers a continuous stratified group action of $(\LpG_{\Ran(X)},\Inc_\Ran)^\an_{\PSh\Str}$ on $(\Gr_{G,\Ran(X)}, \mathfrak{s}_\Ran)^\an_{\PSh\Str,\lft}$ over $( \Ran(X), \Inc_{\Ran} )^\an_{\PSh\Str,\lft}$. 
\end{proof}

\section{Isotopy invariance}\label{Section-invariance}
\subsection{Lifting isotopies}

Most of the proof of the main result of the paper, \Cref{homotopy-equivalence-D-D'}, is based on the following three lemmas.

\begin{lem}\label{remark-lift-automorphism} 
    Let $R$ be a $\C$-algebra locally of finite type. Any $R$-linear automorphism $f:X_R\to X_R$ induces an automorphism of ind-$R$-schemes $\Phi_{f}: (\Gr_{G,X^I})_R\to (\Gr_{G,X^I})_R$. The map $f\mapsto \Phi_f$ is natural in $R$. 
    So it defines a morphism of presheaves $\PSh(\Aff^\lft_\C)$ 
    \begin{equation}\label{morphism-between-Aut}
        \Phi: \uAut_\C(X)\to \uAut_\C(\Gr_{G,X^I}),\quad f\mapsto \Phi_f.
    \end{equation} 
\end{lem}

\begin{proof}
Let $A$ be an $R$-algebra locally of finite type, $\tau: \Spec A \rightarrow \Spec R$. Denote by $f_A$ the base change of $f$ to $X_A$ by $\tau$. 
If $y_I$ is the composition 
$$
\Spec A
\xhookrightarrow{x_I \times \id_A }X^I_A\xrightarrow{(f_A^{-1})^I} X_A^I \xrightarrow{\textup{pr}_{X^I}} X^I,
$$
define
$$
\Phi_{f,A}: (\Gr_{G,X^I})_R (A) \to 
(\Gr_{G,X^I})_R (A), 
\quad
( x_I, \mathcal{F},\alpha, \tau ) \mapsto (y_I,f^*_A\cF,f^*_A\alpha,\tau).
$$
This is well-defined because $f_A^*(\mathcal{F}|_{X_A\setminus \Gamma_{x_I}})\simeq (f_A^*\mathcal{F})|_{X_A\setminus \Gamma_{y_I}}$ and $f^*_A\mathcal{T}_{G,X_A}\simeq\mathcal{T}_{G,X_A}$.
Since the formation of $f_A$ is natural in $A$, so is $\Phi_{f, A}$.

\end{proof}

\begin{lem}\label{lemma-stratified-automorphism}
Let $R$ be a $\C$-algebra locally of finite type.
Let $f$ be an automorphism of $X_R$.
For any $N$, the automorphism $\Phi_f$ induces
an automorphism of $\left( (\Gr_{G,X^I}^{(N)})_R, (\mathfrak{s}_I)_R \right)$.
In particular $\Phi_f$ upgrades to an automorphism  of $\left( (\Gr_{G,X^I})_R, (\mathfrak{s}_I)_R \right)$ in $\PSmall( \Str\Sch_R^\lft /_{(X_R^I, (\Inc_I)_R)} )$.
\end{lem}

\begin{proof}
For the sake of notation, we write the proof for $R=\C$. The general case is analogous.
The map $f^I:X^I\to X^I$ respects the incidence stratification on $X^I$,  so $\Phi_f$ restricts to $\left ( \prod_{j=1}^{|J|}\Gr_{G,X} \right )_\disj$. 
Moreover, since pulling back along an automorphism commutes with the operation of gluing torsors, for any $X^\phi\subset X^I$, $\Phi_f$ commutes with the factorization isomorphism \eqref{factorization-property}
\[
\mathfrak{f}_\phi : \left ( \prod_{j=1}^{|J|} \Gr_{G,X} \right )_{\disj} \xrightarrow{\sim} \Gr_{G, X^\phi}.
\]
Since this isomorphism restricts to each $N$-level, it is enough to check the statement for $\Gr_{G, X}^{(N)}$.

Let us then consider the stratum $\Gr_{G,X,\nu}$, $\nu\leq N$, together with the isomorphism 
\[
\mathfrak{bl}  : \widehat{X}\times^{\uAut_\C \C \taylor } \Gr_{G,\nu} \xrightarrow{\sim} \Gr_{G,X,\nu}
\]
defined in \Cref{characterization-X-hat}.
Let $A$ be a $\C$-algebra and pick an $A$-point 
$$
[(x,\eta,\widetilde{\mathcal{F}},\widetilde{\alpha})] \in \left( (\widehat{X}\times \Gr_{G,\nu} )/ \uAut_\C\C \taylor \right) (A).
$$
Let $(x,\mathcal{F},\alpha)$ be the image of $[(x,\eta,\widetilde{\mathcal{F}},\widetilde{\alpha})]$ in $\Gr_{G,X,\nu}$, characterized by \eqref{theisomoorphism} as the pair $(\mathcal{F},\alpha)$ such that (see the notation in \Cref{formalcoordandtorsorrecall})
$$
\eta^*i_{\widehat{x}}^*\mathcal{F} \cong \widetilde{\mathcal{F}}, \quad 
(i_{\widehat{x}}\circ \eta)|_{\Spec A\laurent}^*  \alpha \cong \widetilde \alpha.
$$ 
Now $\Phi_{f,A}(x,\mathcal{F},\alpha)=( f^{-1} x, f_A^*\mathcal{F},f_A^*\alpha )$.
In particular
\begin{equation}\label{isomorphismonpairs}
    (f_A^{-1}\circ i_{\widehat{x}}\circ \eta)^* f_A^*\mathcal{F} \cong \widetilde{\mathcal{F}}, \quad 
     (f_A^{-1}\circ i_{\widehat{x}}\circ \eta)|_{\Spec A\laurent}^*  f_A^*|_{\Spec A\laurent} \alpha \cong \widetilde \alpha.
\end{equation}
Using the cartesian diagram
\begin{equation}\label{diagramforeta}
\begin{tikzcd}
\Spec A\taylor
\arrow[d, equal]
\arrow[r, "{\widehat{f}_{A, x}^{-1}\circ \eta}"]
&
\underline{\Spec}_{X_A}( \widehat{ \cO}_{\Gamma_{f^{-1} x} })
\arrow[d, "{\widehat{f}_{A, x}}"]
\arrow[r,"{i_{\widehat{f^{-1} x}}}"]
& 
X_A
\arrow[d,"{f_A}"]
\\
\Spec A\taylor
\arrow[r,"\eta"]
&
\underline{\Spec}_{X_A}( \widehat{ \cO}_{\Gamma_x})
\arrow[r,"{i_{\widehat{x}}}"]
& 
X_A,
\end{tikzcd}
\end{equation}
equalities in \eqref{isomorphismonpairs} can be rewritten as
$$
(\widehat{f}_{A,x}^{-1}\circ\eta)^* i_{\widehat{f^{-1}x}}^* ( f_A^*\mathcal{F}) \cong \widetilde{\mathcal{F}} , \quad
(\widehat{f}_{A,x}^{-1}\circ\eta|_{\Spec R\laurent} )^* i_{\widehat{f^{-1}x}}^*
(f_A^*\alpha ) \cong \widetilde \alpha,
$$
which means that $\Phi_f$ acts on $\widehat{X} \times \Gr_{G,\nu} / \uAut_\C \C\taylor$ sending
\begin{equation}\label{mapPhifatloclevel}
[(x,\eta, \widetilde{\mathcal{F}}, \widetilde{\alpha})] \mapsto 
[( \textup{pr}_{X}f_A^{-1}x, \widehat{f}_{A,x}^{-1}\circ\eta, \widetilde{\mathcal{F}}, \widetilde \alpha  )]
\end{equation}
Passing to the sheafification, this implies that $\Phi_f$ only modifies the first component of $\widehat{X}\times^{\uAut_\C \C \taylor } \Gr_{G,\nu}$ and therefore preserves the stratification.
\end{proof}

\begin{defin}\label{defin-algebraic-isotopy}
Let $Y$ be an object in $\PSh(\Aff^\lft_\C)$.
An \emph{algebraic isotopy} of $Y$ is a morphism in $\PSh(\Aff^\lft_\C)$
$$
    F:U \to \uAut_{\C}(Y),
$$
where $U$ is an open of $\Ac$ such that $[0,1]\subset U^{\an}$.
\end{defin}

\begin{rem}\label{lift-isotopy}  
Given an algebraic isotopy of $X$, by \Cref{remark-lift-automorphism} we get an algebraic isotopy
$$
    \Phi\circ F: U \to \uAut_\C(\Gr_{G,X^I}).
$$
Let us consider $U$ as a stratified scheme with the trivial stratification. 
Composing with the evaluation
$$
    \textup{ev}: \uAut_{\C}(\Gr_{G,X^I})\times_{\C}  \Gr_{G,X^I} \to \Gr_{G,X^I}, \quad \quad
    (f,x)\mapsto f(x)
$$
we get a map of ind-$\C$-schemes
\[
    \textup{ev}\, \circ  \, (\Phi\circ F, \id ) : U \times_{\C} \Gr_{G,X^I}\xrightarrow{} \Gr_{G,X^I}.
\] 
By \Cref{lemma-stratified-automorphism}, this map is actually stratified, giving a map in $\PSmall(\Str\Sch^\lft_\C/_{(X_I,\Inc_I)})$ 
\begin{equation}\label{A1-isotopy}
\textup{ev}\, \circ  \, (\Phi\circ F, \id ) : \left(U \times_{\C} \Gr_{G,X^I}, \textup{triv}\times \mathfrak{s}_I\right) \xrightarrow{} \left( \Gr_{G,X^I}, \mathfrak{s}_I\right).
\end{equation}
Let us take the analytification $(-)^\an_{\PSh\Str}$ of \eqref{A1-isotopy}
\begin{equation}\label{isotopy} 
\Psi_U \coloneqq \textup{ev}\, \circ  \, ( (\Phi\circ F)^\an, \id ): \left( U^\an \times \Gr_{G,X^I}^\an, \textup{triv}\times \mathfrak{s}_I^\an \right)\to \left(\Gr_{G,X^I}^\an, \mathfrak{s}_I^\an \right).
\end{equation} 
Therefore, for every $t\in U^\an$, the map $\Psi_U(t,-)$ is equal to $\Phi_{F(t)}^\an(-)$.
Note that a priori $(-)^\an_{\PSh\Str}$ does not preserve fiber products: however, since $U^\an$ is Hausdorff and locally compact we can apply \cref{stratified-Harpaz} because we are first taking the analytification at the $\Gr_{G,X^I}^{(N)}$-level and then taking the $N$-colimit.

Restricting \eqref{isotopy} to $[0,1]$, we get a stratified map  
    \begin{equation}\label{I-isotopy}
    \Psi_{[0,1]} = \Psi_U|_{[0,1]}:\left( [0,1]\times \Gr_{G,X^I}^\an, \textup{triv} \times \mathfrak{s}_I^\an\right)\to \left( \Gr_{G,X^I}^\an, \mathfrak{s}_I^\an\right).
    \end{equation}
\end{rem}

\begin{defin}Let $f,g:(Y,s_Y)\to (W,s_W)$ be two maps of stratified topological spaces. 
Let $\textup{triv}\times s_Y$ be the stratification of $[0,1]\times Y$ induced by the projection $[0,1]\times Y \to Y$ (and hence trivial in the first component). 
A \textit{stratified homotopy} between $f$ and $g$ is a stratified map 
$$
H:([0,1]\times Y, \textup{triv}\times s_{Y})\to (W,s_W)
$$ 
such that $H(0,-)=f, H(1,-)=g$.
It is said to be a \emph{stratified isotopy} if $H(t,0)$ is a closed embedding for any $t\in [0,1]$.
\end{defin}

Note that the morphism \eqref{I-isotopy} is a stratified isotopy.

\begin{defin}\label{defin-stratified-homotopy-equivalence}
A \textit{stratified homotopy equivalence} of stratified topological spaces is then a stratified map $f:(Y,s_Y)\to (W,s_W)$ such that there exist a stratified map $g:(W,s_W)\to (Y,s_Y)$ and stratified homotopies $gf\sim \id_{(Y,s_Y)}, fg\sim \id_{(W,s_W)}$.
\end{defin}

\begin{lem}\label{prop-alg-isotopy}
Consider two opens $D' \overset{i}{\subset} D \subset X^{\an}$.
If there exists an algebraic isotopy $F:U\to \uAut_\C(X)$ such that
    \begin{enumerate}
        \item for every $t\in [0,1] \subset U^\an$ we have $F_t^\an(D')\subset D'$ and $F_t^\an(D)\subset D$,
        \item $F^\an_0|_D=\id_D$ and $F_1^\an(D)=D'$,
    \end{enumerate}
then the open inclusions 
    \begin{gather*}
    i_I^{(N)}:\left( \Gr_{G,{D'}^I}^{(N)}, \mathfrak{s}_I^\an \right) \hookrightarrow \left( \Gr_{G,D^I}^{(N)}, \mathfrak{s}_I^\an \right),\quad \textup{and}\quad 
    i_I:\left( \Gr_{G,{D'}^I} , \mathfrak{s}_I^\an \right) \hookrightarrow \left( \Gr_{G,D^I}, \mathfrak{s}_I^\an \right),
    \end{gather*}
are stratified homotopy equivalences and the homotopies involved can be taken to be isotopies.
\end{lem}
\begin{proof}
    Consider the stratified map $\Psi_{[0,1]}$ from \eqref{I-isotopy}. 
    By condition 1, for any $t\in [0,1]$ the image of $\Psi_{[0,1]}(t,-)|_{\Gr_{G, D^I}}$ lies all in $\Gr_{G,D^I}$.
    Moreover, condition 2 reads as
    $$
    \Psi_{[0,1]}(0,-)|_{\Gr_{G, D^I}} = \id_{\Gr_{G,D^I}}, 
    \quad \textup{and} \quad
    \textup{Im}\left(\Psi_{[0,1]}(1,-)|_{\Gr_{G,{D}^I}}\right) \subset \Gr_{G,{D'}^I} \xhookrightarrow{i_I} \Gr_{G,D^I}.
    $$
    Therefore, the map $\Psi_{[0,1]}|_{\Gr_{G,D^I}}$
    gives a stratified isotopy between $\id_{\Gr_{G,D^I}}$ and $i_I \circ \Psi_{[0,1]}(1,-)|_{\Gr_{G,D^I}}$.   

    Consider now $\Psi_{[0,1]}(1,-)|_{\Gr_{G,D^I}} \circ i_I$ which is the same as $\Psi_{[0,1]}(1,-)|_{\Gr_{G,{D'}^I}}$.
    Again by condition 1, for any $z$ the image of $\Psi_{[0,1]}(t,-)|_{\Gr_{G,{D'}^I}}$ is all contained in $\Gr_{G,{D'}^I}$. 
    Then 
    $$
    \Psi_{[0,1]}|_{\Gr_{G,{D'}^I}}: \left( [0,1]\times \Gr_{G,{D'}^I}, \textup{triv}\times \mathfrak{s}_I^\an\right) \to \left(\Gr_{G,{D'}^I}, \mathfrak{s}_I^\an\right)
    $$ 
    gives a stratified isotopy between $\id_{\Gr_{G,{D'}^I}}$ and $\Psi_{[0,1]}(1,-)|_{\Gr_{G,D^I}} \circ i_I$.

    Therefore $\Psi_{[0,1]}(1,-)|_{\Gr_{G, D^I}} : \Gr_{G,D^I}\to \Gr_{G,{D'}^I}$ is a stratified homotopy inverse of the inclusion $i_I:\Gr_{G,{D'}^I}\hookrightarrow \Gr_{G,D^I}$. 

    The proof for $i_I^{(N)}$ is analogous (thanks to \Cref{lemma-stratified-automorphism}).
\end{proof}

\begin{thm}\label{homotopy-equivalence-D-D'} 
Let $z_0,z_0'\in \C$, and $ r>r'\in \R_{>0}$ such that $B(z_0',r')\subset B(z_0,r)\subset \C$.
Denote by $D'$ the ball $B(z_0',r')$, and by $D$ the ball $B(z_0,r)$.
The induced open embeddings 
\begin{gather*}
    i_I^{(N)}:\left( \Gr_{G,{D'}^I}^{(N)}, \mathfrak{s}_I^\an \right) \hookrightarrow \left( \Gr_{G,D^I}^{(N)}, \mathfrak{s}_I^\an \right),\quad  
    i_I:\left( \Gr_{G,{D'}^I} , \mathfrak{s}_I^\an \right) \hookrightarrow \left( \Gr_{G,D^I}, \mathfrak{s}_I^\an \right),
\end{gather*}
are stratified homotopy equivalences, and the homotopies involved can be taken to be isotopies.
\end{thm}

\begin{proof}
Consider the map
$$
    F: \Ac \to  \uEnd_\C(\C[z]) \cong \uEnd_{\C}(\Ac)
$$ 
defined at the level of $R$-points as 
$$
t\in R \mapsto F_t, \textup{ where } F_t(z) = 
z\Big(\frac{r'}{r}t + (1-t) \Big)+t\Big(z_0'-\frac{r'}{r}z_0\Big).
$$
Note that $F_t^*$ is an automorphism of $\mathbb{A}^1_R$ if and only if the scaling factor $\lambda(t)= \frac{r'}{r}t + (1-t)$ is in $R^{\times}=\mathbb{G}_{m,\C}(R)$. 
This happens if and only if $\lambda(t)$ belongs to the open $U\subseteq \Ac$ obtained as the fiber product
$$
    \begin{tikzcd}
        U
        \arrow[r]
        \arrow[d, hook]
        &
        \mathbb{G}_{m, \C}
        \arrow[d, hook]
        \\
        \Ac
        \arrow[r, "{\lambda}"]
        &
        \Ac.
    \end{tikzcd}
$$
If $t\in \C$, then $\lambda(t)\notin \C^{\times}$ if and only if $t=\frac{r}{r-r'}$: since $r>r'$, then  $[0,1]\subset U^{\an}$. 
Then $F|_U$ is an algebraic isotopy in the sense of \Cref{defin-algebraic-isotopy} and it satisfies the hypotheses of \Cref{prop-alg-isotopy}.
\end{proof}

\begin{rem}\label{lifting-for-Ran}
The map $$\uAut_\C(X)\to \uAut_\C(\Gr_{G, X^I})$$ in \Cref{remark-lift-automorphism} is natural in $I\in \fs^\op$. 
Therefore, it upgrades to a morphism of presheaves $$\Psi^\Ran:\uAut_\C(X)\to \uAut_\C(\Gr_{G,\Ran(X)}).$$
\end{rem}

\begin{cor}\label{homotopy-invariance-Ran}
Let $D'\subset D \subset \C$ be as in \Cref{homotopy-equivalence-D-D'}.
The induced open embedding 
$$
i_\Ran : \left(\Gr_{G,\Ran(D')}, \mathfrak{s}_\Ran^\an\right)\hookrightarrow   \left(\Gr_{G,\Ran(D)}, \mathfrak{s}_\Ran^\an\right)
$$
is a stratified homotopy equivalence, and the homotopies involved can be taken to be isotopies.
\end{cor}

\begin{proof}  
With \cref{lifting-for-Ran} at hand, one can argue as in \Cref{lift-isotopy}. That is, given any algebraic isotopy $U\to \uAut_\C(X)$, we obtain a stratified map 
$$
\left(
U^\an\times\Gr_{G,\Ran(X)}^\an, \textup{triv}\times\mathfrak{s}_\Ran^\an
\right)
\to 
\left(
\Gr_{G,\Ran(X)}^\an, \textup{triv}\times\mathfrak{s}_\Ran^\an\right)
$$ 
and hence a stratified isotopy 
$$
\Psi^\Ran_{[0,1]}:
\left(
[0,1]\times\Gr_{G,\Ran(X)}^\an, \textup{triv}\times\mathfrak{s}_\Ran^\an
\right)
\to
\left(
\Gr_{G,\Ran(X)}^\an, \mathfrak{s}_\Ran^\an\right).
$$
Let $X$ be $\mathbb{A}^1_\C$: the analogues of \Cref{prop-alg-isotopy} and \Cref{homotopy-equivalence-D-D'} for $\Psi^\Ran_{[0,1]}$ are deduced in the same way as above.
\end{proof}

\subsection{Equivariance}
\begin{rem}\label{lift-isotopy-for-groups}
Using the same notation as in \Cref{formalcoordandtorsorrecall}, by the same arguments as in \Cref{remark-lift-automorphism}), let us define the following morphisms of presheaves
\begin{gather*}
\forall m\in \nn,  \quad   
\Phi^{\LmG} : \uAut_\C(X)\to \uAut_\C(\LmG_{X^I}),
\quad
f\mapsto \Phi^{\LmG}_f : (x_I,g)\mapsto ((f^{-1})^I(x_I),f|_{\Gamma^m_{(f^{-1})^I(x_I)}}^*g),
\\
\Phi^{\LpG}:\uAut_\C(X)\to \uAut_\C(\LpG_{X^I}),
\quad
f\mapsto \Phi^{\LpG}_f: (x_I,g)\mapsto 
((f^{-1})^I(x_I), \widehat{f}_{x_I}^*g)
\\
\Phi^{\LpG_{\Ran(X)}}:\uAut_\C(X)\to \uAut_\C(\LpG_{\Ran(X)}),
\quad
f\mapsto \Phi^{\LpG}_f: (\underline x,g)\mapsto (f^{-1}(\underline x),\widehat{f}_{\underline x}^*g).
\end{gather*}
\end{rem}
Following the same steps of the proofs of \Cref{prop-alg-isotopy}, \Cref{homotopy-equivalence-D-D'} and \Cref{homotopy-invariance-Ran}, we have the following result as well.

\begin{prop}\label{homotopy-equivalence-D-D'-groups} 
Let $D'\subset D \subset \C$ be as in \Cref{homotopy-equivalence-D-D'}. Let $N\in \nn$ and $m\geq m_{N,I}$. 
Then the induced open embeddings 
\begin{gather*}
i^m_I:\left(\LmG_{D'^I}, \Inc_I\right) \hookrightarrow  \left(\LmG_{D^I}, \Inc_I\right), \quad 
i^+_I:\left(\LpG_{D'^I}, \Inc_I\right)\hookrightarrow  \left(\LpG_{D^I}, \Inc_I\right),
\\
i^+_{\Ran}: \left(\LpG_{\Ran(D')}, \Inc_\Ran\right)\hookrightarrow  \left(\LpG_{\Ran(D)}, \Inc_\Ran\right)
\end{gather*} 
are stratified homotopy equivalences, and the homotopies involved can be taken to be isotopies. 
\end{prop}

By their definition, the open embedding $i_I^+$ and $i_I$ fit in the commutative diagram
    $$
        \begin{tikzcd}
        \left( \LpG_{{D'}^I} \times_{{D'}^I} \Gr_{G, {D'}^I}, \mathfrak{s}_I\right)
        \arrow[r, hook, "{i^+_I\times i_I}"]
        \arrow[d, "\act_I"]
        &
        \left(\LpG_{D^I} \times_{D^I} \Gr_{G, D^I}, \mathfrak{s}_I\right)
        \arrow[d, "\act_I"]
        \\ 
        \left(\Gr_{G, {D'}^I}, \mathfrak{s}_I\right)
        \arrow[r, hook, "i_I"]
        &
        \left(\Gr_{G, D^I}, \mathfrak{s}_I\right)
        \end{tikzcd}
    $$
where the vertical maps are the action maps. 
Analogous versions for $\LmG$ and $\LpG_\Ran$ are true as well. 

Actually, furthermore, all the mentioned isotopies in \Cref{homotopy-equivalence-D-D'} and \Cref{homotopy-equivalence-D-D'-groups} are compatible with the above diagram, in the following sense.

\begin{thm}\label{isotopies-are-equivariant}
    Let $D'\subset D$ be metric disks in $\C$ and let $I\in\fs$.
    Let $i_I$ and $i_I^+$ be as in \Cref{homotopy-equivalence-D-D'} and \Cref{homotopy-equivalence-D-D'-groups} respectively.
    There exists a stratified map 
    \[
    \Psi^{\textup{equiv}}_{[0,1]}: \left( [0,1]\times  \LpG_{\C^I} \times_{\C^I} \Gr_{G, \C^I}, \textup{triv}\times \mathfrak{s}_I^\an \right)\to 
    \left( \LpG_{\C^I} \times_{\C^I} \Gr_{G, \C^I}, \mathfrak{s}_I^\an\right)
    \]
    such that 
    \begin{enumerate}
    \item for any $t\in [0,1]$, $\Psi^{\textup{equiv}}_{[0,1]}(t,-)$ is a closed embedding, and 
    \item makes the diagram \[
    \begin{tikzcd}[column sep=3cm]
    \left( [0,1]\times \LpG_{D^I}\times_{D^I}\Gr_{G, D^I}, \textup{triv}\times \mathfrak{s}_I^\an\right)
    \arrow[r, "{\Psi^{\textup{equiv}}_{[0,1]}|_{\Gr_{G, D^I}}}"]
    \arrow[d, "\id_{[0,1]}\times\act_I"]
    &
    \left( \LpG_{D^I}\times_{D^I}\Gr_{G,D^I}, \mathfrak{s}_I^\an\right)
    \arrow[d, "\act_I"]
    \\
    \left( [0,1]\times \Gr_{G,D^I}, \textup{triv}\times \mathfrak{s}_I^\an \right)
    \arrow[r, "{\Psi_{[0,1]}|_{\Gr_{G, D^I}}}"]
    &
    \left( \Gr_{G, D^I}, \mathfrak{s}_I^\an \right),
    \end{tikzcd}
    \]
    commute. 
    \end{enumerate}
    In particular, the morphisms $\Psi^{\textup{equiv}}_{[0,1]}|_{\Gr_{G, D^I}}$ and $\Psi^{\textup{equiv}}_{[0,1]}|_{\Gr_{G, {D'}^I}}$ show that $(i_I^+\times i_I)$ is a stratified homotopy equivalence (whose homotopies can be taked to be isotopies).

    An analogous statement holds for $\LmG_{\C^I}\times_{\C^I} \Gr^{(N)}_{G,\C^I}$ for any $N\in \nn$ and $m\geq m_{N,I}$.
\end{thm}

\begin{proof}
By their definitions, the automorphism $\Phi$ acts on the $X^I$-coordinate of $\Gr_{G, X^I}$ in the same way as $\Phi^{\LpG}$ acts on the $X^I$-coordinate of $\LpG_{X^I}$. Therefore they can be combined together to obtain 
\begin{gather*}
    \Phi^{\LpG}\times_{X^I} \Phi: \uAut_\C(X)\to \uAut_\C(\LpG_{X^I}\times_{X^I}\Gr_{G,X^I}).
\end{gather*}
Similarly, for any $N\in\nn, m\geq m_{N,I}$, we have
\begin{gather*}
    \Phi^{\LmG}\times_{X^I} \Phi: \uAut_\C(X)\to \uAut_\C(\LmG_{X^I}\times_{X^I}\Gr_{G,X^I}^{(N)}).
\end{gather*}
Let $F$ and $U$ be as in the proof of \Cref{homotopy-equivalence-D-D'} and consider the evaluation morphism for $\LpG_{X^I}\times_{X^I}\Gr_{G,X^I}$.
Then we get
\begin{equation}\label{mapforequivariantstatement}
\textup{ev}\circ \left(\left(\Phi^{\LpG}\times_{X^I} \Phi\right)\circ F, \id\right): U \times \LpG_{X^I}\times_{X^I}\Gr_{G,X^I} \rightarrow \LpG_{X^I}\times_{X^I}\Gr_{G,X^I}.
\end{equation}
In particular, since the stratification of $\LpG_{X^I}$ is controlled by $\Inc_I$, we have that the map \eqref{mapforequivariantstatement} respects the stratifications.
Therefore when we pass to the analytifications, by applying \cref{stratified-Harpaz} in the usual way we get
\[
\left( U^\an \times \LpG_{X^I}^\an\times_{(X^\an)^I}\Gr_{G,X^I}^\an , \textup{triv}\times \mathfrak{s}_I^\an \right) \rightarrow \left( \LpG_{X^I}^\an\times_{(X^\an)^I}\Gr_{G,X^I}^\an, \mathfrak{s}_I^\an\right).
\]

Let $X$ be the affine line.
Restricting to $[0,1]\subset U^\an$, we finally get
\[
\Psi^{\textup{equiv}}_{[0,1]} \coloneqq \textup{ev} \circ \left( \Phi^{\LpG, \an}_{[0,1]} \times_{\C^I} \Phi^\an_{[0,1]} , \id \right).
\]
By its definition, $\Psi^{\textup{equiv}}_{[0,1]}(1,-)$ restricts to
\[
     \left(\LpG_{D^I} \times_{D^I} \Gr_{G,D^I},\mathfrak{s}_I^\an \right) \to  
     \left( \LpG_{{D'}^I} \times_{{D'}^I}\Gr_{G,{D'}^I},\mathfrak{s}_I^\an \right)
\]
and, by the same proof of \Cref{prop-alg-isotopy} and \Cref{homotopy-equivalence-D-D'}, it gives a stratified homotopy inverse to $i^+_I\times i_I$.

Therefore it remains to show that, for any
$t\in [0,1]$, $\Psi^{\textup{equiv}}_{[0,1]}(t,-)|_{\Gr_{G, D^I}}$ and $\Psi_{[0,1]}(t,-)|_{\Gr_{G, D^I}}$ fit in the commutative diagram
\[
\begin{tikzcd}[column sep=3cm]
\LpG_{D^I}\times_{D^I}\Gr_{G, D^I} 
\arrow[r, "{\Psi^{\textup{equiv}}_{[0,1]}(t,-)|_{\Gr_{G, D^I}}}"]
\arrow[d, "{\act_I}"]
&
\LpG_{D^I}\times_{D^I}\Gr_{G,D^I}
\arrow[d, "{\act_I}"]
\\
\Gr_{G,D^I}
\arrow[r, "{\Psi_{[0,1]}(t,-)|_{\Gr_{G, D^I}}}"]
&
\Gr_{G, D^I}.
\end{tikzcd}
\]
This, in turn, is implied by checking that for any $f\in \Aut_\C(X)$ and each $X^\phi$, the diagram
\begin{equation*}
\begin{tikzcd}[column sep=3cm]
\LpG_{X^I}|_{X^\phi}\times_{X^\phi}\Gr_{G, X^\phi}
\arrow[r, "{(\Phi^{\LpG}_f \times_{X^I} \Phi_f)|_{X^\phi}}"]
\arrow[d]
&
\LpG_{X^I}|_{X^\phi}\times_{X^\phi}\Gr_{G, X^\phi}
\arrow[d]
\\
\Gr_{G, X^\phi}
\arrow[r, "{\Phi_f|_{X^\phi}}"]
&
\Gr_{G, X^\phi},
\end{tikzcd}
\end{equation*} 
is well-defined and commutes. As done in the proof of \Cref{lemma-stratified-automorphism}, by the factorization property \eqref{factorization-property}, it is enough to deal with the case $I=\{*\}$ using the formal coordinates
\[
\mathfrak{bl} : \widehat{X}\times^{\uAut_\C \C\taylor}\Gr_{G}  \triv \Gr_{G,X}.
\]
Recall that at the level of the presheaf quotient $( \widehat{X}\times\Gr_{G} )/\uAut_\C \C\taylor$, the map $\Phi_f$ sends
\[
[(x,\eta,\widetilde{\mathcal{F}},\widetilde{\alpha})]\mapsto [(f^{-1}x, \widehat{f}^{-1}_x\circ \eta, \widetilde{\mathcal{F}},\widetilde{\alpha})]
\]
(see equation \eqref{mapPhifatloclevel}).
Therefore given $(x,g)\in \LpG_X$, on one side we have 
\[
    \begin{tikzcd}[column sep=3cm]
        (x,g), [(x,\eta,\widetilde{\mathcal{F}},\widetilde{\alpha})] 
        \arrow[d, "{\act_{\{*\}}}"]
        &
              \\
        {[(x,\eta,\widetilde{\mathcal{F}}, \eta^* g|_{\widetilde{\Gamma}_x\setminus \Gamma_x}\circ (\eta^{-1})^*\widetilde{\alpha})]}
        \arrow[r, "{\Phi_f}"]
        &
        (f^{-1}x, \widehat{f}^{-1}_x\circ \eta, \widetilde{\mathcal{F}}, \eta^* g|_{\widetilde{\Gamma}_x\setminus \Gamma_x}\circ (\eta^{-1})^*\widetilde{\alpha}).
    \end{tikzcd}
\]
On the other side, we have
\[
    \begin{tikzcd}[column sep=3cm]
        (x,g),(x,\eta,\widetilde{\mathcal{F}},\widetilde{\alpha}) 
        \arrow[r, "{\Phi_f^{\LpG}\times_X \Phi_f}"]
        &
        (f^{-1}x, \widehat{f}_x^*g),(f^{-1}x, \widehat{f}^{-1}_x\circ \eta, \widetilde{\mathcal{F}},\widetilde{\alpha})
        \arrow[d, "{\act_{\{*\}}}"]
        \\
         &
        (f^{-1}x, \widehat{f}^{-1}_x\circ \eta, \widetilde{\mathcal{F}}, (\widehat{f}^{-1}_x\circ \eta)^* (\widehat{f}_x^*g)|_{\widetilde{\Gamma}_{f^{-1}x}\setminus \Gamma_{f^{-1}x}}\circ ((\widehat{f}^{-1}_x\circ \eta)^{-1})^*\widetilde{\alpha}).
    \end{tikzcd}
\]
One concludes computing explicitly the last term:
\begin{align*}
    (\widehat{f}^{-1}_x\circ \eta)^* (\widehat{f}_x^*g)|_{\widetilde{\Gamma}_{f^{-1}x}\setminus \Gamma_{f^{-1}x}}\circ ((\widehat{f}^{-1}_x\circ \eta)^{-1})^*\widetilde{\alpha} = & \  \eta^*  (\widehat{f}^{-1}_x)^*(\widehat{f}_x^*g)|_{\widetilde{\Gamma}_{f^{-1}x}\setminus \Gamma_{f^{-1}x}}\circ \widehat{f}_x^* (\eta^{-1})^*\widetilde{\alpha} \\
    =& \ \eta^* g|_{\widetilde{\Gamma}_{f^{-1}x}\setminus \Gamma_{f^{-1}x}}\circ (\eta^{-1})^*\widetilde{\alpha}.
\end{align*}
The analogous statement holds for the $(N,m)$-truncated objects by an identical argument.
\end{proof}

\begin{thm}\label{equivariance-theorem-Ran}Let $D$ be a metric disk in $\C$. There exists a stratified map $\Psi^{\textup{equiv},\Ran}_{[0,1]}$
\[
\left([0,1]\times \LpG_{\Ran(\C)} \times_{\Ran(\C)} \Gr_{G, \Ran(\C)}, \textup{triv}\times \mathfrak{s}_{\Ran} \right)\to 
\left( \LpG_{\Ran(\C)} \times_{\Ran(\C)} \Gr_{G, \Ran(\C)}, \mathfrak{s}_\Ran^\an\right)
\]
   such that 
   \begin{enumerate}
       \item for any $t\in [0,1]$, $\Psi^{\textup{equiv},\Ran}_{[0,1]}(t,-)$ is a closed embedding, and
       \item the following square commutes:
       \[
       \begin{tikzcd}[column sep=3cm]
        {[0,1]\times \LpG_{\Ran(D)}\times_{\Ran(D)}\Gr_{G, \Ran(D)} }
        \arrow[r, "{\Psi^{\textup{equiv}, \Ran}_{[0,1]}|_{\Gr_{G, \Ran(D)}}}"]
        \arrow[d, "\id_{[0,1]}\times\act_{\Ran}"]
        &
        \LpG_{\Ran(D)}\times_{\Ran(D)}\Gr_{G,\Ran(D)}
        \arrow[d, "\act_{\Ran}"]
        \\
        {[0,1]\times \Gr_{G,\Ran(D)}}
        \arrow[r, "{\Psi^\Ran_{[0,1]}|_{\Gr_{G, \Ran(D)}}}"]
        &
        \Gr_{G, \Ran(D)}.
        \end{tikzcd}
       \]
   \end{enumerate}
\end{thm}

\begin{proof}
    The only difference with respect to the previous proof is that one builds the map $\Psi^{\textup{equiv},\Ran}_{[0,1]}$ in the same way as \cref{extension-of-analytification}, by filtering $\Gr_{G,\Ran(X)}^\an$ and then inducing maps on perfect quotients. Therefore, by construction, $\Psi^{\textup{equiv}, \Ran}_{[0,1]}$ agrees with the action of $\LpG_{\Ran(X)}^\an$.
\end{proof}

\begin{rem}\label{rem-Hecke}
    A nice way to rephrase the \cref{equivariance-theorem-Ran} is the following. One can form a stratified topological stack defined as the quotient stack, relative to $\Ran(D)$,
    $$\hck_{G,\Ran(D)}=\Gr_{G,\Ran(D)}/\LpG_{\Ran(D)}$$ for any metric disk, and then use \cref{equivariance-theorem-Ran} to prove that the induced embedding $$\hck_{G,\Ran(D')}\to \hck_{G,\Ran(D)}$$ is a stratified homotopy equivalence of stacks.
    We chose not to delve into this formalism in the present paper, but the reader can find all the needed terminology in \cite[Appendix B.3]{E3}, \cite{Jansen}.
\end{rem}

\subsection{\texorpdfstring{$\E_2$} --algebra structure}

The aim of this final subsection is to prove \cref{GrGRanisE2}.

\begin{recall}\label{preliminaries-operads}Let $\Fin_*$ be the category of pointed finite sets, and denote by $\langle n\rangle$ the pointed set $\{*,1,\ldots,n\}$. 
For $1 \leq i\leq n$ denote by $\rho_i:\langle n\rangle \to \langle 1\rangle$ the morphism sending $i$ to $1$ and every other element to $*$. 
This morphism is \emph{inert} in $\Fin_*$ (see \cite[Definition 2.1.1.8]{HA}).

Let $N:\cat\to\Cat$ be the simplicial nerve functor. 
Recall that a functor of $\infty$-categories $p:\mathcal{O}^\otimes\to \ner(\Fin_*)$ is an
$\infty$-\emph{operad} if it satisfies the conditions of \cite[Definition 2.1.1.10]{HA}, 

and a \emph{map of $\infty$-operads} $\alpha: \mathcal{O}^\otimes \to {\mathcal{O}'}^\otimes$ is a functor of $\infty$-categories over $\ner(\Fin_*)$ satisfying the conditions of \cite[Definition 2.1.2.7]{HA}.
\end{recall}

We are here interested in $\infty$-operads of the form $\ner(\mathcal{C})\to \ner(\Fin_*)$. 
In this case, we can check whether this map is an $\infty$-operad at the level of $1$-morphisms.

\begin{defin}
    Let $p:\mathcal{C}\rightarrow \Fin_*$ be a functor between categories. Given
    \[
    x,y\in \mathcal{C}, f \in \Hom_{\cc}(x, y),
    \]
    we say that $f$ is $p$-\textit{coCartesian} if for every $z\in \cc$, $g \in \Hom_\cc(x, z)$ and $\overline h\in \Hom_{\Fin_*}(p(y),p(z))$ such that $\overline h\circ p(f)=p(g)$, there exists a unique $h\in \Hom_\cc(y, z)$ such that $h\circ f=g$ and $p(h)=\overline h$.

    We say that $f$ as above is \textit{inert} if it is $p$-cocartesian and $p(f)$ is inert in $\Fin_*$.

    Finally, given $x,y \in\cc$, $\overline{f} \in \Hom_{\Fin_*}(p(x),p(y))$, let $\Hom^{\overline{f}}_{\cc}(x,y)$ be the subset of $\Hom_{\cc}(x,y)$ consisting of morphisms lying over $\overline{f}$. 
\end{defin}

\begin{lem}\label{1-operads}
Let $p:\mathcal{C}\rightarrow \Fin_*$ be a functor between categories.
Suppose that $p$ satisfies the following properties:
\begin{enumerate}
    \item Given an inert morphism $\overline f \in \Hom_{\Fin_*}( \langle m\rangle, \langle n\rangle)$ and $x\in \cc$ s.t. $p(x)=\langle m \rangle$, there exists a $p$-coCartesian morphism $f : x \to y$ s.t. $p(f)=\overline f$.
    
    \item Let $x,y \in\cc$, $\overline{f} \in \Hom_{\Fin_*}(p(x),p(y))$.
    Consider the inert morphism $\rho_i$ and let $y\to y_i$ be a $p$-coCartesian morphism lying over $\rho_i$. 
    Then the induced map $\Hom^{\overline{f}}_{\cc}(x,y) \to \prod_{i}\Hom^{\rho_i\circ \overline{f}}_{\cc}(x,y_i)$ is a bijection.
    
    \item For every finite collection of objects $y_1,\dots,y_n \in \cc$ lying over $\langle 1\rangle$, there exists an object $x\in\cc$ lying over $\langle n\rangle$ and a collection of $p$-coCartesian morphisms $x \to y_i$ lying over $\rho_i$.
\end{enumerate}
Then the induced functor of $\infty$-categories $\ner(p):\ner(\cc)\to \ner(\Fin_*)$ exhibits $\ner(\cc)$ as an $\infty$-operad.
\end{lem}

\begin{proof}
Let $x$ and $y$ be two objects of $\mathcal{C}$.
Recall that the topological space $\Hom^R_{\ner(\mathcal{C})}(x,y)$ of the \emph{right homomorphisms} (see its definition at \cite[page 27]{HTT}) describes the homotopy type $\textup{Map}_{\ner(\mathcal{C})}(x,y)$.
Furthermore the topological space $\Hom^R_{\ner(\mathcal{C})}(x,y)$ is a discrete space in bijection with $\Hom_{\mathcal{C}}(x,y)$.
In particular, the conditions on (products of) mapping subspaces involved in the definition of $\infty$-operad for $\ner(\mathcal{C})\rightarrow \ner(\Fin_*)$ all translate in conditions on (product of) subsets of morphisms in $\mathcal{C}$.
\end{proof}

By analogous consideration we have the following lemma.

\begin{lem}
    Let $f:\mathcal{C}\rightarrow \mathcal{C}'$ be a morphism of categories over $\Fin_*$. 
    If $f$ sends inert morphisms to inert morphisms, then 
    $\ner(f):\ner(\mathcal{C})\rightarrow \ner(\mathcal{C}')$ is a map of $\infty$-operads.
\end{lem}

\begin{recall}\cite[Definition 5.4.4.1]{HA}
Denote by $\textup{Surj}$ the full subcategory of $\Fin_*$ with only surjective maps. 
Given an $\infty$-operad $p:\mathcal{O}^\otimes\longrightarrow \ner(\Fin_*)$, its \emph{non-unital} version $p_{\nun}:\cO^\otimes_\nun\longrightarrow \ner(\Fin_*)$ is defined via the fiber product over $\ner(\textup{Surj})$:
\[
\begin{tikzcd}
\cO^\otimes_\nun
\arrow[d, "{p_{\nun}}"]
\arrow[r]
&
\cO^\otimes
\arrow[d, "{p}"]
\\
\ner(\textup{Surj})
\arrow[r, hook]
&
\ner(\Fin_*).
\end{tikzcd}
\]
\end{recall}
\begin{rem}
By \cite[Remark 2.1.1.3]{HA} $p$ above is a categorical fibration. Hence the above homotopy fiber product coincides with the strict pullback in the category of simplicial sets. 
The composition of $p_{\textup{nu}}$ with the inclusion $\ner(\textup{Surj})\xhookrightarrow{} \ner(\Fin_*)$ exhibits $\cO_{\nun}^\otimes$ as an $\infty$-operad.
\end{rem}

\begin{recall}{\cite[cf. Definition 2.4.1.1, Construction 2.4.1.4 and Corollary 2.4.1.8]{HA}}
Let $\cc$ be a category with finite products.
The product structure induces an $\infty$-operad $q:\ner(\mathcal{C})^{\times} \longrightarrow \ner(\Fin_*)$ such that the $\langle 1\rangle$-fiber (which again coincides with the pull-back in $\sets_\Delta$)
\[
\begin{tikzcd}
\ner(\mathcal{C})^{\times}_{\langle 1\rangle}
\arrow[r]
\arrow[d]
&
\ner(\mathcal{C})^{\times}
\arrow[d,"q"]
\\
\{*\}
\arrow[r, "{\langle 1 \rangle}"]
&
\ner(\Fin_*)
\end{tikzcd}
\]
is isomorphic to the simplicial nerve $\ner(\mathcal{C})$.
More generally the $\langle n\rangle$-fiber $\ner(\mathcal{C})^{\times}_{\langle n\rangle}$ is isomorphic to the product (in $\sets_\Delta$) of $n$-copies of $\ner(\mathcal{C})$.
\end{recall}

\begin{recall}\label{defin-locally-constant-algebra}
    Let $p:\mathcal{O}^{\otimes}\to \ner(\Fin_*)$ be an $\infty$-operad. 
    Let $\cc$ be a category with finite products.
    An $\cO^\otimes$-\textit{algebra object in} $\ner(\cc)^\times$ is a map of $\infty$-operads $\alpha: \cO^\otimes\to \ner(\cc)^\times$. 
    These form an $\infty$-category $\Alg_{\cO^\otimes}(\ner(\cc)^\times)$.
    A \textit{non-unital} $\cO^\otimes$-algebra object in $\ner(\cc)^\times$ is a $\cO^\otimes_\nun$-algebra object in $\ner(\cc)^\times$.

    A $\mathcal{O}_\nun^\otimes$-algebra object $\alpha$ in $\ner(\cc)^\times$ is \emph{locally constant} if the map
    \[
    (\cO_\nun^\otimes)_{\langle 1 \rangle}\xrightarrow{\alpha_{\langle 1 \rangle}} \ner(\mathcal{C})^{\times}_{\langle 1 \rangle}\rightarrow \ner(\mathcal{C})^{\times}
    \]
    sends every morphism of $(\cO^\otimes_{\nun})_{\langle 1\rangle}$ to an isomorphism of $\ner(\mathcal{C})^{\times}$.    
\end{recall}

Let $\Disk(\R^2)$ be the category of opens $U\subset \R^2$
homeomorphic to $\R^2$, where morphisms are the inclusions. 
Let $\MDisk(\R^2)$ be its full subcategory of metric disks $D\subset \R^2$.
\begin{defin}
Let $\Disk(\R^2)^\otimes$ be the fiber category over $\Fin_*$ whose objects are $n$-uples of opens $(U_1,\dots,U_n)$ and whose morphisms $(U_1,\dots,U_m)\to (U_1',\dots,U_n')$ consist of $\overline{f} : \langle m\rangle \to \langle n\rangle$
such that
\begin{enumerate}
    \item $\forall\ 1\leq i \leq n$, if $\overline{f}(j)=i$ then $U_j\subset U_i'$;
    \item $\forall\ 1\leq j'<j\leq m$ s.t. $\overline{f}(j')=\overline{f}(j)=i$ we have   
    $U_{j'}\cap U_j = \varnothing$.
\end{enumerate}
The map $\Disk(\R^2)^\otimes \to \Fin_*$ sends $(U_1,\dots,U_n)\mapsto \langle n\rangle$ (and is the identity on morphisms).
Denote by $\MDisk(\R^2)^\otimes$ the full subcategory of $\Disk(\R^2)^\otimes$ spanned by tuples of metric disks $(D_1,\dots,D_n)$. 
\end{defin}
Taking the simplicial nerve of $\Disk(\R^2)^\otimes \to \Fin_*$ we get a map of $\infty$-categories  $\ner(\Disk(\R^2)^{\otimes}) \to \ner(\Fin_*)$. 
Either checking the conditions of \cref{1-operads} or by noticing that $\ner(\Disk(\R^2)^{\otimes}) $ coincides with the $\infty$-operad $\ner(\Disk(\R^2))^\otimes$ (see \cite[Definition 5.4.5.6]{HA}), we have that $\ner(\Disk(\R^2)^\otimes)$ is an $\infty$-operads.
The same holds true for $\ner(\MDisk(\R^2)^\otimes)$.

\begin{rem}
Let $\Disk(\mathbb{R}^2)^{\otimes}_{\nun}$ be subcategory of $\Disk(\mathbb{R}^2)^{\otimes}$ defined as the fiber product $$\Disk(\mathbb{R}^2)^{\otimes}\times_{\Fin_*} \textup{Surj}.$$
Since the nerve commutes with limits, the nerve $\ner(\Disk(\mathbb{R}^2)^{\otimes}_{\nun})$ coincides with $\ner(\Disk(\R^2)^\otimes)_{\nun}$.
Same definition and property hold for $\MDisk(\R^2)$.
\end{rem}

\begin{recall}\label{E2}
Recall the definition of the \emph{little 2-disks $\infty$-operad} $\E_2$ from \cite[Definition 5.1.0.2]{HA}. 
Its objects are the same as $\Fin_*$, but $\Map_{\E_2}(\langle m\rangle,\langle n\rangle)$ is the homotopy type of 
$$
\coprod_{\overline{f} :\langle m\rangle \to \langle n\rangle} 
\prod_{i=1}^n \textup{Rect}((-1,1)^2\times \overline{f}^{-1}(\{i\}),(-1,1)^2)$$ 
where $(-1,1)$ is the interval in $\R$ and $\textup{Rect}$ stays for the space of \textit{rectilinear embeddings} (see \textit{loc. cit.}).
\end{recall}

\begin{recall}\label{E2-vs-Disk}
Unlike $\ner(\Disk(\R^2)^\otimes), \ner(\MDisk(\R^2)^\otimes)$, $\E_2$ is not the nerve of a category.
However, by \cite[Theorem 5.4.5.15]{HA} there is an equivalence between the $\infty$-category of $(\E_{2})_{\nun}$-algebra objects in $\ner(\cc)^\times$ and the $\infty$-category of locally constant $\ner(\Disk(\R^2)^\otimes)_{\nun}$-algebra objects in $\ner(\cc)^\times$ (where $\cc$ is a category with finite products). 
\end{recall}

The following slight modification of \cref{E2-vs-Disk} is the main tool of the present subsection. 

\begin{prop}\label{E2-vs-MDisk}
Let $\cc$ be a category with finite products.
There is an equivalence between the $\infty$-category of $(\E_2)_\nun$-algebra objects in $\ner(\mathcal{C})^\times$ and the $\infty$-category of locally constant $\ner(\MDisk(\R^2)_{\nun}^\otimes)$-algebra objects in $\ner(\mathcal{C})^\times$.
\end{prop}
\begin{proof}
The aforementioned \cite[Theorem 5.4.5.15]{HA} rests upon \cite[Lemma 5.4.5.10, Lemma 5.4.5.11]{HA}. 
Both lemmas hold if one replaces $\Disk(\R^2)^\otimes$ with $\MDisk(\R^2)^\otimes$: indeed, they rely on the categorical Seifert-Van Kampen Theorem \cite[Theorem A.3.1]{HA}, and therefore one can consider any subbase of the collection of all disks of $\R^2$. 
This means that \cite[Theorem 5.4.5.15]{HA} holds with $\MDisk(\R^2)^\otimes_\nun$ in place of $\Disk(\R^2)^\otimes_\nun$.
\end{proof}

\begin{thm}\label{corollary-algebra-structure}
Let $W$ be the class of stratified homotopy equivalences in $\Str\Top$. The functor 
$$
    \Gr_{G,\Ran(-)}: \MDisk(\R^2)\to \Str\Top,\quad \quad D\mapsto (\Gr_{G,\Ran(D)}, \mathfrak{s}_{\Ran}^\an)
$$
upgrades to a locally constant $\ner(\MDisk(\R^2)_\nun^{\otimes})$-algebra object
\[
\Gr_{G,\Ran(-)}^{\otimes} : \ner(\MDisk(\R^2)_\nun^{\otimes}) \to \ner(\Str\Top[W^{-1}])^\times.
\]
Therefore, for any $D\in \MDisk(\R^2)$, $\Gr_{G,\Ran(D)}$ carries a non-unital $\mathbb{E}_2$-algebra structure in $\Str\Top[W^{-1}]^\times$, independent of the choice of $D$.
\end{thm}

\begin{proof}
First of all, let us define a functor of 1-categories $\mathfrak{G}:
\MDisk(\R^2)^\otimes_\nun \to \Str\Top
$, sending
$$(D_1,\dots, D_n)\mapsto \prod_{i=1}^n\left( \Gr_{\Ran(D_i)}, \mathfrak s_{\Ran}^\an \right).$$
On morphisms, we define it by steps.
For maps $(D_1,\dots, D_n)\rightarrow D'$ over the inert morphism $\rho_i$, it is defined as the projection on the $i$-th component followed by the inclusion $i_{\Ran}$:
\[
\prod_{j=1}^n \left(\Gr_{G,\Ran(D_j)}, \mathfrak{s}_{\Ran}^\an \right)
\xrightarrow{\pi_i} \left(\Gr_{G,\Ran(D_i)}, \mathfrak{s}_{\Ran}^\an \right) \xrightarrow{i_{\Ran}}
\left( \Gr_{G,\Ran(D')}, \mathfrak{s}_{\Ran}^\an \right).
\]
Consider now maps $(D_1,\dots, D_n)\rightarrow D'$ over the \textit{active} morphism
\[
a_n: \langle n\rangle \to \langle 1 \rangle, \quad *\mapsto *, \langle n\rangle\setminus \{*\} 
\mapsto 1,
\]
where $D_i$'s are then all disjoint and contained in $D'$. 
Let $I_1,\dots,I_n\in \fs$, and consider  $({\C}^{I_1}\times \dots \times \C^{I_n})_\disj$ (see definition in \Cref{binary-factorization-property}).
Fix $N\geq 0$. 
By using the factorization property \eqref{factorizationpropertyfordifferentsetsofpoints} and then analytifying (recall that $(-)^{\an}_{\Str, \lft}$ preserves finite limits), consider the isomorphism
$$
\mathfrak{f}_{(I_i)_{i=1}^n}^{(N), \an} : \prod_{i=1}^n
\left(
\Gr_{G,\C^{I_i}}^{(N)},
\mathfrak{s}_{I_i}^{(N)}
\right)_\disj
\to
\left( 
\Gr_{G, \C^{ \sqcup_i I_i}}^{(N)}, \mathfrak{s}_{\sqcup_i I_i}^{(N)} 
\right).
$$
Restricting to $ \prod_i D_i^{I_i} = (\prod_i D_i^{I_i})_{\disj} \subset (\C^{I_1}\times\C^{I_2})_{\textup{disj}}$ on the LHS and to $(D')^{\sqcup_i I_i} \subseteq  \C^{\sqcup_i I_i}$ on the RHS induces a map
$$
\prod_{i=1}^n
\left(
\Gr_{G,D_i^{I_i}}^{(N)},
\mathfrak{s}_{I_i}^{(N)}
\right) 
\to 
\left(
\Gr_{G, {D'}^{\sqcup_i I_i}}^{(N)}, \mathfrak{s}_{\sqcup_i I_i}^{(N)}
\right).
$$
Thanks to \cref{Har15}, taking the colimit of these maps in $N$ gives in turn a map
$$
\prod_{i=1}^n
\left(
\Gr_{G,D_i^{I_i}},
\mathfrak{s}_{I_i}
\right) 
\to 
\left(
\Gr_{G, {D'}^{\sqcup I_i}}, \mathfrak{s}_{\sqcup I_i}
\right).
$$ 
Post-composing by the quotient map into $\Gr_{G, \Ran_{\leq |\sqcup_i I_i|}(D')}$, we thus obtain a morphism
\begin{equation}\label{mapinran2n}
\prod_{i=1}^n \left(
\Gr_{G,D_i^{I_i}},
\mathfrak{s}_{I_i}
\right) 
\to 
\left(
\Gr_{G, \Ran_{\leq |\sqcup_i I_i|}(D')}, \mathfrak{s}_{\Ran}
\right).
\end{equation}
Recall that the relation which defines the quotient map 
$
\left(
\Gr_{G,D^{I}}, \mathfrak{s}_I
\right)
\to 
\left(
\Gr_{G, \Ran_{\leq |I|}(D)}, \mathfrak{s}_{\Ran}
\right)
$
is 
$$(x_I, \cF,\alpha) \sim (x'_I, \cF', \alpha') \iff \{x_1,\dots, x_{|I|}\}=\{x_1',\dots, x_{|I|}'\}, \cF\simeq \cF', \alpha\simeq \alpha'.$$ 
Since also the product map $\prod_{i=1}^n \Gr_{G, D_i^{I_i}} \rightarrow \prod_{i=1}^n \Gr_{G, \Ran_{\leq |I_i|}(D_i)}$ is a quotient map (by \Cref{remark-Rann-Hausdorff} and \Cref{properties-of-perfect-quotients}), the morphism \eqref{mapinran2n}
factors as 
$$
\prod_{i=1}^n
\left(
\Gr_{G, \Ran_{\leq |I_i|}(D_i)}, 
\mathfrak{s}_{\Ran_{\leq |I_i|}}
\right)
\to 
\left(
\Gr_{G, \Ran_{\leq |\sqcup_i I_i|}(D')},
\mathfrak{s}_{\Ran_{\leq |\sqcup_i I_i|}}
\right).
$$
Note that this map is also stratified by the same argument at the end of the proof of \cref{extension-of-analytification}.
We can now use \cref{Har15} again and obtain a continuous map at the level of $\Ran$'s:
$$
\prod_{i=1}^n
\left(
\Gr_{G, \Ran(D_i)}, 
\mathfrak{s}_\Ran
\right)
\to 
\left(
\Gr_{G, \Ran(D')}, 
\mathfrak{s}_{\Ran}
\right).
$$ 
Note also that this assignment on active morphisms respects composition, because the operation of gluing torsors via trivializations away from disjoint systems of points is associative (see the description in \cref{BD-stratification} and \cref{binary-factorization-property}).
Finally, note that any morphism in $\MDisk(\R^2)^\otimes_\nun$ can be written uniquely as a product of inert morphisms followed by a product of active morphisms.

Let now $\mathfrak{G}_{[W^{-1}]}:\MDisk(\R^2)^\otimes_\nun\to \Str\Top[W^{-1}]$ be the functor obtained by postcomposing $\mathfrak{G}$ with the (1-categorical) localization at $W$.
Taking the nerve we get a functor of $\infty$-categories
\[
\ner(\mathfrak{G}_{[W^{-1}]}):
\ner(\MDisk(\R^2)^\otimes_\nun)
\to 
\ner(\Str\Top[W^{-1}]).
\]
It turns out that $\ner(\mathfrak{G}_{[W^{-1}]})$ is \emph{lax} \cite[Definition 2.4.1.1]{HA}: 
for any object $(D_1,\dots, D_n) \in \MDisk(\R^2)^\otimes_\nun)_{\langle n\rangle}$ the inert maps $\mathfrak{G}(\rho_i): \mathfrak{G}(D_1,\dots, D_n) \to \mathfrak{G}(D_i)$ exhibit $\mathfrak{G}(D_1,\dots, D_n)$ as a product $\prod_i \mathfrak{G}(D_i)$. 
Localizing by a class $W$ of maps closed under products preserves products, and so does taking the nerve.
Hence $\ner(\mathfrak{G}_{[W^{-1}]})$ is lax.
By \cite[Proposition 2.4.1.7]{HA} we then obtain a map of $\infty$-operads 
$$
\Gr_{\Ran(-)}^\otimes:
\ner(\MDisk(\R^2)^\otimes_\nun)
\to 
\ner(\Str\Top[W^{-1}])^\times
$$ 
such that $\pi\circ \Gr_{\Ran(-)}^\otimes$ is $\ner(\mathfrak{G}_{[W^{-1}]})$, where $\pi$ is defined in \cite[Proposition 2.4.1.5]{HA}.

Thanks to \cref{E2-vs-MDisk}, in order to conclude the proof it remains to check that $\Gr_{\Ran(-)}^\otimes$ is locally constant: this is a property at the level of the $\langle 1 \rangle$-fiber, over which the functor $\pi|_{\langle 1 \rangle}$ is the identity (see its definition in \cite[Notation 2.4.1.2 and Proposition 2.4.1.5]{HA}).
Therefore it is enough to check that  
\[
\ner(\mathfrak{G}_{[W^{-1}]})_{\langle 1\rangle}:
\ner(\MDisk(\R^2)^\otimes_\nun)_{\langle 1 \rangle} 
\to
\ner(\Str\Top[W^{-1}])_{\langle 1 \rangle}
\]
sends any morphism to an isomorphism of $\ner(\Str\Top[W^{-1}])$.
This is precisely \Cref{homotopy-invariance-Ran} which says that, for $D'\subset D$ metric disks, the induced map $\Gr_{\Ran(D')} \xhookrightarrow{i_\Ran} \Gr_{\Ran(D)}$ is a stratified homotopy equivalence.  
\end{proof}

Note that underlying stratified space (up to stratified homotopy equivalence) of our algebra object is given by the value $\Gr_{\Ran(D_0)}$, for any choice of $D_0\in\MDisk(\R^2)$ (different choices induce values stratified homotopy equivalent to each other. The equivalence is also canonical if the two chosen disks are one contained into the other).

\begin{rem}\label{remark-algebra-infinity-localization}
The same statement of \cref{corollary-algebra-structure} is true if one replaces the 1-categorical localization $\Str\Top[W^{-1}]$ with the $\infty$-categorical localization $\ner(\Str\Top)[W^{-1}]$ together with its Cartesian symmetric monoidal structure. 
The proof is \textit{verbatim} the same until the end of the definition of $\mathfrak{G}$. 
Then, one considers the functor $\ner(\mathfrak{G})$ and post-composes it with the $\infty$-categorical localization at $W$, $\ner(\Str\Top) \to \ner(\Str\Top)[W^{-1}]$, thus obtaining a functor $\ner(\mathfrak{G})_{[W^{-1}]}$. 
One can then apply \cite[Proposition 2.4.1.7]{HA} to $\ner(\mathfrak{G})_{[W^{-1}]}$ in the same way as we applied it to $\ner(\mathfrak{G}_{[W^{-1}]})$, and conclude in the same way.
\end{rem}

\begin{rem}
Note that in general, the universal property of localizations induces a canonical functor of $\infty$-categories $\ner(\Str\Top)[W^{-1}]\to \ner(\Str\Top[W^{-1}])$.
In this sense, the statement of \cref{corollary-algebra-structure} is formally weaker than its $\infty$-categorical version in \cref{remark-algebra-infinity-localization}.
\end{rem}

\begin{rem}
    In the setting of stratified topological stacks mentioned in \cref{rem-Hecke}, one can prove in the same way a statement analogous to \cref{corollary-algebra-structure} involving the $\Hck_{\Ran(D)}$'s, by means of \cref{homotopy-invariance-Ran} and \cref{rem-Hecke}. 
\end{rem}

\begin{appendices}
\section{Recollections and complements on the Beilinson-Drinfeld Grassmannian}\label{recollection}
\noindent In this Appendix, we recall some definitions and properties needed in the paper, stressing some details and proving some folklore properties.
Two sources containing very good introductions to the affine Grassmannian and to the Beilinson--Drinfeld Grassmannians are \cite{Zhu} and \cite{Baumann-Riche}. 
Other useful properties of the Ran Grassmannian can be found in \cite{James}.

\subsection{The stratification of the affine Grassmannian}

\begin{recall}[Definition of $\Gr_G$]\label{recall-Gr}
\cite[(1.2.1)]{Zhu} The \emph{affine Grassmaniann} is the presheaf
$$
    \Gr_G:\Aff_\C^\op\to \sets, \quad \Spec R \mapsto \{(\mathcal{F},\alpha):\,  
    \mathcal{F}\in \Bun_G(\Spec R \taylor), \, \alpha:\mathcal{F}|_{\Spec R \laurent }\triv\cT_{G,\Spec R \laurent }\}/_\sim
$$
where $(\cF, \alpha)\sim (\cG,\beta)$ if and only if there is an isomorphism $\psi:\cF\xrightarrow{\sim} \cG$ whose restriction makes the following diagram commute 
$$
\begin{tikzcd}
    \cF|_{\Spec R \laurent }
    \arrow[rr,"{\psi|_{\Spec R \laurent } }"]
    \arrow[dr, "{\alpha}"']
    &
    &
    \cG|_{\Spec R \laurent }
    \arrow[dl, "{\beta}"]
    \\
    &
    \cT_{G,{\Spec R \laurent }}.    
    &
\end{tikzcd}
$$
By \cite[Theorem 1.22]{Zhu}, $\Gr_G$ is ind-representable by $\underset{N\geq 0}{\colim}\, \Gr_G^{(N)}$,
where each $\Gr_G^{(N)}$ is a projective $\C$-scheme and the transition maps are closed embeddings.
By \cite[Proposition 1.3.6]{Zhu}, it can also be described as the \'etale sheafification 
\begin{equation}\label{equationGrG}
\Gr_G \simeq \left [\faktor{\LG}{\LpG} \right ]_{\ett}\end{equation}
where $\LpG, \LG$ are \'etale sheaves in groups defined as
$$
    \begin{tikzcd}[row sep=0.5mm]
        \LpG : \Aff_\C^\op\to \Grp, \quad
        \Spec R \mapsto G(R\taylor),
    \end{tikzcd}
    \quad  \textup{ and }  \quad
    \begin{tikzcd}[row sep=0.5mm]
        \LG : \Aff_\C^\op\to \Grp, \quad
        \Spec R \mapsto G(R\laurent).
    \end{tikzcd}
$$
By \cite[Proposition 1.3.2]{Zhu}, the presheaf $\LpG$ is representable by the inverse limit
$$
    \LpG \cong \underset{m\geq 0}{\lim}\, \LmG,
$$
where $\LmG$ is the affine group-scheme of finite type over $\C$ representing the functor
$$
    \begin{tikzcd}[row sep=0.5mm] 
        \LmG : \Aff_\C^\op\to \Grp, \quad
        \Spec R \mapsto G(R[t]/(t^m)).
    \end{tikzcd}
$$
\end{recall}

\begin{fact}\label{sheaf-is-presheaf}
As proven in \cite[Theorem 3.4]{Cesnavitius}, the quotient presheaf $\LG/\LpG$ is already an \'etale sheaf.
Indeed every reductive group over a separably closed field is split because it contains a maximal torus \cite[(22.23)]{Milne-iAG} and every torus over a separably closed field is split \cite[(14.25)]{Milne-iAG}. 
Hence $G$ is totally isotropic (see \cite[Example 3.2]{Cesnavitius}).
Therefore in equation \eqref{equationGrG} we do not need to sheafify.
\end{fact}

\noindent Thanks to \Cref{sheaf-is-presheaf}, the schemes $\Gr_G^{(N)}$ have a very explicit description.
\begin{recall}[Cartan decomposition]\label{theactionfactors}
Fix a maximal torus $T\subset \GL_n$ and let $\mathbb{X}_{\bullet}(T)$ be the group $\Hom(\mathbb{G}_m,T)$ of coweights of $T$. 
Fix a set of positive coroots $\Psi^+$ of $T$ and denote by $\xt$ the set of dominant coweights of $T$. 
Endow $\mathbb{X}_{\bullet}(T)$ by its usual partial order, namely 
$$
    \nu \leq \mu \iff \mu - \nu \in \mathbb{N}\cdot \Psi^+ .
$$
This restricts to a partial order on $\xt$.
Finally fix an embedding of posets $\xt\hookrightarrow \N^n$.
Then:
\begin{gather*}
\Gr_{\GL_n}^{(N)}(R) \simeq \big\{[M]\in \GL_n(R\laurent)/\GL_n(R\taylor) : M
\textup{ has a Cartan decomposition } M=ADB,\\
\textup{where } A,B\in \GL_n(R\taylor) \textup{ and }
D=\textup{diag}(t^{-\nu_1},\dots,t^{-\nu_n}) \textup{ with } 0\leq \nu_n\leq\dots\leq \nu_1\leq N\big\}. 
\end{gather*}
In the case of an arbitrary $G$, fix a faithful representation $\rho:G\to \GL_n$ for some $n$, and this induces a closed embedding $\Gr_G\hookrightarrow\Gr_{\GLn}$ (see \cite[Proposition 1.2.5, 1.2.6]{Zhu}).
One then defines the $\Gr_G^{(N)}$'s as the preimage of $\Gr_{\GL_n}^{(N)}$ in $\Gr_G$. Note that $\rho$ also provides an embedding of posets $\xt\hookrightarrow\N^n$.
\end{recall}

\begin{recall}[Stratification on $\Gr_G$ and $\LpG$-action]\label{Schubert}
Consider the action $\LpG\times \Gr_G\to \Gr_G$ by left multiplication $(g, \mathcal{F}, \alpha) \mapsto (\mathcal{F}, g|_{t\neq 0} \circ \alpha)$: by \cite[§ 2.1, Proposition 2.1.5]{Zhu}, its orbits are smooth quasi-projective schemes of finite type over $\C$.
They are called \emph{Schubert cells} $\Gr_{G, \mu}$ and they are indexed by $\mu\in \xt$.
Given $\mu=( \mu_n\leq \dots \leq \mu_1)\in \xt$ then
\begin{gather*}
\Gr_{\GL_n, \mu}(R) \cong \{ [M]\in \Gr_{\GL_n}(R) : M=ADB, 
\textup{with } A,B\in \GL_n(R\taylor)\ \& \
D=\textup{diag}(t^{-\mu_1},\dots,t^{-\mu_n}) \}.
\end{gather*}
In general, $\Gr_{G, \mu}$ is the preimage of $\Gr_{\GL_n, \mu}$ via the closed embedding $\Gr_G\hookrightarrow\Gr_{\GL_n}$ mentioned in \Cref{theactionfactors}.
In particular, 
$$
\overline{\Gr_{G, \mu}}= \bigcup_{\nu\leq\mu} \Gr_{G,\nu},\quad \textup{ and }\quad {(\Gr_{G}^{(N)})}_{\textup{red}} = \bigcup_{\mu_1\leq N} \Gr_{G,\mu}.
$$
Therefore the collections $\{\Gr_{G, \mu}\}_{\mu\in \xt, \mu_1\leq N}$ (resp. $\{\Gr_{G, \mu}\}_{\mu\in \xt}$) give a stratification of $\Gr_G^{(N)}$ (resp. $\Gr_{G}$), making $(\Gr_G^{(N)}, \xt_{\leq N}, \mathfrak{s}^{(N)} : \Gr_G^{(N), \Zar} \to \Alex( \xt_{\leq N} ) )$ (resp. $(\Gr_G, \xt,\mathfrak{s} : \left( \Gr_G \right)^\Zar \to \Alex( \xt ))$), into an element of $\PSmall(\Str\Sch^{\lft}_\C)$.

Endow $\LpG$ with the trivial stratification: by the definition of the strata as the $\LpG$-obits, the left multiplication 
\begin{equation}\label{eq:Lpgagctsstratifies}
(\LpG, \textup{triv} )\times (\Gr_G, \mathfrak{s}) \to  
(\Gr_G, \mathfrak{s}), \quad (g, \mathcal{F}, \alpha)\mapsto (\mathcal{F}, g|_{t\neq 0}\circ \alpha)  
\end{equation}
is a stratified action.
\end{recall}

\begin{rem}\label{nonreducedness}
Let us note that $\Gr_G$ is reduced, for example, when $G$ is semisimple and simply connected (\cite[Theorem 1.3.11]{Zhu}), but for instance it is not for $G=\Gm$ (\cite[Example 1.3.12]{Zhu}).
Indeed, $\Gr_G$ and $\Gr_G^{(N)}$ are not reduced in general, while the $\Gr_{G,\mu}$'s are {by definition}. 
\end{rem}

\begin{recall}[Action of $\LmG$ on $(\Gr_G^{(N)}, \mathfrak{s}^{(N)})$]\label{action-restricts}
    The action of $\textup{L}^+\GL_n$ on $\Gr_{\GL_n}$ restricts to each $\Gr_{\GL_n}^{(N)}$: indeed the action is a left-multiplication by a matrix with coefficients in $R\taylor$, so the order of the poles does not increase.
    Moreover left-multiplication by a matrix of the form $A'+t^N B' \in \textup{L}^+\GL_n(R)$, where $A'\in \GL_n(R), B'$ an $n\times n$ matrix with coefficients in $R$, sends $M\in \Gr_{\GL_n}^{(N)}(R)$ to $A'M C$ with $C\in \GL_n(R\taylor)$ (and not simply $\GL_n(R\laurent)$ because $t^N$ solves the poles in $M$).

    Hence the action factors through $\GL_n(R\taylor / t^N R\taylor)\cong \GL_n(R[t]/t^N)$: so we get 
    \[
    (\textup{L}^N\GL_n, \textup{triv}) \times (\Gr_{\GL_n}^{(N)}, \mathfrak{s}^{(N)}) \to (\Gr_{\GL_n}^{(N)}, \mathfrak{s}^{(N)}).
    \]
    Thanks to the closed embedding $\Gr_G\xhookrightarrow{} \Gr_{\GL_n}$, we recover the general case:
    \begin{equation*}
    \forall N\in \nn,\ \forall m\geq m_N,\quad (\LmG, \textup{triv} )\times (\Gr_G^{(N)}, \mathfrak{s}^{(N)}) \to  
    (\Gr_G^{(N)}, \mathfrak{s}^{(N)})\quad \textup{in } \Str\Sch^\lft_\C.
    \end{equation*}
\end{recall}

\subsection{The stratification of the Beilinson--Drinfeld Grassmannian}

Denote by $\fs$ the category of non-empty finite sets with surjective maps between them.
\begin{notation}[Graphs of points]\label{graphs-of-points}
    Let $R$ be a $\C$-algebra, $I\in \fs$ and $x_I\in X^I(R)$. 
    Let $\textup{pr}_i:X^I\to X$ be the projection onto the $i$-th coordinate and denote by $x_i$ the composite $\textup{pr}_i\circ x_I$. 
    
    We denote by $\Gamma_{x_I}$ the closed (possibly not reduced) subscheme of $X_R$ corresponding to $R$-point of $\textup{Hilb}^{|I|}_X$ via
    $$
        \Spec R\to X^I\to \textup{Sym}^{|I|}_X\simeq \textup{Hilb}^{|I|}_X.
    $$ 
    This subscheme is supported over the union of the graphs $\Gamma_{x_i}$. 
    For instance, if $R=\C$, $I=\{1,2\}$ and $x_1=x_2$ is a closed point of $X$, then $\Gamma_{x_I}$ is the only closed subscheme supported at the point and of length $2$.
\end{notation}

\begin{recall}[Beilinson--Drinfeld Grassmannian]\label{BD-Grassmannian}
\cite[§3.1]{Zhu} For any $I\in \fs$, the \emph{Beilinson--Drinfeld Grassmannian} of power $I$ is the presheaf 
\begin{gather*}
    \Gr_{G, X^I}:\Aff_\C^\op\to \sets,
    \\
   \Spec R\mapsto \big\{(x_I,\mathcal{F},\alpha):\,
   x_I\in X^I(R),\,
   \cF\in \Bun_G(X_R) \textup{ and } \alpha:\cF|_{X_R\setminus \Gamma_{x_I}}\triv \cT_{G,X_R\setminus\Gamma_{x_I}}\big\}/_\sim,
\end{gather*}
where $(x_I, \cF, \alpha)\sim (y_I,\cG,\beta)$ if and only if $x_I=y_I$ in $X^I(R)$ and there is an isomorphism $\psi:\cF\xrightarrow{\sim} \cG$ whose restriction to $X_R\setminus \Gamma_{x_I}$ makes the following diagram commute:
$$
\begin{tikzcd}
    \cF_{X_R\setminus \Gamma_{x_I}}
    \arrow[rr,"{\psi|_{X_R\setminus  \Gamma_{x_I} }}"]
    \arrow[dr, "{\alpha}"']
    &
    &
    \cG_{X_R\setminus  \Gamma_{x_I}}
    \arrow[dl, "{\beta}"]
    \\
    &
    \cT_{G, X_R\setminus \Gamma_{x_I}}.   
    &
\end{tikzcd}
$$
As shown in \cite[Theorem 3.1.3]{Zhu}, the functor $\Gr_{G,X^I}$ is ind-representable by a colimit of projective $X^I$-schemes $\Gr_{G, X^I}^{(N)}$, and the transition maps are closed embedding.
\end{recall}

\noindent If $I=\{*\}$, for any point $x_0:\Spec \C \to X$ we have $\Gr_{G,X}\times_{X}\{x_0\}\simeq\Gr_G$ (\cite[§3.1]{Zhu}): if $X=\mathbb{A}^1_\C$, using the translation automorphism of $\Ac$, we get a splitting $\Gr_{G,\mathbb{A}^1_\C}\simeq \mathbb{A}^1_\C \times \Gr_G$.
However, in general no such splitting is guaranteed: what we have instead is that $\Gr_{G,X}$ is isomorphic to a ``twisted product'', as we now recall.

\begin{recall}[Formal coordinates and the torsor $\widehat{X}$]\label{recall-infinitesiam-formal-nbd}\label{formalcoordandtorsorrecall}
Given an $R$-point $x_I:\Spec R\to X^I$, denote by $\widehat \cO_{\Gamma_{x_I}}$ the sheaf of rings $\underset{n\geq 0}{\lim} \, \mathcal{O}_{X_R}/\mathcal{I}_{\Gamma_{x_I}}^n$. Recall that this limit does not depend on the scheme structure of the closed $\Gamma_{x_I}$ but only on its topology.
Denote by $\widetilde{\Gamma}_{x_I}$ the relative spectrum $\underline{\Spec}_{X_R}( \widehat \cO_{\Gamma_{x_I}} )$: then we get
$$
\begin{tikzcd}
    \Gamma_{x_I}
    \arrow[r, hook, "i_{x_I}"]
    \arrow[d, hook]
    &
    X_R .
    \\
    \widetilde{\Gamma}_{x_I} 
    \arrow[ur, "{i_{\widehat{x}_I}}"']
    &
\end{tikzcd}
$$
If $I=\{*\}$, denote by $\eta_x$ the isomorphism $\Spec R \to \Gamma_x$.
A \emph{formal coordinate at} $x$ is a map $\widehat{x}: \Spec R\taylor \to X$ such that $\widehat{x}|_{t =0} =x$ and such that it factors as 
\[
\begin{tikzcd}
    \Spec R\arrow[r, "\eta_x", "\sim"']
    \arrow[d, hook]
    &
    \Gamma_x
    \arrow[d, hook]
    \arrow[r, hook, "{i_x}"]
    &
    X_R
    \\
    \Spec R\taylor
    \arrow[rru, bend right=60, "{\widehat{x}}"']
    \arrow[r, "{\eta}", "\sim"']
    &
    \widetilde{\Gamma}_{x} 
    \arrow[ur, "{i_{\widehat{x}}}"']
    &
\end{tikzcd}
\]
where $\eta$ is an isomorphism.
Hence $\widetilde \Gamma_{x}$ (and by extension $\widetilde \Gamma_{x_I}$) can be viewed as an \emph{infinitesimal formal neighborhood} of $\Gamma_{x}$ (resp. $\Gamma_{x_I}$).

By abuse of notation, we will denote by $i_{\widehat x_I}$ also its restriction to the open $\widetilde{\Gamma}_{x_I} \setminus \Gamma_{x_I}$.

\medskip\noindent 
The presheaf of formal coordinates $\widehat{X}: \Aff_{\C}^\op\to \sets$ is then defined as
\begin{gather*}
    \Spec R \mapsto \widehat{X}(R)=\{ (x,\eta) : x\in X(R), \eta: \Spec R\taylor \overset{\sim}{\to}  \underline{\Spec}_{X_R}( \widehat{\mathcal{O}}_{\Gamma_x}) \textup{ such that }  \eta|_{t=0}=\eta_x \}.
\end{gather*}
Let $\pi:\widehat{X}\to X$ be the projection $(x,\eta)\mapsto x$.
Then we have an action of the ind-group-scheme $\uAut_{\C} \C\taylor$ on it by
$$
    \uAut_{\C} \C\taylor \times_X \widehat{X}\to \widehat{X},\quad  (g, x,\eta)\mapsto (x, \eta\circ g).
$$
This makes $\widehat{X}$ into a right $\uAut_{\C} \C\taylor$-torsor over $X$ (see \cite[§5.3.11]{BD-Hitchin}).

\end{recall}

\begin{recall}[Twisted product]\cite[\S 0.3.3]{Zhu}. 
Consider the right-action of $\uAut_{\C} \C\taylor$ on $\Gr_G$ by pull-back, $g\cdot ( \mathcal{F}, \alpha) \mapsto ( g^{*} \mathcal{F}, g^{*} \alpha )$.
Given the $\uAut_{\C} \C\taylor$-torsor $\widehat{X}$ and the $\uAut_{\C} \C\taylor$-functor $\Gr_G$, their \emph{twisted product}
is
$$
    \widehat{X}\times^{\uAut_\C\C\taylor} \Gr_G = \left( \widehat{X} \times \Gr_G / \uAut_\C\C\taylor \right)_{\ett}
$$
with $\uAut_\C\C\taylor$ acting diagonally.
The twisted product is also called \emph{contracted product}.
\end{recall}

\begin{rem}\label{remarkXhatisatorsor}
    The functor $\widehat X$ is an \'etale torsor. Indeed, the e curve $X$ is  \'etale-locally isomorphic to $\mathbb{A}^1_{\C}$.
    In this setting $X_R$ is $\Spec R[t]$, the ideal $\mathcal{I}_{\Gamma_x}$ is $(t-r)$, $r\in R$, and thus $ \widehat \cO_{\Gamma_x}\simeq R\taylor$.
    Moreover when $X=\Ac$ the twisted product $\widehat X\times^{\uAut_\C\C\taylor}\Gr_G$ indeed trivializes as $\Ac\times\Gr$. Hence, the twisted product is \'etale-locally a product $X\times \Gr_G$.
\end{rem}

\begin{prop}\label{characterization-X-hat}
There is a (noncanonical) isomorphism 
$$
\mathfrak{bl} : \widehat X\times^{\uAut_\C\C\taylor}\Gr_G \triv \Gr_{G,X}.$$
\end{prop}

\begin{proof}Let $x:\Spec R\to X$ be an $R$-point.
Recall that the Beauville-Laszlo theorem \cite{Beauville-Laszlo} tells us that the restriction map $\Bun_G(X_R)\to \Bun_G( X_R\setminus \Gamma_{x} )$ fits in the equivalence of categories
\begin{equation}\label{BeauvilleLazloeq}
    \Bun_G(X_R) \simeq \Bun_G( \Spec R\taylor ) 
    \times_{ \Bun_G( \Spec R\laurent ) } 
    \Bun_G( X_R\setminus \Gamma_{x} ).
\end{equation} 
This induces a morphism of presheaves
\begin{equation}\label{theisomoorphism}
    \widehat X\times\Gr_G \to \Gr_X, \quad \quad [(x,\eta,\widetilde\cF,\widetilde\alpha)] \mapsto [(x,\cF,\alpha)]
\end{equation}
where $(\cF,\alpha)$ is a pair such that 
$$
\eta^*i_{\widehat{x}}^*\mathcal{F} \cong \widetilde{\mathcal{F}}, \quad 
\eta|_{\Spec R\laurent}^* i_{\widehat{x}}^* \alpha \cong \widetilde \alpha,
$$
which is uniquely determined (up to isomorphism) by \eqref{BeauvilleLazloeq}.
Note that \eqref{theisomoorphism} is $\uAut_\C\C\taylor$-equivariant, because for $[(x,\eta \circ g, g^{*}\widetilde{\mathcal{F}}, g^{*}\widetilde \alpha)]$ the same pair $(\mathcal{F},\alpha)$ works fine:  
$$
     g^{*} \widetilde{\mathcal{F}} = g^* ( \eta^* i_{\widehat{x}}^*\mathcal{F} ), \quad  g^{*} \widetilde \alpha = g^* ( \eta^* i_{\widehat{x}}^* \alpha ).
$$
Therefore we get a map of presheaves
$$
\widehat X\times\Gr_G / \uAut_\C\C\taylor \to \Gr_{G, X},
$$
which then induces a map between the \'etale sheaves 
\begin{equation}\label{map-twisted-product-to-GrX}
\mathfrak{bl} : \widehat X\times^{\uAut_\C\C\taylor}\Gr_G\to \Gr_{G, X}.
\end{equation} 
The map \eqref{map-twisted-product-to-GrX} is an isomorphism. Indeed, up to passing to an \'etale chart parametrized by $\Ac$, it can be rewritten as the identity map $$\Ac\times \Gr\to \Ac\times \Gr$$ (the fact that it is the identity comes from the fact that the identification of $\Gr_{\Ac}$ with $\Ac\times\Gr$ is exactly the Beauville-Laszlo gluing procedure used in the definition of the map \eqref{map-twisted-product-to-GrX}). 
\end{proof}

\begin{recall}[Stratification of $\Gr_{G, X}$]\label{stratification-GrX}
(\cite[\S 2.1 and Theorem 1.1.3]{Zhu})
By definition of $\Gr_{G, \nu}$ and $\Gr_{G}^{(N)}$, the action of $\uAut_\C \C\taylor$ on $\Gr_G$ restricts to each $\Gr_{G, \nu}$ and to each $\Gr_{G}^{(N)}$: therefore one can set 
\begin{gather*}
    \Gr_{G,X,\nu} \coloneqq \mathfrak{bl}(\widehat X\times^{\uAut_\C \C\taylor}\Gr_{G,\nu}),
    \quad
    \Gr_{G,X,\leq \mu} \coloneqq \mathfrak{bl}(\widehat X\times^{\uAut_\C \C\taylor}\Gr_{G,\leq \mu}),
    \\
    \Gr_{G,X}^{(N)} \coloneqq \mathfrak{bl}(\widehat X\times^{\uAut_\C \C\taylor}\Gr_{G}^{(N)}).
\end{gather*}
With this description, it is clear that $\{ \Gr_{G,X,\mu} \}_{\mu\leq N}$ are reduced schemes defining stratifications on the $\Gr_{G,X}^{(N)}$'s, which are compatible with the transition maps in $N$:
therefore we have
\[
(\Gr_{G,X}^{(N)}, \mathfrak{s}^{(N)}) \in \Str\Sch^{\lft}_\C, \quad 
(\Gr_{G,X}, \mathfrak{s}) \in \PSmall(\Str\Sch^{\lft}_\C).
\]
\end{recall}

\begin{recall}[Stratification on $\Gr_{G,X^I}^{(N)}$ and on $\Gr_{G,X^I}$]\label{BD-stratification} 
(\cite[\S 4.2]{Nadler-Perverse-real}, \cite[\S 4.3]{Cass-et-al} and \cite[\S 3.2]{Zhu}) 
Given $I\in \fs$, consider a surjection $\phi: I\twoheadrightarrow J$ of non-empty sets: call $\Delta_\phi$ the associated diagonal embedding
$$
    \Delta_{\phi}:
    X^J\xhookrightarrow{} X^I,
    \quad \quad 
    (x_1', \dots , x_{|J|}') \mapsto x_I \textup{ where } x_i = x'_{\phi(i)}. 
$$
This defines the so-called \textit{incidence stratification} $(X^I, \Inc_I)$, whose stratifying poset consists of partitions of $I$, partially ordered by refinement.
Given $\phi$, let $X^{\phi}$ be the locally closed subschemes of $X^I$ defined as
$$
    X^{\phi} \coloneqq \{ x_I\in X^I: x_i= x_j \textup{ iff } \phi(i)=\phi(j), \textup{ and } \Gamma_{x_i}\cap \Gamma_{x_j}=\varnothing \textup{ otherwise}\}.
$$ 
Furthermore, denote by $\left( \prod^{|J|}_{j=1}\Gr_{G,X} \right)_{\disj}$ the restriction of $\left( \prod^{|J|}_{j=1} \Gr_{G,X} \right)$ to the open $X^{\id_J}$, which is explicitly
$\{ x_J \in X^{J} : \Gamma_{x_i} \cap \Gamma_{x_j} = \varnothing \ \forall \, i\neq j \}$. 
Let $\Gr_{G, X^{\phi}}$ be the restriction of  $\Gr_{G, X^I}$ to ${X^\phi}$. 
By \cite[Proposition 4.2.1]{Nadler-Perverse-real}, over $X^\phi$ we have an isomorphism
\begin{equation}\label{factorization-property}
\mathfrak{f}_\phi : \left ( \prod_{j=1}^{|J|}\Gr_{G,X} \right )_{\disj} \triv \Gr_{G, X^{\phi}},
\end{equation} 
which is usually referred to as the \emph{factorization property}.
On points, it is defined as
\[
\left( (x_1,\mathcal{F}_1, \alpha_1) ,\dots,  (x_{|J|},\mathcal{F}_{|J|}, \alpha_{|J|}) \right) \mapsto ( \Delta_\phi^{-1}(x_1,\dots, x_{|J|}), \mathcal{F}, \alpha )
\]
where $\mathcal{F}$ is the torsor obtained by gluing $(\mathcal{F}_i, \bigcap_{j\neq i}\Gamma_{x_j}^c)$ with $(\mathcal{F}_{i'}, \bigcap_{j\neq i'}\Gamma_{x_j}^c)$ using $\alpha_{i'}^{-1} \circ \alpha_i$ on $\bigcap_{j}\Gamma_{x_j}^c$.  
By the definition of $\Gr_{G, X^I}^{(N)}$, the isomorphism $\mathfrak{f}_\phi$ restricts to $\Gr_{G, X^I}^{(N)}$: 
$$
\mathfrak{f}_\phi^{(N)} : \left( \prod_{j=1}^{|J|} \Gr_{G, X}^{(N)} \right)_{\disj}
\triv
\Gr_{G, X^\phi}^{(N)}\coloneqq \Gr_{G, X^I}^{(N)}|_{X^\phi} 
$$
(see \cite[Thm. 3.1.3]{Zhu}).
For any $\underline\nu=(\nu^1,\dots,\nu^{|J|})\in(\xt)^{|J|}$ we denote by $\Gr_{G,X^{\phi},\underline{\nu}}$ the locally closed subsheaf of $\Gr_{G,X^{\phi}}$ defined as the $\mathfrak{f}_\phi$-image of
\begin{equation}\label{eq:factorizationprop}
    \left ( \prod_{j=1}^{|J|}\Gr_{G,X} \right )_{\disj} \bigcap \prod_{j=1}^{|J|}\Gr_{G,X,\nu^j}. 
\end{equation}
Let $P_I$ be the set $\{(\phi:I\twoheadrightarrow J, \underline{\nu})\}_{\phi, \underline{\nu}}$: we say that  $(\phi: I \twoheadrightarrow J, \underline{\nu})\leq (\phi': I \twoheadrightarrow J', \underline{\nu}')$ if and only if there exists a surjection $\psi: J' \twoheadrightarrow J$ such that $\phi = \psi \circ \phi'$ (so $\phi$ identifies more coordinates than $\phi'$) and for every $j\in J$ 
$$
    \nu_j \leq 
    \sum_{j'\in\, \psi^{-1}\{j\}} \nu_{j'}'.
$$ 
Note that for any $(\phi,\underline{\nu})\in P_I$ we have $\Gr_{G, X^{\phi},\underline{\nu}} \subseteq \Gr_{G, X^{I}}^{(N)}$ for $N$ big enough, which in particular means that 
\[
\Gr_{G, X^{\phi},\underline{\nu}}  = \mathfrak{f}_\phi^{(N)} \left( \left ( \prod_{j=1}^{|J|}\Gr_{G,X}^{(N)} \right )_{\disj} \bigcap \prod_{j=1}^{|J|}\Gr_{G,X,\nu^j}  \right).
\]
The stratification on $\Gr_{G,X^I}^{(N)}$ (resp. on  $\Gr_{G,X^I}$) induced by $\Gr_{G, X^{\phi},\underline{\nu}}$'s will be denoted as:
\[
(\Gr_{G,X^I}^{(N)}, \mathfrak{s}_I)\in {\Str\Sch_\C^\lft}/_{ (X^I,\Inc_I)},\ 
(\Gr_{G,X^I}, \mathfrak{s}_I)\in \PSmall( {\Str\Sch_\C^\lft}/_{ (X^I,\Inc_I)} ).
\]

Then by definition, the isomorphisms $\mathfrak{f}_\phi$ and $\mathfrak{f}_\phi^{(N)}$ are of stratified presheaves.
Note that the restriction to the fiber at any diagonal point $(x,\dots,x)$ is the scheme $\Gr_{G}^{(|I|N)}$ (resp. the ind-scheme $\Gr_{G}$) with their original stratifications from \cref{Schubert}.
\end{recall}

\begin{rem}\label{binary-factorization-property}
Let $I_1,\dots, I_n\in \fs$. 
The same proof as the one for the factorization property shows that a similar isomorphism holds over the open 
\[
\left( X^{I_1}\times \dots \times X^{I_n}\right)_\disj = \{ (x_{I_1}, \dots , x_{I_n})\in X^{\sqcup_i I_i}: \Gamma_{x_{I_i}}\cap \Gamma_{x_{I_j}}=\varnothing\ \forall i\neq j \}.
\]
Gluing torsors along $ \cap_{j\neq i,i'}  \Gamma_{x_{I_j}}^c$ induces an isomorphism of stratified presheaves 
\begin{gather}\label{factorizationpropertyfordifferentsetsofpoints}
\mathfrak{f}_{(I_j)_{j=1}^n} : 
\left( \prod_{i=1}^n \Gr_{G, X^{I_i}}, \prod_{i=1}^n \mathfrak{s}_{I_i} \right)|_{\left( X^{I_1}\times \dots \times X^{I_n} \right)_\disj}
\triv 
\left( \Gr_{G, X^{ \sqcup_i I_i }} , \mathfrak{s}_{ \sqcup_i I_i} \right)|_{\left( X^{I_1}\times \dots \times X^{I_n} \right)_\disj}.
\end{gather}
\end{rem}

\subsection{Action of \texorpdfstring{$\mathrm{L}^+G_{X^I}$}{} on \texorpdfstring{$(\mathrm{Gr}_{G,X^I}, \mathfrak{s}_I)$}{}}
In \Cref{Schubert} we have seen that we have a stratified action of $(\LpG, \textup{triv})$ on $(\Gr_G, \mathfrak{s})$. 
This can be extended to $(\Gr_{G, X^I}, \mathfrak{s}_I)$.

\begin{recall}[Beilinson-Drinfeld version of $\LpG$]\label{defin-BD-arc-group}
For $I\in \fs$, define
\begin{gather*}
    \LpG_{X^I}:\Aff_\C^\op\to \sets, \quad \quad
    \Spec R \mapsto \{(x_I,g):\, x_I\in X^I(R), g\in G(\widetilde{\Gamma}_{x_I} ) \}.
\end{gather*}
Note that $G(\widetilde{\Gamma}_{x_I})\simeq \uAut( \cT_{G, \widetilde{\Gamma}_{x_I}})$, because any $G$-equivariant automorphism $G\times \widetilde{\Gamma}_{x_I}\to 
G\times \widetilde{\Gamma}_{x_I}$ over $\widetilde{\Gamma}_{x_I}$ is determined by $\{e_G\}\times \widetilde{\Gamma}_{x_I}\to G$.
\end{recall}

\begin{rem}\label{compatibilityLpGXandLpG}
It is indeed an extension of $\LpG$:let $I=\{*\}$, $X=\mathbb{A}^1_{\C}$ and consider the point $0: \Spec \C \to \mathbb{A}^1_{\C}$. 
Since $R\taylor\simeq  \widehat \cO_{\Gamma_{0}}$ then $\Aut( \cT_{G, \widetilde{\Gamma}_{0} } ) \simeq \Aut( \cT_{G, \Spec R \taylor } )$ and $\LpG_{\mathbb{A}^1_{\C}}|_{0}\simeq \LpG$.
\end{rem}

\begin{rem}\label{LmGXI}
Consider $$
    \LmG_{X^I}:\Aff^{\op}_{\C}\to \sets, \quad \quad \Spec R\mapsto \{(x_I,g): x_I\in X^I(R), g\in G(\Gamma_{x_I}^m)\}
$$
where $\Gamma_{x_I}^m$ is a short-hand for $\Spec_{X_R}\mathcal{O}_{X_R}/\mathcal{I}_{\Gamma_{x_I}}^m$. 
These are smooth group $X^I$-schemes (locally of finite type) and there is an isomorphism 
$$
\LpG_{X^I}\simeq\lim_{m\geq 0}\LmG_{X^I}
$$ 
(see \cite[Lemma 2.5.1]{Raskin-Principal-Series-II}).
Consider the forgetful functor $\LmG_{X^I}\to X^I$: pulling back the incidence stratification on $X^I$, we get a stratified presheaf $(\LmG_{X^I},\Inc_I)$.
Moreover since 
\[
\LmG_{X^I} \times_{X^I} \LmG_{X^I} \to \LmG_{X^I}, \quad (x_I,g)\cdot (x_I,g') \mapsto (x_I, gg')
\]
respects the incidence stratification, we get that $(\LmG_{X^I},\Inc_I) \in \Grp\left( {\Str\Sch^\lft_\C}/_{(X^I,\Inc_I)} \right)$.
Since all the $(\LmG_{X^I}, \Inc_I)$ have the same stratification, by \Cref{StrTopiscocomplete} we have {$(\LpG_{X^I}, \Inc_I)\in \Grp \left( {\Str\Sch_\C}/_{(X^I,\Inc_I)} \right)$}. 
\end{rem}

\begin{rem}\label{compstrbetweenLpGXand LpG}
    Over $X$ the incidence stratification is trivial: thus, when restricted to the fiber $0:\Spec \C \to \mathbb{A}^1_\C$, by \Cref{compatibilityLpGXandLpG} we get that $(\LmG_{\mathbb{A}^1_\C}, \textup{triv})|_{0} \cong (\LmG, \textup{triv})$, and by \Cref{StrTopiscocomplete} the same is true for $(\LpG_{\mathbb{A}^1_\C}, \textup{triv})$.
\end{rem}    

In order to define a global action of $(\LpG_{X^I}, \Inc_I)$ on $(\Gr_{G, X^I}, \mathfrak{s}_I)$, we recall the definition of $\Gr_{G,X^I}^\loc$.

\begin{defin}\label{defin-Grloc}
    For $I\in\fs$, we denote by $\Gr_{G,X^I}^\loc$ the presheaf    \begin{gather*}
        \Gr_{G,X^I}^\loc: \Aff_{\C}^{\op} \to \sets,
        \\
        \Spec R \mapsto \{ (x_I,\widetilde \cF, \widetilde \alpha) : x_I\in X^I(R),
        \widetilde \cF \in \Bun_G( \widetilde \Gamma_{x_I} ), 
        \widetilde \alpha : \widetilde \cF |_{ \widetilde \Gamma_{x_I} \setminus \Gamma_{x_I} } \triv \cT_{G,\, \widetilde \Gamma_{x_I}\setminus \Gamma_{x_I} }\}/_\sim
    \end{gather*}
    where the equivalence relation is the analogue of the one for $\Gr_{G, X^I}$. 
\end{defin}

\begin{lem}\label{lemma-glob-loc}
    The restriction map 
    $$
        \mathfrak{r}_I : \Gr_{G, X^I} \to \Gr_{G, X^I}^\loc , \quad \quad
        (x_I,\cF,\alpha)\mapsto (x_I, i_{\widehat{x_I}}^* \cF, i_{\widehat{x_I}}^*\alpha )
    $$
    is an isomorphism of presheaves.
\end{lem}

\begin{proof}
The restriction map commutes with the forgetful functor towards $X^I$: so it is enough to check it is an isomorphism on fibers. 
So let us fix $x_I\in X^I(R)$ and compare the two fibers
    \begin{gather*}
        \Gr_{G, X^I}|_{x_I}(R) = \{ \cF \in \Bun_G( X_R ), \alpha : \cF |_{X_R\setminus \Gamma_{x_I} } \triv \cT_{G,\, X_R\setminus \Gamma_{x_I} } \}_{/\sim},\\
        \Gr_{G, X^I}^\loc|_{x_I}(R) =\{ \widetilde \cF \in \Bun_G( \widetilde \Gamma_{x_I} ), \widetilde \alpha : \widetilde \cF |_{ \widetilde \Gamma_{x_I} \setminus \Gamma_{x_I} } \triv \cT_{G,\, \widetilde \Gamma_{x_I} \setminus \Gamma_{x_I} } \}_{/\sim}.
    \end{gather*}
At the level of fibers the map $\mathfrak{r}_I$ coincides with taking the $\pi_0$ of the restriction map of groupoids 
    \begin{equation}\label{comparison}
        \Bun_G(X_R) \times_{\Bun_G (X_R\setminus \Gamma_{x_I})} \{ \cT_{G,\, X_R\setminus \Gamma_{x_I}} \}
        \to 
        \Bun_G( \widetilde \Gamma_{x_I} ) \times_{ \Bun_G ( \widetilde \Gamma_{x_I} \setminus \Gamma_{x_I} ) } \{ \cT_{G,\,  \widetilde \Gamma_{x_I} \setminus \Gamma_{x_I} }\},
    \end{equation}
again given by restricting via $\widehat x_I : \widetilde \Gamma_{x_I} \setminus \Gamma_{x_I} \to X_R \setminus \Gamma_{x_I}$. 
It thus suffices to show that the map at the level of groupoids is an equivalence: this is exactly the ``family'' version of the Beauville-Laszlo theorem \cite[Remark 2.3.7]{BD-Hitchin}.
Indeed, it says that the restriction map gives an equivalence between $\Bun_G(X_R) \times_{ \Bun_G(X_R\setminus \Gamma_{x_I})}\{ \cT_{G,\, X_R\setminus \Gamma_{x_I}} \}$ and
\begin{gather*}
    \Bun_G( \widetilde \Gamma_{x_I} ) \times_{ \Bun_G(  \widetilde \Gamma_{x_I} \setminus \Gamma_{x_I} ) } 
    \Bun_G(X_R\setminus \Gamma_{x_I})
    \times_{\Bun_G(X_R\setminus \Gamma_{x_I})} \{ \cT_{G,\, X_R\setminus \Gamma_{x_I}} \}
\end{gather*}
which is in turn equivalent to the right-hand side of \eqref{comparison}
    $$
    \Bun_G( \widetilde \Gamma_{x_I} ) \times_{ \Bun_G ( \widetilde \Gamma_{x_I} \setminus \Gamma_{x_I} ) } \{ \cT_{G,\,  \widetilde \Gamma_{x_I} \setminus \Gamma_{x_I} }\}.
    $$
\end{proof}

\begin{rem}\label{remarkaboutGrLocbeingthetwisted product}
In particular the functor $\Gr_{G,X^I}^\loc$ is an \'etale sheaf.
Furthermore, for $I=\{*\}$, it is canonically isomorphic to the twisted product $\widehat{X}\times^{\uAut_\C \C\taylor} \Gr_G$.
Indeed pick an affine \'etale cover of $X$ made of $\mathbb{A}^1_\C$: over the affine line the two descriptions are the same via 
\[
     \mathfrak{d}: \widehat{X}\times^{\uAut_\C \C\taylor} \Gr_G \triv \Gr_{G,X}^\loc, \quad (x,\eta, \widetilde{\mathcal{F}},\widetilde{\alpha})\mapsto (x, (\eta^{-1})^*\widetilde{\mathcal{F}}, (\eta^{-1})^*\widetilde{\alpha}).
\] 
\end{rem}

\begin{rem}\label{action-BD}
The functor $\LpG_{X^I}$ acts on $\Gr_{G,X^I}^{\loc}$ over $X^I$ by modification of the trivialization $\widetilde{\alpha}\mapsto g|_{\tilde \Gamma_{x_I}\setminus \Gamma_{x_I}}\circ \widetilde{\alpha}$. 
By \Cref{lemma-glob-loc} we get an induced action $\act_I$ over $X^I$ via pullback by $\mathfrak{r}_I$:
\begin{equation}\label{actionofLpGonGrGXI}
\begin{tikzcd}
\LpG_{X^I} \times_{X^I} \Gr_{G,X^I}
\arrow[r, "\act_I"]
\arrow[d, "\id\times \mathfrak{r}_I", "\wr"']
& 
\Gr_{G,X^I}
\arrow[d, "\mathfrak{r}_I", "\wr"']
\\
\LpG_{X^I} \times_{X^I} \Gr_{G,X^I}^\loc 
\arrow[r, "\act_I^\loc"]
& 
\Gr_{G,X^I}^\loc . 
\end{tikzcd}
\end{equation}
\end{rem}

\begin{prop}\label{stratifiedactions}
The action $\act_I$ is stratified.   
Moreover, for every $N\geq 0$ there exists an integer $m_{N,I}$ such that for any $m\geq m_{N,I}$ the action $\act_I$ factors as a stratified action over $X^I$:
$$
\act_I^{(N)}: (\LmG_{X^I}, \Inc_I) \underset{(X^I, \Inc_I)}{\times} (\Gr_{G,X^I}^{(N)}, \mathfrak{s}_I^{(N)}) \to  (\Gr_{G,X^I}^{(N)}, \mathfrak{s}_I^{(N)}).
$$
\end{prop}
\begin{proof}
Let us prove that the action is stratified.
First restrict the action to $X^\phi$, $\phi:I\twoheadrightarrow J$: by factorization property \eqref{factorization-property} we get 
$$
\LpG_{X^I}|_{X^\phi} \times_{X^\phi} \left( \prod_{j=1}^{|J|} \Gr_{G, X} \right)_\disj 
\xrightarrow[\sim]{\id \times \mathfrak{f}_\phi}
\left( \LpG_{X^I} \times_{X^I} \Gr_{G, X^I} \right)|_{X^\phi}
\xrightarrow{\act_I} 
\Gr_{G,X^I}. 
$$
Hence it is enough to deal with the $I=\{*\}$ case. 
Consider the stratum $\Gr_{G, X, \nu}$
and the diagram
\[
\begin{tikzcd}
\LpG_{X} \times_{X} \widehat{X} \times^{\uAut_\C \C\taylor} \Gr_{G, \nu}
\arrow[d, "\id\times \mathfrak{bl}", "\wr"']
& 
\widehat{X}\times^{\uAut_\C \C\taylor} \Gr_G
\arrow[d, "\mathfrak{bl}", "\wr"']
\\
\LpG_{X^I} \times_{X^I} \Gr_{G,X, \nu} 
\arrow[r, "\act_{\{*\}}"]
& 
\Gr_{G,X}. 
\end{tikzcd}
\]
We want to check that $\act_{\{*\}} (\LpG_{X^I} \times_{X^I} \Gr_{G,X, \nu} )$ lies in $\Gr_{G, X, \nu}$.
So let us pick $\left( (x,g), (x,\eta, \widetilde{\mathcal{F}},\widetilde{\alpha}) \right) \in \LpG_{X^I} \times_{X^I} \Gr_{G,X, \nu}$.
Via $\id\times \mathfrak{bl}$, it maps to $\left( (x,g),(x,\cF,\alpha) \right)$ where $\mathcal{F},\alpha$ are such that 
$$
i_{\widehat{x}}^*\mathcal{F} \cong (\eta^{-1})^*\widetilde{\mathcal{F}}, \quad 
i_{\widehat{x}}^* \alpha \cong (\eta|_{t\neq 0}^{-1} )^* \widetilde \alpha.
$$
The restriction isomorphism $\id\times \mathfrak{r}_{\{*\}}$ sends it to $\left( (x,g), (x, i_{\widehat{x}}^*\mathcal{F}, i^*_{\widehat{x}}\alpha) \right)$, which is then equal to $\id\times \mathfrak{d}\left((x,g), (x, i_{\widehat{x}}^*\mathcal{F}, i^*_{\widehat{x}}\alpha)  \right)$, by the above equalities. In particular $\mathfrak{r}_{\{*\}}\circ \mathfrak{bl} = \mathfrak{d}$.
Hence we have
\[
\begin{tikzcd}
\left((x,g), (x,\eta, \widetilde{\mathcal{F}},\widetilde{\alpha}) \right) 
\arrow[d, "\id\times \mathfrak{d}"]
&
\\
\left( (x,g), (x, (\eta^{-1})^*\widetilde{\mathcal{F}}, (\eta^{-1})^*\widetilde{\alpha}) \right)
\arrow[r, "\act_{\{*\}}^\loc"]
&
(x, (\eta^{-1})^*\widetilde{\mathcal{F}}, g|_{\widetilde{\Gamma}_X\setminus \Gamma_x} \circ (\eta^{-1})^*\widetilde{\alpha}).
\end{tikzcd}
\]
Since $g_{\widetilde{\Gamma}_X\setminus \Gamma_x}$ is the same as $(\eta^{-1})^*(g_{t\neq 0})$ (where $g$ is now viewed as an element of $\Aut(\mathcal{T}_{G, t\neq 0})$) we have that $\mathfrak{d}^{-1}\left( (x, (\eta^{-1})^*\widetilde{\mathcal{F}}, g|_{\widetilde{\Gamma}_X\setminus \Gamma_x} \circ (\eta^{-1})^*\widetilde{\alpha}) \right) = (x, \widetilde{\mathcal{F}}, g|_{t\neq 0} \circ \widetilde{\alpha})$. 
This belongs to $\widehat{X}\times^{\uAut_\C \C\taylor} \Gr_{G, \nu}$ by \Cref{eq:Lpgagctsstratifies}.
The same argument implies that the restriction map is compatible with the stratification on $\Gr_{G,X^I}^{(N)}$.

The fact that $\act_I^{(N)}$ factors through the quotient $\LpG _{X^I}\twoheadrightarrow \LmG_{X^I}$ for any $m\geq m_{N, I}$ has been proven in \cite[Corollary 2.7]{Richarz-New-Satake}. Note that from the proof of \cite[Corollary 2.7]{Richarz-New-Satake} one can see that $m_{N, I}$ depends on $N$ but not on $I$.
\end{proof}

\end{appendices}

\bibliographystyle{alpha}
\bibliography{WM.bib}

\makeatletter
\def\printauthorinfo{%
    \par\vfill
    \hrule\vspace{1em}
    \address{\noindent\emph{Guglielmo Nocera.} Institut des Hautes \'Etudes Scientifiques, 35 Rte de Chartres, 91440 Bures-sur-Yvette, France. \emph{Email:} \texttt{nocera-at-ihes.fr}}\\[\baselineskip]
    \@author
    \address{\noindent \emph{Morena Porzio.} Department of Mathematics, University of Toronto, 40 St. George Street, Toronto, ON M5S 2E4, Canada. \emph{Email:} \texttt{m.porzio-at-utoronto.ca}}\\[\baselineskip]
    \@author
}
\makeatother

\printauthorinfo

\end{document}

%% file: commands.tex
\ifpersonal
\newcommand*{\personalMP}[1]{\textcolor[rgb]{1,0,0}{MP: #1}}
\newcommand*{\personalGN}[1]{\textcolor[rgb]{0,1,0}{GN: #1}}

 \else
 \newcommand*{\personal}[1]{\ignorespaces}
\newcommand*{\personalMP}[1]{\ignorespaces}
\newcommand*{\personalGN}[1]{\ignorespaces}

 \renewcommand*{\todo}[1]{\ignorespaces}
 \fi

\renewcommand{\epsilon}{\varepsilon}

\newcommand{\C}{\mathbb C}
\newcommand{\E}{\mathbb E}
\newcommand{\R}{\mathbb R}

\renewcommand{\cong}{\simeq}

\newcommand{\cc}{\mathcal C}

\newcommand{\cF}{\mathcal F}
\newcommand{\cH}{\mathcal H}
\newcommand{\cG}{\mathcal G}

\newcommand{\cO}{\mathcal O}

\newcommand{\cT}{{\mathcal T}}
\newcommand{\cU}{\mathcal U}

\newcommand{\End}{\textup{End}}

\newcommand{\Aut}{\textup{Aut}}
\newcommand{\Sch}{\textup{Sch}}

\newcommand{\Bun}{\textup{Bun}}
\newcommand{\LG}{\textup{L}G}
\newcommand{\LpG}{\textup{L}^+G}

\newcommand{\PSh}{\textup{PSh}}

\newcommand{\Aff}{\textup{Aff}}

\newcommand{\Top}{{\textup{Top}}}

\newcommand{\Ran}{\textup{Ran}}

\newcommand{\Alg}{\textup{Alg}}

\newcommand{\Fin}{\textup{Fin}}
\newcommand{\Cat}{\mathcal C\textup{at}_\infty}
\newcommand{\ner}{\textup{N}}

\newcommand{\cat}{\mathcal C\textup{at}}
\newcommand{\sets}{\textup{Set}}

\newcommand{\Disk}{\textup{Disk}}
\newcommand{\loc}{\textup{loc}}

\newcommand{\Zar}{{\textup{Zar}}}

\newcommand{\ett}{\textup{\'et}}
\newcommand{\Gr}{\textup{Gr}}

\newcommand{\an}{\textup{an}}

\newcommand{\nun}{\textup{nu}}

\newcommand{\id}{\textup{id}}

\newcommand{\disj}{\textup{disj}}

\newcommand{\surj}{\textup{surj}}

\newcommand{\LmG}{\textup{L}^{m}G}

\newcommand{\Gm}{\mathbb G_\textup{m}}

\newcommand{\triv}{\xrightarrow{\sim}}

\newcommand{\Hck}{{\textup{Hck}}}
\newcommand{\hck}{\cH{\textup{ck}}}

\newcommand{\act}{\textup{act}}
\newcommand{\str}{\textup{str}}

\newcommand{\Str}{\textup{Str}}

\newcommand{\Pos}{\textup{Pos}}

\newcommand{\Grp}{\textup{Grp}}

\newcommand{\Alex}{\textup{Alex}}

\newcommand{\Fun}{\textup{Fun}}

\DeclareMathOperator{\Hom}{Hom}

\DeclareMathOperator{\Map}{Map}

\DeclareMathOperator{\Spec}{\textup{Spec}}

\DeclareMathOperator*{\colim}{colim}

\newcommand{\GL}{{\textup{GL}}}
\newcommand{\xt}{\mathbb X_\bullet(T)^+}

\newcommand{\op}{\textup{op}}
\newcommand{\taylor}{\llbracket t \rrbracket}
\newcommand{\laurent}{(\!(\!t\!)\!)}
\newcommand{\fs}{\Fin_{\geq 1, \surj}}

\newcommand{\Ac}{\mathbb A^1_\C}

\newcommand{\lft}{\textup{lft}}

\newcommand{\N}{\mathbb N}
\newcommand{\nn}{\mathbb N}
\newcommand{\uAut}{\underline{\Aut}}
\newcommand{\uEnd}{\underline{\End}}

\newcommand{\GLn}{\GL_n}

\DeclareFontFamily{U}{min}{}

\DeclareFontShape{U}{min}{m}{n}{<-> udmj30}{}

\newcommand\yo{\!\text{\usefont{U}{min}{m}{n}\symbol{'207}}\!}

\newcommand{\tB}{\textup{B}}

\newcommand{\Fgt}{\textup{Fgt}}

\newcommand{\MDisk}{\textup{MDisk}}

\newcommand{\PSmall}{\PSh^\textup{small}}

\newcommand{\Inc}{\textup{Inc}}